\documentclass[11pt]{article}
\usepackage{amssymb, amsmath,amsthm}
\usepackage{mathrsfs, yfonts, color, calligra}
\usepackage[all]{xy}
\usepackage{enumitem}
\setitemize{itemsep=1pt,topsep=0pt,parsep=0pt,partopsep=0pt}
\usepackage{multirow,float,longtable} 
\usepackage{pgffor,xifthen}
\usepackage{tikz}
\usepackage{changes}
\usepackage{caption}
\usepackage{blkarray}
\usepackage{rotating}
\usepackage{diagbox}
\usepackage{dynkin-diagrams}
\usepackage{makecell}
\usepackage{verbatim}
\usepackage{xxcolor}
\usepackage{hyperref}
\usepackage{array}
\usepackage{booktabs}
\usepackage{subcaption}
\usepackage{stmaryrd}

\usepackage{geometry}
\renewcommand{\thesubfigure}{(\alph{subfigure})}

\hypersetup{
	colorlinks,
	linkcolor=black, 
	citecolor={black},
	urlcolor={black}
}

\makeatletter
\renewcommand{\p@subfigure}{}
\makeatother

\usetikzlibrary{calc}
\tikzset{
dot/.style={draw,circle,inner sep=1.5pt,fill=black},
cir/.style={draw,circle,inner sep=0.8pt},
->-/.style={
    decoration={
      markings,  
      mark=at position 0.6 with {\arrow[thin,scale=2]{>}}
    },
postaction={decorate} 
}, 
-<-/.style={
    decoration={
      markings,
      mark=at position 0.6 with {\arrow[thin,scale=2]{<}}
    },
    postaction={decorate}
  }
}

\setlist{nosep}

\newtheorem{Def}{Definition}[section]
\newtheorem{Prop}[Def]{Proposition}
\newtheorem{Thm}[Def]{Theorem}
\newtheorem{Lem}[Def]{Lemma}
\newtheorem{Cor}[Def]{Corollary}

\theoremstyle{definition} 
\newtheorem{Remark}[Def]{Remark}
\newtheorem{Exam}[Def]{Example}

\setlist[enumerate]{label = {\rm (\arabic*)}}

\DeclareMathOperator{\diag}{diag}

 \DeclareMathOperator{\End}{End}

 \DeclareMathOperator{\Tr}{Tr}

 \DeclareMathOperator{\sgn}{sgn}
 \DeclareMathOperator{\defect}{Def}

\DeclareMathOperator{\size}{\sf size}
\DeclareMathOperator{\ex}{ex}

\DeclareMathOperator{\lso}{\mathfrak{so}}
\DeclareMathOperator{\lsp}{\mathfrak{sp}}
\DeclareMathOperator{\lsl}{\mathfrak{sl}}
\DeclareMathOperator{\lgl}{\mathfrak{gl}}
\DeclareMathOperator{\lieg}{\mathfrak{g}}
\DeclareMathOperator{\liegd}{\dot{\mathfrak{g}}}
\DeclareMathOperator{\lieh}{\mathfrak{h}}
\newcommand{\liehd}{\dot{\mathfrak{h}}}
\DeclareMathOperator{\lghat}{\hat{\mathfrak{g}}}
\DeclareMathOperator{\lgdhat}{\hat{\dot{\mathfrak{g}}}}
\DeclareMathOperator{\lhdhat}{\hat{\dot{\mathfrak{h}}}}
\DeclareMathOperator{\lsohat}{\widehat{\mathfrak{so}}}
\DeclareMathOperator{\lglhat}{\widehat{\mathfrak{gl}}}
\DeclareMathOperator{\lslhat}{\widehat{\mathfrak{sl}}}
\DeclareMathOperator{\lsphat}{\widehat{\mathfrak{sp}}}
\DeclareMathOperator{\lgtilde}{\tilde{\mathfrak{g}}}
\newcommand{\lgdtilde}{\tilde{\dot{\mathfrak{g}}}}
\newcommand{\lhdtilde}{\tilde{\dot{\mathfrak{h}}}}

\DeclareMathOperator{\Span}{Span}
\DeclareMathOperator{\Cl}{{\sf Cl}} 

\newcommand{\defCategory}[2]{
  \newcommand{#1}{#2\defvariable}}

\newcommand{\defvariable}[2][]{
\if\relax\detokenize{#1}\relax  
\if\relax\detokenize{#2}\relax
    \else  ({#2})  \fi
    \else  ^{{\rm #1}}({#2})  \fi}

\defCategory{\C}{\mathscr{C}}
\defCategory{\K}{\mathscr{K}}
\defCategory{\D}{\mathscr{D}}

\foreach \x in {A,...,Z}{%
\expandafter\xdef\csname cal\x\endcsname{\noexpand\ensuremath{\noexpand\mathcal{\x}}}
\expandafter\xdef\csname scr\x\endcsname{\noexpand\ensuremath{\noexpand\mathscr{\x}}}
\expandafter\xdef\csname bb\x\endcsname{\noexpand\ensuremath{\noexpand\mathbb{\x}}}}

\newcommand{\levelR}{\mathcal{R}}
\newcommand{\levelL}{\mathcal{L}}
\newcommand{\W}{{\bf W}} 

\newcommand{\F}{\mathcal{F}}
 
\newcommand{\B}{\mathsf{B}}

\newcommand{\lgo}{\mathring{\lieg}}
\newcommand{\sgmo}{\mathring{\Sigma}}
\newcommand{\dlho}{\mathring{\lieh}{}^*}

\newcommand{\Vir}{{\sf Vir}}

\newcommand{\ii}{\sqrt{-1}}
\def\F{\mathscr{F}}
\def\P{\mathscr{P}}

\def\m{{\sf m}}
\def\r{{\sf r}}
\newcommand{\M}{\mathsf{M}}
\newcommand{\Y}{\mathsf{Y}}
\newcommand{\Z}{\mathsf{Z}}

\newcommand{\dCox}{{\sf h}\!^\vee}
\newcommand{\ddCox}{\dot{\sf h}\!^\vee}
\newcommand{\rkCar}{l}
\newcommand{\lLevel}{R}

\newcommand{\lleft}{}

\newcommand{\bigalg}{\rm w}

\newcommand{\gh}[1]{{#1}^*}
\newcommand{\floor}[1]{\lfloor #1\rfloor}
\newcommand{\nord}[1]{:\!{#1}\!:\;}

\newcommand{\grws}[2]{{#1}_{#2_{\bar{0}}}\oplus {#1}_{#2_{\bar{1}}}}

\newcommand{\lra}{\longrightarrow}

\newcommand{\sigdom}{P_{\Sigma}^+}

\newcommand{\sigdomdel}{\tilde{P}_{\Sigma}^+}

\renewcommand{\leq}{\leqslant}
\renewcommand{\geq}{\geqslant}
\renewcommand{\le}{\leqslant}
\renewcommand{\ge}{\geqslant}
\renewcommand{\min}{\mathrm{min}}
\renewcommand{\max}{\mathrm{max}}

\renewcommand{\succeq}{\trianglerighteqslant} 
\renewcommand{\succ}{\rhd}

\title{Level-rank dualities and moving vectors}
\author{Wei Hu, Feiyue Huang, Yanbo Li, Xiangyu Qi}
\date{}

\begin{document}

\maketitle

\renewcommand{\thefootnote}{\alph{footnote}}
\setcounter{footnote}{-1} 
\renewcommand{\thefootnote}{\alph{footnote}}
\setcounter{footnote}{-1} 

\begin{abstract}
Duality relations between Lie algebras are a significant phenomenon in Lie algebra representation theory, with  level-rank duality as a famous example. 
Level-rank dualities for affine Lie algebras of type $A^{(1)}$ were first discovered by Frenkel in 1982, and later extended to all classical non-twisted affine types by Hasegawa in 1989 through 
elaborate character calculations.
In this paper, for all classical affine Lie algebras, we construct appropriate Fock spaces in a uniform way and establish corresponding combinatorial models (Maya diagrams and abaci), extending Uglov map to all classcial affine types. 
Through the action of the Virasoro algebra, we completely characterize the joint highest weight vectors in the Fock space, thereby obtaining the corresponding level-rank duality theory. 
Our method no longer relies on character calculations. Using this new level-rank duality theory, the defect of the cyclotomic KLR algebra $\mathscr{R}^{\Lambda}_{\beta}$ of classical affine type can be interpreted as the sum of the components of the correponding moving vector.
\end{abstract}
\setcounter{tocdepth}{2}
\tableofcontents

\section{Introduction}

Duality theories provide a powerful framework in representation theory, where two algebraic objects act on a common vector space in a mutually commuting manner. 
Under favorable circumstances, the information of one algebra can be translated into that of the other, thereby converting complicated problems into more tractable ones. 
A prototypical example is the classical Schur--Weyl duality (\cite{Schur1927,Weyl1939}) and its numerous generalizations .

In a celebrated work \cite{Howe}, Howe established a duality theory involving a classical group $G$ and a Lie algebra $\liegd$ acting on a tensor space $V$. 
In this setting, the decomposition of $V$ as a bimodule yields a correspondence between the irreducible representations of $G$ and those of $\liegd$.
For affine Lie algebras, Frenkel \cite{Frenkel1982} realized the level-rank duality of two type A affine Lie algebras on a Fock space. 
Subsequently, through intricate character computations, Hasegawa \cite{Hasegawa1989} extended Frenkel's results to all classical non-twisted affine Kac--Moody algebras.

From a combinatorial perspective, Uglov \cite{U} developed a description of the level-rank duality for affine type $A$, introducing what is now known as the Uglov map, and applied it to the study of Kazhdan--Lusztig theory.
Chuang and Miyachi \cite{ChuangMiyachi2012} conjectured that Uglov's bijections could be lifted to a categorical level, giving rise to Koszul-type equivalences between blocks of highest-weight covers for Ariki--Koike algebras, or equivalently, blocks of categories $\mathcal{O}$ for cyclotomic rational double affine Hecke algebras (DAHAs). 
This categorical level-rank duality was subsequently proved by Rouquier, Shan, Varagnolo, and Vasserot in \cite{Rouquier2016, Shan2014}.

In this paper, we  develop the level-rank duality theory for all classical affine types using a unified and transparent method. 
We construct Fock spaces using $\mathbb{Z}_2$-graded vector spaces, a construction that is uniform across all types and applicable to both twisted and non-twisted affine Lie algebras. This construction admits an intuitive combinatorial characterization via Maya diagrams. More precisely, for the Fock space $\F$, we consider a grid board. 
A Maya diagram is obtained by coloring a finite number of these cells black, and up to a sign, each Maya diagram corresponds to a monomial basis vector in $\F$. For the precise definition of the Fock space and Maya diagrams, see Section 3 and 4. 

After defining the Fock space, we discuss the actions of affine Lie algebras on it. The main idea is as follows: for a given affine pair $(\lghat, \lgdhat)$, we embed their direct sum into a larger affine algebra $\lghat_{\bigalg}$. 
Using the standard construction from \cite{Feingold1985}, we realize the Fock space as a level-1 module over $\lghat_{\bigalg}$. 
Upon restriction, this module carries mutually commuting actions of $\lghat$ and $\lgdhat$. These actions admit a natural interpretation on Maya diagrams (or abaci): the action of $\lghat$ corresponds to horizontal left-right movements of beads, while the action of $\lgdhat$ corresponds to vertical up-down movements.

Our duality pairs of affine Lie algebras are quite general. We divide the types of  classical affine Lie algebras into four classes  
\[\{A^{(1)}\}, \quad \{A^{(2)}\}, \quad \{C^{(1)}\}, \quad O^{(\r)}=\{B^{(1)}, D^{(1)},D^{(2)}\}.\]
The dual pair $(\lghat_{\lleft}, \lgdhat)$ can be chosen to be any two affine algebras from the same class. For each of these pairs, we  uniformly construct the corresponding Fock space $\F$ as a subspace of a suitable wedge space constructed from a $\mathbb{Z}_2$-graded vector space, and establish the following level-rank duality theorem.

\begin{Thm}\label{mainDualThmIntro}
Let $(\lghat_{\lleft},\lgdhat)$ be a pair of affine Lie algebras given in Table \ref{tbl-intro}, and let $\F$ be the associated Fock space  defined in Definition \ref{def-Fock-space}. Then, as a $\lghat_{\lleft} \oplus \lgdhat$-module, the level-$1$ $\lghat_{\bigalg}$ module $\F$ decomposes as a direct sum of irreducible $(\Sigma_{\lleft} \ltimes \lghat_{\lleft}, \dot{\Sigma} \ltimes \lgdhat)$-bimodules: 
    \[
    \F = \bigoplus_{Y} V_{\Sigma_{\lleft} \ltimes \lghat_{\lleft}}(\lambda_{\Y}; \levelR) \otimes V_{\dot{\Sigma} \ltimes \lgdhat}(\dot{\lambda}_{\Y}; \levelL),
    \]
where $\Sigma$ and $\dot{\Sigma}$ are automorphism groups of $\lghat_{\lleft}$ and $\lgdhat$ respectively  induced by diagram automorphisms,  $V_{\Sigma\ltimes\lghat}(\lambda_{\Y}; \levelR)$ is the irreducible $\Sigma\ltimes\lghat$-module with highest weight $\lambda_{\Y}$ and level $\levelR$, and $V_{\dot{\Sigma}\ltimes\lgdhat}(\dot{\lambda}_{\Y}; \levelL)$ is the irreducible $\dot{\Sigma}\ltimes\lgdhat$-module with highest weight $\dot{\lambda}_{\Y}$ and level $\levelL$.  Depending on the type of the duality pair, $\Y$ ranges over specific Young diagrams, and  $\lambda_{\Y}$ and $\dot{\lambda}_{\Y}$ are the  weights corresponding to the Young diagram $\Y$.
\end{Thm}

\begin{table}[!htbp]
\centering
\scalebox{0.9}{
\renewcommand{\arraystretch}{1.4}
\begin{tabular}{c|c|c|c|c}
\hline
Type                                    &   $A^{(1)},A^{(1)}$            &    $C^{(1)}, C^{(1)}$       &    $O^{(\r)},O^{(\r)}$                                                     &    $A^{(2)},A^{(2)}$\\ 
\hline
parameters 
  &  \makecell{$L_{\bar{0}}=2l, R_{\bar{0}}=2r$,\\ $L_{\bar{1}}=R_{\bar{1}}=0$} & \makecell{$L_{\bar{0}}=2l, R_{\bar{0}}=2r$,\\ $L_{\bar{1}}=R_{\bar{1}}=0$} & \makecell{$L_{\bar{0}}=2l$ or $2l+1$,\\
 $R_{\bar{0}}=2r$ or $2r+1$,\\
 $L_{\bar{1}}, R_{\bar{1}}\le 1$
 } & \makecell{$L_{\bar{0}}=2l$ or $2l+1$,\\
 $R_{\bar{0}}=2r$ or $2r+1$,\\
 $L_{\bar{1}}=R_{\bar{1}}= 0$}
  \\
\hline 
$\lghat_{\lleft}, \lgdhat$     & $\lglhat(l), \lslhat(r)$      & $\lsphat(2l), \lsphat(2r)$& $\lsohat(L_{\bar{0}}, L_{\bar{1}}), \lsohat(R_{\bar{0}}, R_{\bar{1}})$    &    $\lglhat{}^{(2)}(L), \lslhat{}^{(2)}(R)$\\
\hline
$\lghat_{\bigalg}$                                &  $\lsohat(2lr)$               & $\lsohat(4lr)$            & $\lsohat(L_{\bar{0}}R_{\bar{0}}+L_{\bar{1}}R_{\bar{1}}, L_{\bar{0}}R_{\bar{1}}+L_{\bar{1}}R_{\bar{0}})$   &   $\lglhat{}^{(2)}(LR)$ \\
\hline
level $\levelR, \levelL$                            &  $\levelR=r,\levelL=l$                        & \multicolumn{3}{c}{$\levelR=R_{\bar{0}}+R_{\bar{1}}, \levelL=L_{\bar{0}}+L_{\bar{1}}$} \\
\hline
range of $Y$                            &  $\P^l_{r-1}$                 & \multicolumn{3}{c}{$\P^l_r$ }   \\
\hline
$\lambda_{\Y}, \dot{\lambda}_{\Y}$   & see (\ref{weightY-A})    & 
\multicolumn{3}{c}{see (\ref{weightY})}  \\
\hline
\end{tabular}
}
\caption{Data of duality types}
\label{tbl-intro}
\end{table}
The proof of the level-rank duality theorem presented in this paper differs significantly from existing approaches. 
In Hasegawa's work on non-twisted affine algebras \cite{Hasegawa1989}, the result is obtained by directly verifying character identities, which involves many specialized techniques for symmetric polynomials and $\Theta$-functions. 
In contrast, this paper adopts an approach focused on exploring the intrinsic structure and actions of the modules themselves. 
Due to the complete reducibility of the Fock space, it suffices to identify all joint highest weight vectors within it. 
This philosophy aligns with the classical Howe duality theories for finite classical groups \cite{Cheng2012, Howe, Wang1999}. 

However, those classical theories indispensably rely on Jacobson's density theorem, which requires the Fock space, as a module over one of the algebras in the dual pair, to decompose as a direct sum of finite-dimensional irreducible modules. 
This condition is unattainable in the affine case, meaning we cannot directly generalize the classical approach. 
Instead, our method overcomes this obstacle by combining finite-dimensional duality theory with the action of the Virasoro algebra on the Fock space to derive the affine duality.

The combinatorial model provided by Maya diagrams and abaci is a cornerstone of our unified approach and its applications. It serves several crucial roles. Firstly, the Fock space actions of the affine Lie algebras are translated into intuitive graphical operations, making the otherwise complex algebraic manipulations transparent and manageable. Secondly, the one-to-one correspondence between Maya diagrams and monomial basis vectors allows us to identify all joint highest weight vectors and track the change of weights through simple moves of beads. Finally, and perhaps most importantly, this model provides a bridge to the representation theory of cyclotomic KLR algebras $\mathscr{R}^{\Lambda}_{\beta}$. By interpreting the combinatorial data of an abacus—such as the moving vector and $e$-weights studied in \cite{Fayers, JL, LQ}—through the lens of the $\lgdhat$-action, we gain new Lie-theoretic insights. Specifically, we show that the sum of the components of the moving vector for $\mathscr{R}^{\Lambda}_{\beta}$ is precisely the defect of $(\Lambda, \beta)$ (see Theorem \ref{thm-def-weight}). This connection not only generalizes known results but also suggests that this combinatorial model will be a powerful tool for future investigations into block invariants of cyclotomic KLR algebras. 

Furthermore, the combinatorial model in our duality theory has been used recently by the third author and his student to study cores and Diophantine equations in \cite{LZZ}. By extending the concept core to arbitrary charge for all classical affine types, they establish and study Diophantine equations of classical affine types via the height formula given by the so-called Uglov vector, and solve a generalized version of an open problem raised by Brunat, Chapelier-Laget and Gerber in \cite{BCG}. In particular, closed formulae for computing the number of certain kinds of core abaci are given.

This paper is organized as follows. In Section 2, we recall some basic facts about affine Lie algebras. In Section 3, we construct the Fock space and define the actions of affine Lie algebras on it. In Section 4, we will give, for all classical types, a uniform graphical
configuration, and interpret
the action of Lie algebras on this model.
Section 5 is devoted to giving a proof of our  level-rank duality theorem. In Section 6, we prove that the defect of a cyclotomic KLR algebra is equal to the sum of the components of its moving vector. We give some concluding comments in Section 7.
 
\section{Preliminaries for affine Lie Algebras}

In this section, we fix notations and give some basic facts about affine Lie algebras. The main references are \cite{Hasegawa1989} and \cite{Kac}.

\subsection{Affine Kac-Moody Lie algebras}\label{prelim-KM}
We use terminologies of \cite{Kac}. Let $A$ be an affine Cartan matrix of size ${\rkCar}+1$ of type $X^{(\r)}_N$. Let $\lieg$ be a finite dimensional reductive Lie algebra such that $[\lieg, \lieg]$ is simple of type $X_N$ with $(-|-)'$ the normalized invariant form.
Denote by $E_i, F_i, H_i, i=0, 1, \dots, {\rkCar}$ the elements given in \cite[\S 8.3]{Kac}(Concrete realization will be listed in \S \ref{notation-classic-lie} below).
 Fix an automorphism $\sigma$ of $\lieg $ of order $ \m $, with $\r$ the least positive integer such that $\sigma^\r$ is inner.
Then we have a $\mathbb{Z}/ \m \mathbb{Z}$-grading on $\lieg$:
\[
\lieg = \bigoplus_{\overline{j}\in \mathbb{Z}/ \m \mathbb{Z}} \lieg_{\overline{j}},\quad  \lieg_{\overline{j}} = \{x : \sigma(x) = e^{\frac{2\pi\ii}{ \m }{j}}x\}.
\]
Recording the degree of each $E_i\in \lieg_{\overline{s_i}}$ gives a sequence $(s_0, \dots, s_{\rkCar})$ of non-negative integers such that $ \m = \r\sum_{i=0}^{\rkCar} a_i s_i$, where $a_i$ are the indices on the Dynkin diagram. The derived affine Lie algebra of type $X^{(\r)}_N$ can be realized by
\[
\lghat :=  \big(\bigoplus_{j \in \mathbb{Z}} \lieg_{\overline{j}}\otimes t^j\big)\oplus\mathbb{C}c,
\]
\[
[x_1(n_1) + \eta_1 c, x_2(n_2) + \eta_2 c] = [x_1 x_2](n_1 + n_2) + \frac{1}{ \m } (x_1|x_2)' n_1 \delta_{n_1, -n_2} c, 
\]
where $x(n)$ denotes $x\otimes t^n$.
Extending $\lghat$ by a derivation $d'$ with $d'(x(n))=\frac{n}{ \m }x(n)$ and $[d',c]=0$, we obtain $\lgtilde = \lghat \oplus \mathbb{C}d'$ the affine Kac-Moody algebra 
with scaling element $d = a_0 \r (d' - H )$, where $H\in \sum\limits_{i=1}^{\rkCar} \mathbb{C} H_i$ defined by the conditions $\alpha_i(H) = \frac{s_i}{ \m }$ for each $i=1,2,\dots,l$.
The normalized symmetric invariant form $(-|-)$ on $\lgtilde$ is defined by
\[
(x_1(n_1) | x_2(n_2)) = \frac{1}{\r} \delta_{n_1,-n_2} (x_1 | x_2)', \quad (\mathbb{C}c + \mathbb{C}d' | x(n)) = 0, \quad (c|c)= 0 =(d'|d'),\quad  (c|d')=1.
\]
The bilinear form is carried over to $\hat{\lieh}$ and $\hat{\lieh}^*$, and will be again denoted by $(-|-)$.

Denote by $\Vir := \bigoplus_{n\in \mathbb{Z}} \mathbb{C} d(n) \oplus \mathbb{C} z$ the Virasoro algebra, where the Lie bracket is given by
\[
[d(n_1), d(n_2)] = (n_1 - n_2) d(n_1 + n_2) + \frac{n_1^3-n_1}{12}\delta_{n_1, -n_2} z, \quad [z, \Vir] = 0.
\]
As derivations, $\Vir$ acts on $\lghat$ by
\[
[d(n), x(n')] = -\frac{n'}{ \m } x( \m n + n'), \quad [z, \lghat] = 0 = [\Vir, c].
\]
With this action, $\lgtilde$ can be viewed as a subalgebra of $\Vir\ltimes \lghat$ by identifying $d'$ with $-d(0)$.

It is well-known \cite[\S 8.3]{Kac} that $\lgtilde$ is a Kac-Moody Lie algebra associated to the generalized Cartan matrix $A$. Let $a_i, a_i^\vee, i=0,1,\dots,\rkCar$ be indices associated to $A$ and $A^T$, respectively. Then we have the relations
\[
A\left(\begin{matrix}
    a_0 \\ \vdots \\ a_{\rkCar}
\end{matrix}\right) = 0, \quad 
(a_0^\vee, \dots, a_{\rkCar}^\vee) A = 0,
\]
with $c = \sum\limits_{i=0}^{\rkCar} a_i^\vee \alpha_i^\vee$ central, $\delta = \sum\limits_{i=0}^{\rkCar} a_i \alpha_i$ the  imaginary root, where $\alpha_i$ are the simple roots and $\alpha_i^\vee$ the simple coroots. 
We write $\theta = \sum\limits_{i=1}^{\rkCar} a_i \alpha_i = \delta - a_0 \alpha_0$ and write $\dCox = \sum\limits_{i=0}^{\rkCar} a_i^\vee$ for the dual Coxeter number. The Chevalley generators of $\lghat$ are given by 
\[
e_i = E_i(s_i),\quad f_i = F_i(s_i),\quad h_i = H_i(0) + \frac{a_is_i}{a_i^\vee\m} c, \quad i=0,\dots, \rkCar.
\] 
Let $\lgo$ be the finite dimensional simple Lie subalgebra of $\lghat$ generated by $e_i, f_i, h_i, i=1,\dots, \rkCar$. Restricting $(-|-)$ from $\lgtilde$ (not $\lieg$) to $\lgo$, we get a non-degenerate invariant form again denoted by $(-|-)$.
Write $\tilde{\lieh}^* = \dlho \oplus \mathbb{C} \Lambda_0 \oplus \mathbb{C} \delta$, we have the orthogonal projection $\overline{\cdot }: \tilde{\lieh}^* \to \dlho$.
Denote by $\varpi_i$ the $i$-th fundamental weight of $\lgo$, then the $i$-th fundamental weight of $\lgtilde$ is given by $\Lambda_i = a_i^\vee \Lambda_0 + \varpi_i$. 
We put  $\rho = \sum\limits_{i=1}^{\rkCar} \varpi_i = \dfrac{1}{2}\sum\limits_{\alpha\in \mathring{\Delta}{}^+} \alpha \in \dlho$, and $\hat{\rho} = \rho + \dCox \Lambda_0 = \sum\limits_{i=0}^{\rkCar} \Lambda_i \in \hat{\lieh}^*$. 
Note that $(\rho | \theta) = \dCox - a_0^\vee$, $(\rho | \alpha_i) = \frac{a_i^\vee}{a_i}$.

An irreducible highest weight module of $\lghat$ is characterized by $(\lambda; {\lLevel})\in \dlho \times \mathbb{C}$ and is denoted by $V_{\lghat}(\lambda;{\lLevel})$.
We call $\lambda$ the classical part and ${\lLevel}$ the level of the weight, and sometimes we also write $\lambda + {\lLevel}\Lambda_0 \in \hat{\lieh}^*$.
Write 
$$
\tilde{P}^+ ({\lLevel}) = \left\{\sum\limits_{i=0}^{\rkCar} c_i \Lambda_i + c' \delta \mid \sum\limits_{i=0}^{\rkCar} c_i a_i^\vee = {\lLevel}, c'\in \mathbb{C}\right\}
$$ 
the set of dominant integral weights of level ${\lLevel}$.
Note that for $\lambda + {\lLevel} \Lambda_0 \in \tilde{P}^+ ({\lLevel})$, we have $(\lambda|\theta) \le {\lLevel}$.
Write $P^+({\lLevel}) = \{ \lambda \in \dlho | \lambda + {\lLevel}\Lambda_0 \in \tilde{P}^+ ({\lLevel}) \}$ the set of classical parts of $\tilde{P}^+ ({\lLevel})$.

There is another important finite dimensional simple Lie subalgebra $\lgo'$ of $\lghat$ generated by $e_i, f_i, h_i$, $i=0, 1, \dots, \rkCar-1$. Since $\lghat$ is generated by $\lgo$ and $\lgo'$, we can locally combine the actions of two simple Lie algebras to analyse $\lghat$-modules.

For an abelian Lie algebra $\lieg$, we use the conventions that $\lieh = \lieg$, $\rho = 0$, $\dCox = 0$.

\subsection{Segal-Sugawara construction}\label{Segal-Sugawara}
Assume that $\lieg$ is simple or abelian,  and that  $V$ is  a $\lghat$-module satisfying the following two conditions.
\begin{itemize}
    \item $V$ is restricted, that is, for any $v\in V$, we have $\lieg(n).v = 0$ when $n$ is sufficiently large;
    \item $c$ acts as ${\lLevel}\cdot id_V$ for some ${\lLevel}\in \mathbb{C}$ such that ${\lLevel} + \dCox \ne 0$. 
\end{itemize} 
The scalar ${\lLevel}$ is called the \textit{level} of $V$. Recall that the grading $\lghat = \big(\bigoplus_{j \in \mathbb{Z}} \lieg_{\overline{j}}\otimes t^j\big)\oplus\mathbb{C}c$ is isotropic with respect to $(-|-)$ in the sense that 
$(\lieg_{\overline{i}}(i) | \lieg_{\overline{j}}(j)) = 0$ unless $i = -j$ and $(-|-)\big|_{\lieg_{\overline{j}}(j) \times \lieg_{-\overline{j}}(-j)}$ non-degenerate. Choose a basis $u_{i;-j}(-j)$ of $\lieg_{-\overline{j}}(-j)$, and a dual basis $u^{i;j}(j)$ of $\lieg_{\overline{j}}(j)$ with respect to $(-|-)$.
Then the following Segal-Sugawara operators $D^{\lghat}(n) \in \End (V)$ give a $\Vir \ltimes \lghat$-module structure on $V$ (cf. \cite[Exercise \S 12.20]{Kac}):
\begin{align*}
    d(0) \mapsto & D^{\lghat}(0) := \frac{1}{2 \r ({\lLevel}+\dCox)} \sum_{i} (u_{i;0}(0) u^{i;0}(0) + 2\sum_{n>0} u_{i;-\overline{n}}(-n) u^{i;\overline{n}}(n)) \\
    & \hspace{13mm} - H + \r {\lLevel} \frac{(H|H)}{2} + \left(\frac{{\rm dim} \lieg}{24} - \frac{|\rho|^2}{2 \dCox \r}\right) \frac{{\lLevel}}{{\lLevel} + \dCox},     \\
    d(j) \mapsto & D^{\lghat}(j) = \frac{1}{j} \ [d(j), D^{\lghat}(0)], \\
    z \mapsto & z^{\lghat} = \frac{{\lLevel}\cdot \dim \lieg}{{\lLevel}+\dCox}. \tag{2.2.1}\label{eq-central-charge}
\end{align*}
In particular, if $V = V_{\lghat}(\lambda, {\lLevel})$ with highest weight vector $v$ and ${\lLevel} + \dCox \ne 0$, we can fix a $\lgtilde$-module structure canonically by putting $D^{\lghat}(0).v = {\sf d}^{\lghat}_{\lambda; {\lLevel}} v$, where ${\sf d}^{\lghat}_{\lambda; {\lLevel}}$ is the \textit{vacuum anomaly}:
\[
{\sf d}^{\lghat}_{\lambda;{\lLevel}} = \frac{(\lambda \mid \lambda + 2\rho)}{2\r ({\lLevel} + \dCox)} - \lambda(H) + \r {\lLevel} \frac{(H|H)}{2} + \left(\frac{{\rm dim} \lieg}{24} - \frac{|\rho|^2}{2 \dCox \r}\right) \frac{{\lLevel}}{{\lLevel} + \dCox}.
\tag{2.2.2}\label{Formula-vacuum}\]
This allows us to identify $V_{\lghat}(\lambda, {\lLevel})$ with the irreducible highest weight $\lgtilde$-module $V_{\lgtilde}(\lambda + {\lLevel}\Lambda_0 - {\sf d}^{\lghat}_{\lambda;{\lLevel}}\delta)$. For each $\Lambda\in P^+$, the set of weights of $V_{\lgtilde}(\Lambda)$ is denoted by $P(\Lambda)$. 

For a finite dimensional reductive Lie algebra $\lieg = \oplus \lieg_i$  with each $\lieg_i$ simple or abelian, we can generalize the above constructions by defining
\[
D^{\lghat}(n) = \sum D^{\lghat_i}(n), \quad z^{\lghat} = \sum z^{\lghat_i}, \quad {\sf d}^{\lghat}_{\lambda;{\lLevel}}= \sum {\sf d}^{\lghat_i}_{\lambda_i;R_i}.
\]

We briefly state the construction of coset $\Vir$-module (cf.\cite[\S 1.2.4]{Hasegawa1989}). 
Let $V$ be a restricted $\lghat$-module of positive level. The Segal-Sugawara operators give a representation $\pi: \Vir\ltimes \lghat \to \End(V)$.
Suppose there is another (unitary) representation $\pi':\Vir \to \End (V)$ satisfying
\[
[\pi'(d(n_1)), \pi(x(n_2))] = -\frac{1}{ \m } n_2 \pi(x(n_1 + n_2)),
\]
or equivalently, $[\pi'(d(n)) - D^{\lghat}(n), \lghat] = 0$, then $D^{\perp}(n):= \pi'(d(n)) - D^{\lghat}(n)$ and $z^{\perp}:= \pi'(z) - \pi(z)$ define a (unitary) representation of $\Vir$ on ${\rm Hom}_{\lghat}(V_{\lghat}({\lambda};{\lLevel}), V)$ for any $V_{\lghat}({\lambda};{\lLevel})$.
Moreover, if in addition $\pi(z) = \pi'(z)$, then we have $\pi'(d(n)) = D^{\lghat}(n)$ for all $n$, that is, $\pi\big|_{\Vir} = \pi'$. 

\subsection{Diagram automorphism} \label{DiagAuto}
We follow the notations in \cite[\S 1.3]{Hasegawa1989}. For  a generalized Cartan matrix $A$, let $\sigma$ be an automorphism of the Dynkin diagram of $A$. Then it induces automorphisms on both $\lieh(A)^*$ as well as $\lieg(A)$, which are called diagram automorphisms and again denoted by $\sigma$. 
Since the action of $\sigma$ on $\lieh(A)^*$ is given by permuting the simple roots, it also permutes the fundamental weights, and preserves the invariant form $(-|-)$ on $\lieh(A)^*$, as well as the set of positive roots $Q^+$. In particular, the partial order on $\lieh(A)^*$, defined by $\lambda \ge \mu$ if  $\lambda - \mu \in Q^+$, is preserved by $\sigma$.

Let $\Sigma \leq  {\rm Aut} \lieg(A)$ be an (Abelian) subgroup consisting of certain  diagram automorphisms. 
We say a $\lieg(A)$-module $V$ with $\pi:\lieg(A) \to \End (V)$ is a $\Sigma \ltimes \lieg(A)$-module if there exists a representation $\pi':\Sigma\to PGL(V)$ of $\Sigma$ on $V$ which satisfies 
$\pi'(\sigma) \pi(g) \pi'(\sigma)^{-1} = \pi( \sigma (g) )$.
Through this construction we see that for a $\Sigma \ltimes \lieg(A)$-module $V$, the action of any $\sigma \in \Sigma$ gives an isomorphism between weight spaces, that is, $V_\lambda \cong V_{\sigma(\lambda)}$ for any $\lambda \in \lieh(A)^*$. For a weight $\lambda$, we define a $\lieg(A)$-module
\[
V_{\Sigma \ltimes \lieg(A)} (\lambda):= \bigoplus_{\mu \in \Sigma. \lambda} V_{\lieg(A)} (\mu).
\]
This module gives rise to an irreducible $\Sigma \ltimes \lieg(A)$-module by letting $\Sigma$ permutes the highest weight vectors.

\begin{Remark}
Note that if $A$ is of type $D_{\rkCar}$ and $\Sigma = \langle \sigma_{\rkCar-1,\rkCar} \rangle$ permuting vertices $\rkCar-1$ and $\rkCar$ on Dynkin diagram, 
then the module $V_{\Sigma \ltimes \lieg(A)} (\lambda)$ for a dominant integral weight is nothing but a finite dimensional irreducible ${Pin}(2\rkCar)$-module,
while $V_{\lieg(A)} (\lambda)$ is an irreducible ${Spin}(2\rkCar)$-module. 
In other words, for any $\lso(2\rkCar)$-module $V$ which can be lifted to ${Pin}(2\rkCar)$, the action of certain reflection in ${Pin}(2\rkCar)$ on $V$ plays the role of $\sigma_{\rkCar-1, \rkCar}$ making $V$ a $\Sigma \ltimes \lso(2\rkCar)$-module (see also Remark \ref{rem-pin} below).
\end{Remark}

\subsection{Notations for classical Lie algebras} \label{notation-classic-lie}
We list some  conventions needed to construct classical Lie algebras, see also Appendix \ref{appendix-datum-types} for details.

\begin{itemize}
\item For $n\in \mathbb{Z}_{\ge 0}$, we introduce the following index sets with linear order:
\begin{align*}
    {\bf I}_{2n}:=& \{ 1 \lessdot \cdots \lessdot n \lessdot -n \lessdot \cdots \lessdot -1 \}, \quad {\bf I}_0:=\varnothing; \\
    {\bf I}_{2n+1}:=& \{ 1 \lessdot \cdots \lessdot n \lessdot 0 \lessdot -n \lessdot \cdots \lessdot -1 \}.
\end{align*}

\item For each $M\in\mathbb{Z}_{\ge 0}$, the $M$-dimensional $\mathbb{C}$-vector space is denoted by $\W_M$ with basis $\{\psi^i\mid i\in {\bf I}_M\}$. Let $I_M$ be the $M\times M$ identity matrix, and let $J_M=(\delta_{i,-j})_{i,j\in {\bf I}_M}$. The matrix $E_{ij}$ has $(i,j)$-element $1$ and zero otherwise.

\item For a $\mathbb{Z}_2$-graded vector space $\W = \W_M\oplus \W_N$ with $\W_M$ in degree $\bar{0}$ and $\W_N$ in degree $\bar{1}$, we use the index set 
${\bf I}_{M,N}={\bf I}_M\sqcup \gh{{\bf I}}_N$ to label the basis of $\W$ as well as rows and columns of $\lgl(M + N)$, where $\gh{{\bf I}}_N=\{\gh{i}\mid i\in {\bf I}_N\}$.
\end{itemize}

\subsubsection{Notations for orthogonal series}
We list some conventions of orthogonal algebras.
Fix $M\in \mathbb{Z}_{>0}$ and $N\in \mathbb{Z}_{\ge 0}$. Let $m = \floor{\frac{M}{2}}$ and $n = \floor{\frac{N}{2}}$. The Lie algebra $\lgl(M+N)$ naturally acts on the $\mathbb{Z}_2$-graded space $\W =\W_M\oplus \W_N$, where $\W_M$ is in degree $\bar{0}$ and $\W_N$ is in degree $\bar{1}$.  The orthogonal Lie algebra 
$$\lieg = \lso(M, N):= \{x\in \lgl(M + N)\mid x^T J + J x = 0\},$$ where 
$J = \left(\begin{matrix}
    J_{M} &   \\ & J_{N}
\end{matrix}\right)$.
When $N = 0$ we abbreviate to $\lso(M)$. The matrices $D_{ij}: = E_{ij} - E_{-j, -i}$, $i,j\in {\bf I}_{M, N}$ span  $\lieg$ as a vector space. Note that $D_{ij} = - D_{-j, -i}$. 

Let $\mathring{\lieh}$ be the span of its basis $\{ D_{ii} \mid i = 1, 2,\dots, m, \gh{1}, \gh{2},\dots, \gh{n}\}$
with dual basis in $\dlho$: $\{\epsilon_i \mid i = 1, 2, \dots, m, \gh{1}, \dots, \gh{n}\}$. 
Defining
\[(x|y)' = \frac{1}{2} \Tr(xy), x,y\in\lieg \mbox{ and } (\epsilon_i | \epsilon_j)' = \delta_{i, j}\] 
gives rise to invariant symmetric non-degenerate bilinear forms on $\lieg$ and $\dlho$, respectively. 
{
The root vectors we need are as follows (Conventions: $\gh{i}+1 = \gh{(i+1)}$, $-\gh{i} = \gh{(-i)}$). 
\[
    E_i = D_{i, i+1}, i = 1, \dots, m-1,\gh{1}, \dots, \gh{(n-1)}, \quad  E_m = 
    \begin{cases}
    D_{m-1, -m},&  M=2m, \\
         \sqrt{2} D_{m, 0}, &  M=2m+1.
    \end{cases}
\]
\[ 
    E_{\gh{n}} = 
    \begin{cases}
    D_{\gh{(n-1)}, \gh{-n}},& \  N=2n, \\
    D_{-1, \gh{-1}},& \ N=2, \\
    \sqrt{2} D_{\gh{n}, \gh{0}},& \ N=2n+1>1,\\
    \sqrt{2} D_{-1, \gh{0}},& \ N=1.
    \end{cases} \quad
    E_{0} = 
    \begin{cases}
    D_{-1, \gh{1}},& \  N>1, \\
        \sqrt{2} D_{-1, \gh{0}},& \ N=1,\\
        D_{-1, 2},& \ N=0.
    \end{cases}
\]
and $F_i=E_i^T$ is the transpose of $E_i$ for each $i$. 
}

Let $\sigma$ be the automorphism of $\lso(M,N)$ given by the conjugation of $\left(\begin{matrix} I_M & \\ & -I_N\end{matrix}\right)$. 
Then $\sigma^2 = {\rm id}$, so $\m = 2$ and there is a decomposition $\lieg = \lieg_{\bar{0}} \oplus \lieg_{\bar{1}}$ with $\lieg_{\bar{0}} \cong \lso(M) \oplus \lso(N)$ and  $\lieg_{\bar{1}} \cong \mathbb{C}^{MN}$.  
This is compatible with the $\mathbb{Z}_2$-grading on $\W$. 
One can check that if $MN$ even then $\sigma$ is inner, and $\r = 1$; while $MN$ odd we have $\r = 2$.
Since $\sigma(E_0)=-E_0$ and $\sigma$ preserves other $E_i$, the degree of $E_i$ is $s_i = \delta_{i, 0}$ for each $i$, and we take $H = \dfrac{1}{2}\sum\limits_{i =1}^{n} D_{\gh{i}\gh{i}}$.
Then we can construct the affine algebra 
$$
\lghat= \lsohat (M, N) = (\lieg_{\bar{0}} \otimes \mathbb{C}[t^{\pm 1}]) \oplus (\lieg_{\bar{1}} \otimes t^{\frac{1}{2}} \mathbb{C}[t^{\pm 1}]) \oplus \mathbb{C}c.
$$ 
Comparing with the general construction in \ref{prelim-KM}, note that here we use the Laurent polynomial ring $\mathbb{C}[t^{\pm \frac{1}{2}}]$. 
The action of the derivation $d'$ can be uniformly given by $d'(x(s)) = sx(s)$ and  $[d', c] = 0$.

\begin{Remark}\label{ortho-reduction}
The Cartan matrix of $\lghat$ is of rank $l = m + n$ and we relabel the vertices $0, 1, \cdots, l$ on Dynkin diagram by $\gh{n}, \gh{(n-1)}, \cdots, \gh{1}, 0, 1, \cdots, m$. 
By deleting the vertex $\gh{n}$ or $m$ we get two simple Lie subalgebras isomorphic to $\lso(M + 2n)$ or $\lso(2m + N)$, respectively.
Through this observation, the type of $\lghat$ is determined by the parity of $M$, $N$. In fact, one can check that $\lsohat(M+2, N) \cong \lsohat(M, N+2)$. 
Therefore, in practice we shall reduce our discussion to three cases: 
$\lsohat(2l)$ for type $D_{l}^{(1)}$, $\lsohat(2l+1)$ for type $B_l^{(1)}$ and $\lsohat(2l+1,1)$ for type $D_{l+1}^{(2)}$.  We uniformly call the affine algebras $\lsohat(M, N)$ as orthogonal affine algebras of type $O^{(\r)}$.
\end{Remark}

\subsubsection{Notations for general linear series}
We list some conventions of twisted affine general linear algebras. 
Fix $M\in \mathbb{Z}_{>0}$, and let $m = \floor{\frac{M}{2}}$.
The Lie algebra $\lieg = \lgl(M)$ acts naturally on the vector space $\W_{M}$ (which can be regarded as a $\mathbb{Z}_2$-graded space concentrating on degree $\bar{0}$). The Cartan subalgebra $\lieh$ has basis $\{ E_{ii} \mid i \in {\bf I_M}\}$
with dual basis in $\lieh^*$: $\{\epsilon_i \mid i \in {\bf I}_M\}$. Defining
\[(x|y)' = \frac{1}{2} \Tr(xy), x,y\in\lieg \mbox{ and } (\epsilon_i | \epsilon_j)' = \delta_{i, j}\] 
gives rise to invariant symmetric non-degenerate bilinear forms on $\lieg$ and $\lieh^*$, respectively. 

Let $\sigma$ be the outer automorphism of $\lgl(M)$ given by $\sigma(x) = - J_M^{-1} x^{T} J_M$. Then $\sigma^2 = {\rm id}$, $\m = 2$ and $\r = 2$. There is a decomposition $\lieg = \lieg_{\bar{0}} \oplus \lieg_{\bar{1}}$, where $\lieg_{\bar{0}}$ is spanned by $E_{ij} - E_{-j, -i}, i,j\in {\bf I}_M$ and $\lieg_{\bar{1}}$ is spanned by $E_{ij} + E_{-j, -i}, i,j\in {\bf I}_M$. Then we can construct the affine algebra of type $A^{(2)}_{M-1}$:
$$\lghat:= \lglhat{}^{(2)} (M) = (\lieg_{\bar{0}} \otimes \mathbb{C}[t^{\pm 1}]) \oplus (\lieg_{\bar{1}} \otimes t \mathbb{C}[t^{\pm 1}]) \oplus \mathbb{C}c.$$
Let $A_{ij}(n) \coloneq \begin{cases}
E_{ij}(n) - (-1)^n E_{-j, -i}(n),& i > 0;\\
(-1)^n E_{ij}(n) - E_{-j, -i}(n),& i \le 0.
\end{cases}$
Note that $A_{ij}(0) = D_{ij}(0)$. The root vectors are: 
\[e_i = A_{i, i+1}(0), i=1,2,\ldots,m-1, \;e_0= A_{-1, 1}(1), \; e_m = 
    \begin{cases}
    A_{m-1, -m}(0),&  M=2m, \\
         \sqrt{2} A_{m, 0}(0), &  M=2m+1.
    \end{cases} \]
\[f_i = A_{i+1, i}(0), i=1,2,\ldots,m-1,\; f_0=A_{1, -1}(-1),\; f_m = 
    \begin{cases}
    A_{-m, m-1}(0),&  M=2m, \\
         \sqrt{2} A_{0, m}(0), &  M=2m+1.
    \end{cases}\]
The simple Lie subalgebra  
$\lgo = \langle e_i, f_i, h_i\mid i=1,2, \ldots, m \rangle$ is isomorphic to $\lso(M)$. When $M=2m$ is even,  $\lgo$ has an automorphism $\sigma_{m, m-1}$  induced by the diagram automorphism permuting the last two vertices.

Since $\sigma(\lsl(M)) = \lsl(M)$, we can similarly define $\lslhat{}^{(2)}(M) \subset \lglhat{}^{(2)}(M)$. 
Note that the identity matrix lies in $\lieg_{\bar{1}}$, so these two algebras share the same $\hat{\lieh}$ and root data.

Consider the complement $\hat{\mathfrak{I}}$ of $\lslhat{}^{(2)}(M)$ in $\lglhat{}^{(2)}(M)$, which is given by $\hat{\mathfrak{I}}  \coloneq  I \otimes t\mathbb{C}[t^\pm] \oplus \mathbb{C}c$.
Note that $\hat{\mathfrak{I}}$ can be regarded as the affination of the abelian Lie algebra $\mathfrak{I} = \mathbb{C}I$ concentrated in degree $\bar{1}$, with $\dim \mathfrak{I} = 1$ and the dual Coxeter number $\dCox(\mathfrak{I}) = 0$. Therefore, via Segal-Sugawara construction, every $\Vir$-module obtained from $\hat{\mathfrak{I}}$ satisfies $z^{\hat{\mathfrak{I}}} = 1$.
By the direct sum decomposition $\lgl(M) = \lsl(M) \oplus \mathfrak{I}$, for a $\Vir$-module induced from $\lglhat{}^{(2)}(M)$ we have $z^{\lglhat{}^{(2)}(M)} = z^{\lslhat{}^{(2)}(M)} + 1$.

\begin{Remark}\label{reverse-A}
    Different from the usual conventions, under this construction the algebra $\lslhat{}^{(2)} (2 m)$ and $\lslhat{}^{(2)} (2 m +1)$ correspond respectively to the Dynkin diagrams
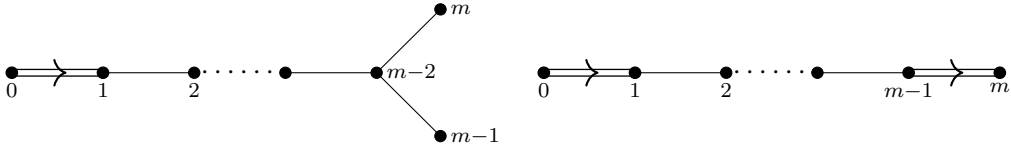
\begin{center}
  \begin{tikzpicture}
  \begin{scope}[scale=1.2]
\draw[line width=.5pt, double distance=2pt,->-] (0,0) node[dot]{} node[below]{${\scriptstyle 0}$} -- (1,0) node[dot]{}node[below]{${\scriptstyle 1}$};
\draw (1,0)--(2,0)node[dot]{}node[below]{${\scriptstyle 2}$};
\node at (2.5,0) {$\cdots\cdots$};
\draw (3,0)node[dot]{}--(4,0)node[dot]{}node[right]{${\scriptstyle m-2}$}--(4.7,0.7)node[dot]{}node[right]{${\scriptstyle m}$};
\draw (4,0)--(4.7,-0.7)node[dot]{}node[right]{${\scriptstyle m-1}$};
  \end{scope}
\begin{scope}[shift={(7,0)},scale=1.2]
\draw[line width=.5pt, double distance=2pt,->-] (0,0) node[dot]{} node[below]{${\scriptstyle 0}$} -- (1,0) node[dot]{}node[below]{${\scriptstyle 1}$};
\draw (1,0)--(2,0)node[dot]{}node[below]{${\scriptstyle 2}$};
\node at (2.5,0) {$\cdots\cdots$};
\draw (3,0)node[dot]{}--(4,0);
\draw[line width=.5pt, double distance=2pt,->-] (4,0) node[dot]{} node[below]{${\scriptstyle m-1}$} -- (5,0) node[dot]{}node[below]{${\scriptstyle m}$};
\end{scope}
  \end{tikzpicture} 
\end{center}
We shall denote by $A^{(2)}_{M-1}$ these two Dynkin diagrams.
Meanwhile, the usual Dynkin diagrams with the opposite directions are denoted by ${}^{t}\!\!A^{(2)}_{M-1}$,
and the corresponding algebras are denoted by ${}^{t}\!\!\lslhat{}^{(2)} (M)$.
It is easy to see that the algebras ${}^{t}\!\!\lslhat{}^{(2)} (M)$ and $\lslhat{}^{(2)} (M)$ are isomorphic.
\end{Remark}

\subsection{Conventions on weights} \label{Parity}

Suppose $\lghat$ is a classical affine Lie algebras of type $\lsohat(M,N)$, $\lglhat{}^{(2)}(M)$, $\lsphat(2l)$ or $\lglhat(l)$, with $\dim \dlho = \rkCar$.
Then we can take a standard basis $\epsilon_1,\epsilon_2,\ldots,\epsilon_{\rkCar}$ of $\dlho$ in the sense that the simple roots can be expressed by $\alpha_{i} = \epsilon_{i} - \epsilon_{i+1}$ whenever $i$ is not an ending point in the Dynkin diagram.
Under this basis we can identify each element in $\dlho$ with its coordinate.
In particular, we use the convention that for an array of coordinates $\lambda = (\lambda_1, \lambda_2, \dots, \lambda_{\rkCar})$, we write $\mathbf{1} = (1, 1, \dots, 1)$ for the array of all $1$'s with the same dimension, and $\lambda \cdot \mathbf{1} = \sum\limits_{i=1}^{\rkCar} \lambda_i$ for the sum of coordinates of $\lambda$.

Let ${\rkCar} \in \mathbb{Z}_{>0}$ and ${r} \in \mathbb{R}_{\ge 0} \cup \{ +\infty \}$. Denote by 
$$\P^{\rkCar}_{{r}} = \{ Y = (y_1, y_2, \dots, y_{\rkCar})\in \mathbb{Z}^{\rkCar} \mid   {r} \ge y_1 \ge \cdots \ge y_{\rkCar} \ge 0 \}$$
the set of partitions with ${\rkCar}$ parts and not ``longer" than ${r}$. For each $Y\in\P^{\rkCar}_{{r}}$, we write $\size(Y)=y_1+y_2+\cdots+y_l$ for the size of $Y$.  If ${r}\in \mathbb{Z}_{>0}$, putting the Young diagram of $Y$ in a rectangle with $r$ rows and ${\rkCar}$ columns, the complement gives rise to another partition $Y^c=(r-y_{\rkCar}, r-y_{\rkCar-1},\ldots,r-y_1)$ in $\P^{\rkCar}_{{r}}$. We denote by $Y^t\in \P^{r}_{\rkCar}$ the conjugate of $Y$, and write  $Y^\dagger = (Y^c)^t$. For instance, if $Y = (2, 1, 1, 0) \in \P^4_3$, then $Y^c = (3, 2, 2, 1)$, $Y^t = (3, 1, 0)$ and $Y^\dagger = (4, 3, 1)$, as illustrated by the following figure. 
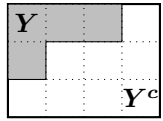
\begin{center}
\begin{tikzpicture}[scale=0.5]
    \draw[thick] (0,0) rectangle (4,-3);

    \draw[fill=lightgray] (0, 0)
    -- (0, -2) -- (1, -2)
    -- (1, -1) -- (3, -1)
    -- (3, 0) -- (0, 0) -- cycle;
    \node[font=\small\bfseries\boldmath] at (0.5, -0.5) {$Y$};
    \node[font=\small\bfseries\boldmath] at (3.5, -2.5) {$Y^c$};
        \draw[dotted] (0,0) grid[step=1] (4,-3);
\end{tikzpicture}
\end{center}

Given a partition $Y = (y_i) \in \P^{\rkCar}_{{\lLevel}}$, we identify $Y$ with the weight  $\sum\limits_{i=1}^{\rkCar} y_i \epsilon_i\in \dlho$ of $\lgo$.
In all finite type cases the weights $Y$ obtained in this way are always dominant and integral.
Conversely, we want to describe all dominant integral weights using partitions, together with certain half-integral weights and diagram automorphisms.
For this purpose, we shall introduce $\sigdomdel({\lLevel})$, a set of desired $\Sigma$-orbit representatives in $\tilde{P}^+ ({\lLevel})$, and its orthogonal projection $\sigdom(\lLevel)$ in $\dlho$ which has a nice correspondence to the set of partitions $\P^{{\rkCar}}_{\lLevel}$. 
The representative set $\sigdomdel({\lLevel})$ consists of all $\sum\limits_{i=0}^{\rkCar} c_i \Lambda_i + c' \delta \in \tilde{P}^+ ({\lLevel})$ satisfying the following additional conditions depending on types, as given in Table \ref{table-sigma-conditions}.
\begin{table}[!htbp]
\centering
\scalebox{0.95}{
\begin{tabular}{c||c|c|c|c}
\hline
$\lghat$    &   $\lsohat(2l)$                           &   $\lsohat(2l+1)$     &   $\lglhat{}^{(2)}(2l)$   &   $\lglhat(l)$  \\ 
\hline
$\Sigma$ action & \makecell{$\sigma_{0,1}:\Lambda_0 \leftrightarrow \Lambda_1$ \\ $\sigma_{n-1,n}:\Lambda_{n-1} \leftrightarrow \Lambda_n$}    &   $\sigma_{0,1}:\Lambda_0 \leftrightarrow \Lambda_1$    &   $\sigma_{n-1,n}:\Lambda_{n-1} \leftrightarrow \Lambda_n$  &   $\sigma_{cyc}^{\#}:\Lambda_i \mapsto \Lambda_{i+1 \pmod l}$ \\
\hline
\makecell{conditions \\ for $\sigdomdel({\lLevel})$} &   $c_0 \ge c_1$, $c_{l} \ge c_{l-1}$      &   $c_0 \ge c_1$       &   $c_{l} \ge c_{l-1}$     &   $c_0 > 0$  \\
\hline
\end{tabular}
}
\caption{The group $\Sigma$ and the set of representatives $\sigdomdel({\lLevel})$ in each type}\label{table-sigma-conditions}
\end{table}

Recall that in each classical affine case the Cartan matrix $A$ is symmetrizable, which means there exists an invertible diagonal matrix $D$ such that $C = DA$ is symmetric. This gives an invariant form on $\tilde{\lieh}{}^*$.
Certainly taking $D_{\rm nor} = {\rm diag}(\frac{a_0^{\vee}}{a_0},\frac{a_1^{\vee}}{a_1},\dots,\frac{a_{\rkCar}^{\vee}}{a_{\rkCar}})$ gives the normalized invariant form $(-|-)_{\rm nor}$.
On the other hand, to make the bilinear form uniformly fitting in each type, through this paper we shall give an alternative choice of $D=\diag(d_1, \dots, d_n)$ such that it induces a natural bilinear form on $\tilde{\lieh}{}^*$ with $(\epsilon_i|\epsilon_j) = \delta_{ij}$. \textbf{This modification causes changes of convention in the following two types.}
\begin{itemize}
    \item For type $D^{(2)}_{{\rkCar}+1}$, we take $\delta = 2(\alpha_0 + \cdots + \alpha_{\rkCar})$ and therefore $D={\rm diag}(\frac{1}{2}, 1,\dots, 1, \frac{1}{2})$. 
        This specific choice of $\delta$ can fit nicely with the action of Virasoro algebra. Under this modification of $D$, the vacuum anomaly of $\lghat = \lsohat(2{\rkCar}+1, 1)$ can be expressed as
        \[
        {\sf d}^{\lghat}_{\lambda;{\lLevel}} = \frac{(\lambda \mid \lambda + 2\rho)}{2 ({\lLevel} + \dCox)} + \left(\frac{{\rm dim} \lieg}{24} - \frac{|\rho|^2}{2 \dCox}\right) \frac{{\lLevel}}{{\lLevel} + \dCox},
        \]
        being in a same form as the ones of $\lsohat(L)$. In this way we can deal with three orthogonal cases uniformly.
    \item For type $C^{(1)}_{\rkCar}$, we take $a^{\vee}_i = 2$ for each $i$ and therefore $D={\rm diag}(2, 1,\dots, 1, 2)$.
        Under this setting the level of each $V(\Lambda_i)$ is $2$, being the same as in $O^{(\r)}$ or $A^{(2)}$ cases for $i$ not an ending point.
\end{itemize}
Detailed data and conventions are listed in Appendix \ref{appendix-datum-types}.

We introduce a parity decomposition for $\tilde{P}^+(\lLevel)$ in non-$A^{(1)}$ cases.
Given $\sum\limits_{i=0}^{\rkCar} c_i \Lambda_i + c'\delta \in \tilde{P}^+({\lLevel})$, since $a_i^\vee = 2$ whenever $i$ is not an ending point on the diagram,
the parity of ${\lLevel}$ depends only on those $c_i$ at ending points. Define $c_h$, $c_t$ as follows:
\begin{table}[!htbp]
\centering
\renewcommand{\arraystretch}{1.4}
\begin{tabular}{c||c|c|c|c|c|c}
\hline
$\lghat$    &   $\lsohat(2l)$                   &   $\lsohat(2l+1)$ &   $\lsohat(2l+1,1)$   &   $\lglhat{}^{(2)}(2l)$           &   $\lglhat{}^{(2)}(2l+1)$ & $\lsphat(2l)$ \\ 
\hline
$c_h$       &   $c_0 + c_1$                     &   $c_0 + c_1$     &   $c_0$               &   $2c_0$                          &   $2c_0$                  &   $2c_0$        \\
\hline
$c_t$       &   $c_{{\rkCar}-1} + c_{\rkCar}$   &   $c_{\rkCar}$    &   $c_{\rkCar}$        &   $c_{{\rkCar}-1} + c_{\rkCar}$   &   $c_{\rkCar}$            &   $2c_{\rkCar}$ \\
\hline
\end{tabular}
\caption{The definition of $c_h$ and $c_t$ in each type}\label{table-ch-ct}
\end{table}

\noindent
The parity decomposition reads as: $\tilde{P}^+({\lLevel}) = \tilde{P}^+({\lLevel})_{\bar{0}, \bar{0}} \sqcup \tilde{P}^+(\lLevel)_{\bar{0}, \bar{1}} \sqcup \tilde{P}^+(\lLevel)_{\bar{1}, \bar{0}} \sqcup \tilde{P}^+(\lLevel)_{\bar{1}, \bar{1}}$, where
\[
\tilde{P}^+(\lLevel)_{\bar{n}_h, \bar{n}_t}:= \left\{ \sum\limits_{i=0}^{\rkCar} c_i \Lambda_i + c'\delta \in \tilde{P}^+({\lLevel}) \mid c_h\equiv n_h\pmod{2},\quad c_t \equiv n_t\pmod{2} \right\}
\]
for $\bar{n}_h, \bar{n}_t \in \mathbb{Z}/2\mathbb{Z}$, and $\sigdomdel(\lLevel)_{\bar{n}_h, \bar{n}_t} = \tilde{P}^+(\lLevel)_{\bar{n}_h, \bar{n}_t} \cap \sigdomdel(\lLevel)$.
Note that $\tilde{P}^+(\lLevel)_{\bar{n}_h, \bar{n}_t}= \varnothing$ unless $n_h + n_t \equiv {\lLevel} \pmod{2}$. 
In case that $\tilde{P}^+(\lLevel)_{\bar{n}_h, \bar{n}_t} \neq \varnothing$, there is a precise description of the sets $\sigdom(\lLevel)_{\bar{n}_h, \bar{n}_t}$ using partitions, that is, 
\[
\sigdom(\lLevel)_{\bar{n}_h, \bar{0}}= \left\{ Y \mid  Y \in \P^{\rkCar}_{\frac{{\lLevel}}{2}} \right\},
\]
\[
\sigdom(\lLevel)_{\bar{n}_h, \bar{1}} = \left\{ Y + \frac{1}{2}\mathbf{1} \mid Y \in \P^{\rkCar}_{\frac{{\lLevel-1}}{2}} \right\}.
\]

In the spirit of \cite[\S 2]{KOO}, two weights $\Lambda, \Lambda' \in \tilde{P}^+(\lLevel)$ are in the same {\em root sieving equivalence class} if and only if $\bar{\Lambda} - \bar{\Lambda}' \in Q$.
With respect to this relation, the parity sets $\tilde{P}^+(\lLevel)_{\bar{n}_h, \bar{n}_t}$ have the following properties.

\begin{Lem}
\label{root-sieving-eq}
   Suppose $R\in\mathbb{N}$. Then the following hold. 
    \begin{enumerate} 
        \item If $\Lambda, \Lambda' \in \tilde{P}^+(\lLevel)$ lie in different parity sets, then $\bar{\Lambda} - \bar{\Lambda}' \notin Q$.
        \item For $\lghat = \lsohat(2l)$, $\lglhat{}^{(2)}(2l)$ or $\lsphat(2l)$, each parity set decomposes as two root sieving equivalent classes,
            \[
           \tilde{P}^+(\lLevel)_{\bar{n}_h, \bar{0}} / {\sim} = \{\lLevel\Lambda_0,\lLevel\Lambda_0 + \varpi_1\}, \quad
            \tilde{P}^+(\lLevel)_{\bar{n}_h, \bar{1}} / {\sim} = \{\lLevel\Lambda_0 + \varpi_{n-1},\lLevel\Lambda_0 + \varpi_{n}\}.
            \]
            The diagram automorphism $\sigma_{0,1}$ (if exists) leaves each equivalent class invariant when $\bar{n}_h = \bar{0}$, while permuting the two classes when $\bar{n}_h = \bar{1}$.
            Similarly, the diagram automorphism $\sigma_{n-1,n}$ (if exists) leaves each equivalent class invariant when $\bar{n}_t = \bar{0}$, while permuting the two classes when $\bar{n}_t = \bar{1}$.
        \item For $\lghat = \lsohat(2l+1)$, $\lsohat(2l+1,1)$ or $\lglhat{}^{(2)}(2l+1)$, each parity set forms a single root sieving equivalent class, say, 
            \[
            \tilde{P}^+(\lLevel)_{\bar{n}_h, \bar{0}} / {\sim} = \{\lLevel\Lambda_0\}, \quad
            \tilde{P}^+(\lLevel)_{\bar{n}_h, \bar{1}} / {\sim} = \{\lLevel\Lambda_0 + \varpi_{n-1}\}.
            \]
    \end{enumerate}
\end{Lem}
\begin{proof}{
    Note that for $\lghat = \lsohat(2l)$, $\lglhat{}^{(2)}(2l)$ or $\lsphat(2l)$, the lattice $Q$ is generated by $\pm\epsilon_i\pm\epsilon_j$, whence $Y - Y'\in Q$ if and only if $\size(Y) -\size(Y') \equiv 0 \pmod 2$ for $Y, Y' \in \mathbb{Z}^l$.
    While for $\lghat = \lsohat(2l+1)$, $\lsohat(2l+1,1)$ or $\lglhat{}^{(2)}(2l+1)$, $Q$ is generated by $\pm\epsilon_i$ and coincides with the lattice $\mathbb{Z}^l$.
    }
\end{proof}

For types $D^{(1)}, B^{(1)}$ and $A^{(2)}_{2 \rkCar -1}$, the Dynkin diagrams have at least one branching point. 
In these cases, we observe that $2\epsilon_{\rkCar} = \alpha_{\rkCar} - \alpha_{\rkCar - 1} \notin \mathring{Q}{}^+$.
To uniformly deal with the weight $2\epsilon_{\rkCar}$ we extend $\tilde{Q}$ to a larger lattice of other types, by replacing the pair of branching vertices with a single vertex with double arrow, as shown in the following table. 
\begin{table}[!htbp]
\centering
\renewcommand{\arraystretch}{1.4}
\begin{tabular}{c|c}
\hline
$\lghat$            &  $\lghat_{\ex}$     \\ 
\hline
$\lsohat(2l)$       &  $\lsohat(2l+1,1)$ or $\lslhat{}^{(2)}(2l+1)$ or ${}^{t}\!\lslhat{}^{(2)}(2l+1)$ or $\lsphat(2l)$  \\
\hline
$\lsohat(2l+1)$       &  $\lsohat(2l+1,1)$ or $\lslhat{}^{(2)}(2l+1)$ \\
\hline
$\lslhat{}^{(2)}(2l)$      &   $\lslhat{}^{(2)}(2l+1)$ or $\lsphat(2l)$  \\
\hline
\end{tabular}
\caption{The possible $\lghat_{\ex}$ for each type of $\lghat$}\label{table-g-ex-all}
\end{table}
In practice the choice of $\lghat_{\ex}$ depends on the actual situation, as shown in Table \ref{table-g-ex}.
In each possible choice, the mapping $\epsilon_i \mapsto \epsilon_i$ induces an embedding $\tilde{{Q}}^+ \hookrightarrow \tilde{{Q}}^+_{\ex}$, and thus we also have $\tilde{{P}}^+_{\ex} \subset \tilde{{P}}^+$.
On the other hand, it is routine to verify that for our above choice of orbit representatives $\sigdomdel({\lLevel})$, we always have
\[
\left(\tilde{{Q}}^+_{\ex} \mid \sigdomdel({\lLevel})\right) \ge 0.
\tag{2.5.1}\label{Qexplus}\]
In particular, for $\beta \in \tilde{{Q}}^+_{\ex}$, we have $\left( \beta | \hat{\rho}\right) \ge 0$, with equality holds if and only if $\beta = k\epsilon_l$ and $(k\epsilon_l|\rho) = 0$.

In non-branching cases, we use the notation $\tilde{{Q}}^+_{\ex}:=\tilde{{Q}}^+$.

\begin{Remark}
    Throughout this paper, we say an affine Lie algebra $\lghat$ (or $\lgtilde$) is of branching type whenever any of its subalgebras $\lgo$ or $\lgo'$ are of the form $\lso(2m)$ with $m \ge 1$, even if the simple Lie algebra degenerates and therefore with no Dynkin diagram. This is because the type-$D$-like algebra $\lso(2m)$ causes some technical issues in the duality theory, and we need to use the tools such as the group $\Sigma$, the representative set $\sigdomdel({\lLevel})$ as well as the extended algebra $\lghat_{\ex}$ to uniformly deal with these issues.

    The degenerate case happens when concerning the small-rank orthogonal algebras $\lso(2m)$ with $m = 1, 2, 3$. In these cases, one can still define the automorphism $\sigma_{m,m-1}$ in a similar way without considering the Dynkin diagram, as well as $\Sigma$ and $\sigdomdel({\lLevel})$. On the other hand, the extended algebra $\lghat_{\ex}$ can be defined directly by the embedding $\epsilon_i \mapsto \epsilon_i$ without referring to the Dynkin diagram. In this way, all the theory can be established in a same way as in the non-degenerate cases.
\end{Remark}

\section{Fock space representations}
\label{sect-Fock-space-constr}

In this section, starting from a pair of affine Lie algebras $(\lghat,\lgdhat)$ in one the four classes: $\{A^{(1)}\}$,  $\{A^{(2)}\}$,  $\{C^{(1)}\}$, and $\{B^{(1)}, D^{(1)},D^{(2)}\}$, we shall construct a Fock space $\F$ which is a subspace of the wedge space of certain vector space determined by the pair and $h=0,\frac{1}{2}$. 
Via the  Clifford algebra action, the Fock space can be realized as a bimodule over  $(\lghat,\lgdhat)$. The decomposition of $\F$ as a bimodule then gives the level-rank duality in Section 5 below.





\subsection{General construction and actions of Clifford algebras}

Suppose $\W$ is a $\mathbb{C}$-vector space with a non-degenerate symmetric bilinear form $(-, -)$ and admits one of the following decompositions:
\begin{itemize}
\item isotropic: $\W = \W^+ \oplus \W^-$ such that $(\W^\pm, \W^\pm) = 0$ and $(-, - )\big|_{\W^+ \times \W^-}$ is non-degenerate;
\item quasi-isotropic: $\W = \W^+ \oplus \mathbb{C}e \oplus \W^-$ such that $(\W^\pm, \W^\pm) = 0$, $(\W^\pm, e) = 0$ and $(e, e)=1$, $(-, -)\big|_{\W^+ \times \W^-}$ is non-degenerate.
\end{itemize}
We uniformly write $\W = \W^{\ge 0} \oplus \W^-$ with $\W^{\ge 0} = \W^+$ or $\W^+ \oplus \mathbb{C}e$ in each case respectively. 

Denote by $T(\W)$ the tensor algebra over $\W$. The Clifford algebra $\Cl(\W)$ is the quotient algebra of $T(\W)$ given by 
\[
\Cl(\W)= T(\W)/\langle x \otimes y + y\otimes x - (x, y) \rangle_{2-\text{sided ideal}}.
\]
Write $xy$ for the residue class of $x \otimes y$ in $\Cl(\W)$.
The exterior algebra $\bigwedge(\W^-)$, being isomorphic to ${\Cl}(\W)/J$ as a vector space, where $J = \Cl(\W)\W^+ $ or $ \Cl(\W)\W^+ + \Cl(\W) (e - \frac{\sqrt{2}}{2})$ is a left ideal generated by $\W^+$ and $e - \frac{\sqrt{2}}{2}$ (if $e$ exists),
gives rise to an irreducible left $\Cl(\W)$-module structure on $\bigwedge(\W^-)$ by defining
\[
w.1 = 0, \forall w\in \W^+; \quad e.1 = \frac{\sqrt{2}}{2}\,  (\mbox{if }e\mbox{ exists}).
\]
Viewed as operators on $\bigwedge(\W^-)$, the elements in  $\W^+$ and $\W^-$ are called {\em annihilation operators} and {\em creation operators}, respectively.


Let $N$ be a positive integer and let $\W_N$ an $N$-dimensional $\mathbb{C}$-vector space with basis $\psi^i$, $i\in {\bf I}_N$ (cf. \S \ref{notation-classic-lie}) and bilinear form $(\psi^i, \psi^j) = \delta_{i, -j}$. Set $\W^{\pm}_{N} = {\rm Span}_{\mathbb{C}}\{ \psi^i \mid i\in {\bf I}_N^{\pm} \}$. 
The module $\bigwedge(W^-_N)$ gives a realization of spin modules of $\lso(N)$:
$\bigwedge(W^-_N) \cong V(\varpi_n)$ if $N = 2n+1$, while $\bigwedge(W^-_N)_{\text{even}} \cong V(\varpi_n)$ and $\bigwedge(W^-_N)_{\text{odd}} \cong V(\varpi_{n-1})$ if $N = 2n$. (See for instance \cite[Chapter 6]{GW})

\begin{Remark}\label{rem-pin}
    According to classical theory of orthogonal groups, the group ${Pin}(N)$ can be realized as a subgroup $\langle \pm v \in \Cl(\W_N) \mid v \in \W_N, (v,v)=-2\rangle$ of $\Cl(\W_N)^{\times}$. 
    In the case $N = 2n$ even, the operator $\tau = \psi^n - \psi^{-n} \in {Pin}(2n)$, being compatible with $\lso(2n)$-action, maps $\bigwedge(W^-_N)_{\text{\rm even}}$ and $\bigwedge(W^-_N)_{\text{\rm odd}}$ onto each other by switching the highest weight vectors.
    Therefore, the assignment $\sigma_{n-1, n} \mapsto \tau$ gives a projective representation of $\Sigma$, the diagram automorphism group of $\lso(2n)$. This  makes $\bigwedge(W^-_N)$ an irreducible $\Sigma \ltimes \lso(2n)$-module (cf. \S \ref{DiagAuto}).
\end{Remark}

Let $h \in \{ 0, \frac{1}{2} \}$. Define the space $\W_N^h= \W_N \otimes t^h \mathbb{C}[t^{\pm 1}]$ with basis $\psi^i(\gamma)$, $i \in {\bf I}_N$, $\gamma \in h + \mathbb{Z}$.
The bilinear form on $\W_N^h$ is defined by $(\psi^i(\gamma), \psi^j(\eta)) = \delta_{i, -j}\delta_{\gamma, -\eta}$.
Put
\[
\W^{h, \pm}_{N} = {\rm Span}_{\mathbb{C}}\{ \psi^i(\gamma) \mid \gamma \in (h + \mathbb{Z})^{\pm}, \text{ or } \gamma = 0, i\in {\bf I}_N^{\pm} \}.
\] 
Observe that when $N$ is odd and $h = 0$, we get a quasi-isotropic decomposition $\W^{h}_{N} = \W^{h, +}_{N} \oplus \mathbb{C}\psi^0(0)\oplus \W^{h, -}_{N}$; otherwise we get an isotropic decomposition $\W^{h}_{N} = \W^{h, +}_{N} \oplus \W^{h, -}_{N}$. 
Under this setting the wedge space $\bigwedge(\W_N^{h,-})$ has a basis consisting of ``monomials" of the form $\psi^{i_1}(\gamma_1) \cdots \psi^{i_m}(\gamma_m)$ with each $\gamma_j < 0$.
By the multiplication in $\Cl(\W_N^h)$, the 
annihilation operator $\psi^i(\gamma)$ with $\gamma > 0$ acts on the monomial as follows 
(the hat symbol indicates that the element is deleted from the sequence).
\begin{equation*}
    \psi^i(\gamma).\psi^{i_1}(\gamma_1) \cdots \psi^{i_m}(\gamma_m) = 
    \begin{cases}
        (-1)^{k+1} \psi^{i_1}(\gamma_1) \cdots \widehat{\psi^{i_k}(\gamma_k)}  \cdots \psi^{i_m}(\gamma_m) ,& \text{ if } \exists k: i_k = -i, \gamma_k = -\gamma; \\
        0 ,& \text{ otherwise. }
    \end{cases}
\end{equation*}

\subsection{Actions of affine Lie algebras}
\label{action-affine}
In this subsection we give concrete constructions of certain wedge product spaces and realize classical affine Lie algebras of types $O^{(\r)}$ and $A^{(2)}$ as operators on them (cf.\cite[\S 3]{Feingold1985}).
These constructions give level-$1$ spin representations of corresponding affine algebras.
We follow the notations in \S \ref{notation-classic-lie}. Fix $M_{\bar{0}}\in \mathbb{Z}_{>0}$ and $M_{\bar{1}}\in \mathbb{Z}_{\ge 0}$ and set $\W=\grws{\W}{M}$ be the $\mathbb{Z}_2$-graded space with $\W_{M_{\bar{0}}}$ in degree $\bar{0}$.

Take $\{ h, h'\} = \{ 0, \frac{1}{2} \}$, that is, $h$ can be $0$ or $\frac{1}{2}$ and $h'=\frac{1}{2}-h$, and define a $\frac{1}{2}\mathbb{Z}$-graded space
$$\W^h = (\W_{M_{\bar{0}}} \otimes t^h \mathbb{C}[t^\pm]) \oplus (\W_{M_{\bar{1}}} \otimes t^{h'} \mathbb{C}[t^\pm]),$$
with basis $\psi^i(\gamma)$ indexed by ${\bf I}^h = ({\bf I}_{M_{\bar{0}}} \times (h + \mathbb{Z})) \sqcup (\gh{{\bf I}}_{M_{\bar{1}}} \times (h' + \mathbb{Z}))$.
The direct sums of (quasi-)isotropic spaces give a (quasi-)isotropic decomposition: $\W^{h, \pm}= \W^{h, \pm}_{M_{\bar{0}}} \oplus \W^{h', \pm}_{M_{\bar{1}}}$.
We have the wedge product space $\bigwedge(\W^{h,-})$ associated to $\W^h$.

To describe the action of affine algebras on $\bigwedge(\W^{h,-})$ we need to introduce the completion of $\Cl(\W^h)$. For simplicity, we write $\Cl$ for $\Cl(\W^h)$ in this subsection. Define a topology on $\Cl$ with open sets $X + \Cl \psi^i(\gamma)$
 where $X\in \Cl, \psi^i(\gamma)\in\W^{h,+}$. 
Then the completion consists of elements of the form: 
\[
\widetilde{\Cl}:= \left\{ X + \sum_{\gamma \ge 0}\sum_i X_{i,\gamma} \psi^i(\gamma)\mid X, X_{i,\gamma} \in \Cl, \psi^i(\gamma)\in \W^{h,+}\right\},
\] 
allowing to take infinite sum of elements with annihilation operators on the right with increasing degrees.
In particular, $\widetilde{\Cl}$ acts in a well-defined way on $\bigwedge(\W^{h,-})$, since for each monomial in $\bigwedge(\W^{h,-})$, all but finitely many annihilation operators act by 0 on it.
We define the following normal ordering $\nord{ \ }$ in $\Cl$: for $x(\gamma),y(\eta)\in\W^h$, 
\begin{equation*}
    \nord{x(\gamma) y(\eta)} = 
    \begin{cases}
        - y(\eta) x(\gamma) ,& \text{ if } \gamma > 0 > \eta;\\
        \frac{1}{2}(x(\gamma) y(\eta) - y(\eta) x(\gamma)) ,& \text{ if } \gamma = 0 = \eta;\\
        x(\gamma) y(\eta) ,& \text{ otherwise}.
    \end{cases}
\end{equation*}
Then, for each $i,j\in {\bf I}_{M_{\bar{0}}}\sqcup \gh{\bf I}_{M_{\bar{1}}}$ and $s\in\frac{1}{2}\mathbb{Z}$, the expression involving infinite sum of the form 
\[
\sum_{\gamma} \nord{\psi^i(s-\gamma) \psi^{-j}(\gamma)} \in \widetilde{\Cl}, 
\]
while hereafter we use the convention that 
the sum with the range of index omitted always runs over the set that ensures all its terms well-defined. For instance, here we need $\psi^i(s-\gamma), \psi^{-j}(\gamma)$ both in $\W^h$ for the above expression.
Such expressions give well-defined operators on $\bigwedge(\W^{h,-})$, which are in fact the homogeneous component of certain vertex operators (cf. \cite[\S 14]{Kac}).  
 
In the following discussion, we shall respectively give the actions of $\lghat = \lsohat{}^{(2)}(M_{\bar{0}}, M_{\bar{1}})$ or $\lglhat{}^{(2)}(M_{\bar{0}})$ on $\bigwedge(\W^{h,-})$, as well as the actions induced by their corresponding graph automorphism groups.
According to the Segal-Sugawara construction (see \S \ref{Segal-Sugawara}), the action of $\lghat$ on $\bigwedge(\W^{h,-})$ induces a representation of the Virasoro algebra $\Vir$ on $\bigwedge(\W^{h,-})$ with central charge $z^{\lghat}$. 
\begin{align*}
    D^{\lghat}: \Vir & \to \End \left(\bigwedge(\W^{h,-})\right), \\
    d(s)&\mapsto D^{\lghat}(s),\\
     z&\mapsto z^{\lghat}=\frac{{\rm dim}\lieg}{1 + \dCox}. 
\end{align*}
We briefly compute the central charge $z^{\lghat}$ in each case.

\subsubsection{Orthogonal case}\label{orthL1}

Keep the notations above. We define the following action of $\lghat = \lsohat{}^{(2)}(M_{\bar{0}}, M_{\bar{1}})$ on $\bigwedge(\W^{h,-})$:
\[
\pi: D_{ij}(s) \mapsto \sum_{\gamma} \nord{\psi^{i} (s - \gamma) \psi^{-j} (\gamma)} \in\widetilde{\Cl}.
\]
Set $M = M_{\bar{0}} + M_{\bar{1}}$, we can quickly compute $\dim \lso(M_{\bar{0}}, M_{\bar{1}}) = \frac{1}{2}M(M-1)$, as well as the dual Coxeter number $\dCox = M - 2$ (cf. Appendix \ref{appendix-datum-types}) .
The central charge of the induced $\Vir$-action is 
$$z^{\lsohat(M_{\bar{0}}, M_{\bar{1}})} = \frac{\dim \lso(M_{\bar{0}}, M_{\bar{1}})}{1 + \dCox} = \frac{M}{2}.$$

We further discuss the action of diagram automorphisms. Suppose $M_{\bar{0}} = 2 m_{\bar{0}}$ is even and $\lfloor \frac{M_{\bar{1}}}{2} \rfloor = m_{\bar{1}}$.
Consider the following finite dimensional subspace of $\W^h$:
\[
\mathring{\W}= \W_{M_{\bar{1}}}^- \otimes t^{-h+\frac{1}{2}} + \W_{M_{\bar{0}}} \otimes t^{-h} + \W_{M_{\bar{1}}}^+ \otimes t^{-h-\frac{1}{2}}.
\] 
Note that $\mathring{\W} \cong \mathbb{C}^{M_{\bar{0}} + 2m_{\bar{1}}}$ is the natural  representation of $\lgo = \lso(M_{\bar{0}} + 2m_{\bar{1}})$,
and there is a decomposition $\W^h = (\mathring{\W} \otimes \mathbb{C}[t^{\pm}]) \oplus (\W_{M_{\bar{1}}}^0 \otimes t^{h-\frac{1}{2}} \otimes \mathbb{C}[t^{\pm}])$ with $\lgo$ acting trivially on the second part.
Restricting to the negative part we have $\mathring{\W} \cap \W^{h, -} = \mathring{\W}$ for $h = \frac{1}{2}$, while for $h=0$, $\mathring{\W} \cap \W_M^{h, -} \cong \mathbb{C}^{m_{\bar{0}} + m_{\bar{1}}}$ and its wedge space becomes the sum of $\pm$ spin representations of $\lgo$.
In both cases the action on $\bigwedge(\mathring{\W} \cap \W^{h, -})$ lifts to ${Pin}(M_{\bar{0}} + 2m_{\bar{1}})$. Observe that as a $\lgo$-module, 
\[
\bigwedge(\W^{h,-}) = \bigwedge(\mathring{\W} \cap \W^{h, -}) \otimes \bigwedge(\mathring{\W} \otimes \mathbb{C}[t^-]) \otimes \bigwedge(\W_{M_{\bar{1}}}^0 \otimes t^{h-\frac{1}{2}} \otimes \mathbb{C}[t^-]).
\]
If $h = 0$, then $\bigwedge(\W^{h,-}) \cong \bigwedge(\mathbb{C}^{m_{\bar{0}} + m_{\bar{1}}}) \otimes \bigwedge(\mathbb{C}^{M_{\bar{0}} + 2m_{\bar{1}}} \otimes \mathbb{C}[t^-]) \otimes \mathbb{C}[t^-]$ lifts to ${Pin}(M_{\bar{0}} + 2m_{\bar{1}})$\cite[Remark 3.5]{Wang1999};
if $h = \frac{1}{2}$, then $\bigwedge(\W^{h,-}) \cong \bigwedge(\mathbb{C}^{M_{\bar{0}} + 2m_{\bar{1}}} \otimes \mathbb{C}[t^-]) \otimes \mathbb{C}[t^-]$ lifts to $O(M_{\bar{0}} + 2m_{\bar{1}})$\cite[Remark 3.1]{Wang1999}.

Eventually using $\tau \in {Pin}(M_{\bar{0}} + 2m_{\bar{1}})$, we are ready to define the action of $\sigma_{m_{\bar{0}}, m_{\bar{0}}-1}$ on $\bigwedge(\W^{h,-})$: up to $\pm$ 
it acts by permuting $1$ with $\psi^{-m_{\bar{0}}}(0)$ on $\bigwedge(\mathbb{C}^{m_{\bar{0}} + m_{\bar{1}}})$(cf. Remark \ref{rem-pin}), and by permuting $\psi^{m_{\bar{0}}}(\gamma)$ with $\psi^{-m_{\bar{0}}}(\gamma)$, $\gamma < 0$ on rest of the spaces.
Similar actions are given for $\sigma_{\gh{m_{\bar{1}}}, \gh{(m_{\bar{1}}-1)}}$ in case $M_{\bar{1}} = 2m_{\bar{1}}$ even.
Later in \S \ref{diag-auto-on-abacus} below we shall give a graphical description of these actions.

\subsubsection{General linear case}
In general linear case we fix $h = 0$ and take $M_{\bar{1}} = 0$ and simply write $M=M_{\bar{0}}$. To realize the twisted general linear algebras on $\W^h$ we slightly modify the bilinear form by a sign twist on the neutral fermonic field:
$(\psi^0(\gamma), \psi^0(\eta)) = (-1)^\gamma \delta_{\gamma, -\eta}$. 
Since $h = 0$ and $\gamma, \eta \in \mathbb{Z}$, the twisted bilinear form is again symmetric.
Note that this modification only happens for type $A^{(2)}_{2\rkCar}$ algebra $\lglhat{}^{(2)}(2\rkCar+1)$. Abusing the notations we shall not distinguish the modified spaces $\Cl(\W^h)$ and $\bigwedge(\W^{h,-})$. 

The action $\lghat = \lglhat{}^{(2)}(M)$ on $\bigwedge(\W^{h,-})$ reads as follows. 
\[    
\pi: A_{ij}(s) \mapsto
\begin{cases}
    \sum\limits_{\gamma\in\mathbb{Z}} \nord{\psi^{i} (s - \gamma) \psi^{-j} (\gamma)} ,& \text{ if } ij > 0 \text{ or } i > 0, j = 0;\\
    \sum\limits_{\gamma\in\mathbb{Z}}(-1)^\gamma \nord{\psi^{i} (s - \gamma) \psi^{-j} (\gamma)} ,& \text{ if } ij < 0 \text{ or } i \le 0, j = 0.
\end{cases}
\]
Compute $\dim \lsl(M) = M^2 - 1$, as well as the dual Coxeter number $\dCox = M$ (see Appendix \ref{appendix-datum-types}).
The central charge of the induced $\Vir$-representation is
$$
z^{\lglhat{}^{(2)}(M)} = z^{\lslhat{}^{(2)}(M)} + 1 = \frac{\dim \lsl(M)}{1 + \dCox} + 1 = M.
$$
Similarly as in the orthogonal cases, when $M = 2m$ is even, the simple Lie subalgebra $\lgo$ is of type $D$. One can again define the action of the diagram automorphism $\sigma_{m, m-1}$ by lifting to the corresponding $Pin$ group.

\begin{Remark}
    The space with $h = \frac{1}{2}$ in general linear case corresponds to the affine algebra ${}^{t}\!\!A^{(2)}_{M}$ with Dynkin diagram labelling in the opposite direction.
    Though these two types of affine algebras are isomorphic, the construction of the algebras and spaces are much more complicated and we will omit it in this paper.
    Analogue dual theory of pairs $({}^{t}\!\!A^{(2)}_{L}, A^{(2)}_{R})$ can be given in a similar way.
\end{Remark}

Summarizing orthogonal and general linear cases we have:
\begin{Prop}[{\cite[Theorem C]{Feingold1985}}]
    \label{prop-affine-action}
    Let $\lghat = \lsohat{}^{(2)}(M_{\bar{0}}, M_{\bar{1}})$ or $\lglhat{}^{(2)}(M)$, and let $\bigwedge(\W^{h,-})$ be the wedge product defined above.
    The map $\pi: \lghat \to \widetilde{\Cl}(\W^h) \subset \End (\bigwedge(\W^{h,-}))$ defined above makes $\bigwedge(\W^{h,-})$ an integrable highest weight $\lghat$-module of level $1$, with the vacuum vector $1$ being a highest weight vector.
    The module $\bigwedge(\W^{h,-})$ is either irreducible or a direct sum of two irreducible summands which are related by a diagram automorphism of $\lghat$.
    In other words, $\bigwedge(\W^{h,-})$ is an irreducible $\Sigma \ltimes \lghat$-module.
\end{Prop}

\subsection{Actions of Virasoro algebras}

By Segal-Sugawara construction (cf. \S \ref{Segal-Sugawara}), the above action of $\lghat = \lsohat{}^{(2)}(M_{\bar{0}}, M_{\bar{1}})$ or $\lglhat{}^{(2)}(M_{\bar{0}})$ on $\bigwedge(\W^{h,-})$ defined above gives rise to a representation 
 $D^{\lghat}: \Vir \to \End \left(\bigwedge(\W^{h,-})\right)$.
Besides $D^{\lghat}$, there is another representation of $\Vir$ on $\bigwedge(\W^{h,-})$ given by the Clifford algebra structure \cite[\S 5]{Feingold1985}:
\begin{align*}
    D^{\Cl}: \Vir & \to \End (\bigwedge(\W^{h,-})), \\
    d(s)&\mapsto D^{\Cl}(s)= -\frac{1}{4} \sum_{(i, \gamma)} (s - 2\gamma) \nord{\psi^{i}(s-\gamma) \psi^{-i}(\gamma)} +\frac{1}{16} {\rm dim}\W^h(0) \delta_{s, 0},\\
    z&\mapsto z^{\Cl}:= \frac{M_{\bar{0}}+M_{\bar{1}}}{2}
\end{align*}
with the following nice properties:
\begin{align*}
    &[D^{\Cl}(s), a(\eta)] = -\left(\frac{s}{2} + \eta\right) a(s + \eta), \quad a(\eta)\in \W^h;\\
    &[D^{\Cl}(s), x(s')] = -s' x(s + s'), \quad x(s')\in \lghat.
\end{align*}
Comparing these two representations using coset $\Vir$-module construction, we have:
\begin{Prop}[{\cite[2.4b]{Hasegawa1989}}]
    In orthogonal cases, we have $z^{\lghat} = z^{\Cl}$, thus $D^{\lghat} = D^{\Cl}$. 
    In general linear cases, we have $z^{\lghat} = 2 z^{\Cl}$, thus $D^{\lghat}(s) = \frac{1}{2} D^{\Cl}(2s) + \delta_{s, 0}\frac{M_{\bar{0}}}{32}$.
\end{Prop}

Using this proposition, each monomial $v = a_1(\eta_1) \cdots a_p(\eta_p).1\in \bigwedge(\W^{h,-})$ is an eigenvector of both $D^{\lghat}(0)$ and $D^{\Cl}(0)$. The eigenvalue of $v$ under  $D^{\Cl}(0)$, denoted by ${\sf d}_v$, is 
\begin{equation}
    {\sf d}_v = -(\eta_1 + \cdots + \eta_p) + \frac{1}{16} {\rm dim}\W^h(0).\tag{3.3.1}\label{degreeV}
\end{equation}

\subsection{The Fock space $\F$ and commutative actions}
\label{subsection-liepair}

In this subsection, based on the previous two subsections, we shall construct the crucial Fock space $\F$ which will be used to establish a duality theory later on. The process is as follows. We consider certain pair of affine Lie algebras  $(\lghat_{\lleft}, \lgdhat)$, and embed $\lghat_{\lleft} \oplus \lgdhat$ into a larger algebra $\lghat_{\bigalg}$ related to the tensor product of natural representations of two algebras. Via restriction, the level-$1$ Fock space representation $\F$  of $\lghat_{\bigalg}$ becomes a $(\lghat_{\lleft}, \lgdhat)$-bimodule.

From now on we use the convention that the un-dotted or dotted datum corresponds to the algebra $\lieg_{\lleft}$ or $\liegd$, respectively, while the datum corresponding to $\lieg_{\bigalg}$ is labelled by the subscript ${\bigalg}$.

Let $\W_L=\grws{\W}{L}$ and $\W_R=\grws{\W}{R}$ be $\mathbb{Z}_2$-graded $\mathbb{C}$-vector spaces with $L_{\bar{0}}, R_{\bar{0}}$ positive. Their tensor product is also $\mathbb{Z}_2$-graded, 
\[\W=\W_L\otimes\W_R=\underbrace{(\W_{L_{\bar{0}}}\otimes \W_{R_{\bar{0}}}\oplus \W_{L_{\bar{1}}}\otimes \W_{R_{\bar{1}}})}_{\W_{\bar{0}}}\oplus \underbrace{(\W_{L_{\bar{0}}}\otimes \W_{R_{\bar{1}}}\oplus \W_{L_{\bar{1}}}\otimes \W_{R_{\bar{0}}})}_{\W_{\bar{1}}},\]
with basis $\psi^{i,p}=\psi^i\otimes \psi^p$, $i\in {\bf I}_{L_{\bar{0}}}\sqcup \gh{\bf I}_{L_{\bar{1}}},p\in {\bf I}_{R_{\bar{0}}}\sqcup \gh{\bf I}_{R_{\bar{1}}}$. One has  subspaces 
\[\W^{\pm}=\mathrm{Span}_{\mathbb{C}}\left(\{\psi^{i,p}\mid i\in {\bf I}_{L_{\bar{0}}}^{\pm}\sqcup (\gh{\bf I}_{L_{\bar{1}}})^{\pm}, \mbox{or }i=0,\gh{0}, p\in {\bf I}_{R_{\bar{0}}}^{\pm}\sqcup (\gh{\bf I}_{R_{\bar{1}}})^{\pm}\}\bigcup \{\psi^{\pm e}, \psi^{\pm \gh{e}}\} \right),\] 
where $\psi^{\pm e}:= \frac{\sqrt{2}}{2} (\psi^{0, 0} \pm \ii \psi^{\gh{0}, \gh{0}})\in \W_{\bar{0}}$ and $\psi^{\pm \gh{e}}:= \frac{\sqrt{2}}{2} (\psi^{0, \gh{0}} \pm \ii \psi^{\gh{0}, 0})\in\W_{\bar{1}}$.  

For $\{ h, h'\} = \{ 0, \frac{1}{2} \}$, one can construct a $\frac{1}{2}\mathbb{Z}$-graded space
\[\W^h=(\W_{\bar{0}}\otimes t^h\mathbb{C}[t^{\pm}])\oplus (\W_{\bar{1}}\otimes t^{h'}\mathbb{C}[t^{\pm}]),\]
and define a bilinear form on $\W^h$ by $(\psi^{i,p}(\gamma), \psi^{j,q}(\eta)) = \delta_{i,-j} \delta_{p, -q} \delta_{\gamma, -\eta}$.  The corresponding (quasi-) isotropic decomposition is as follows. 
$$\W^{h, \pm} = {\rm Span}_{\mathbb{C}}\{ \psi(\gamma)\in\W^h \mid \gamma \in \frac{1}{2}\mathbb{Z}^{\pm}, \text{ or } \gamma = 0, \psi\in \W^{\pm}\}.$$
The space $\bigwedge(\W^{h,-})$ can be viewed as a module over $\lghat_{\bigalg}=\lsohat{}^{(2)}(\dim \W_{\bar{0}}, \dim \W_{\bar{1}})$ or $\lglhat{}^{(2)}(\dim \W)$ (see Proposition \ref{prop-affine-action}).  
Depending on the types of $\lghat$ and $\lgdhat$, we define the Fock space $\F$ associated to the pair $(\lghat,\lgdhat)$ as follows. 

\begin{Def}\label{def-Fock-space}
Keep the notations above with $h=0$ or $\frac{1}{2}$. 
\begin{itemize}
    \item If $(\lghat,\lgdhat)=(\lsohat(L_{\bar{0}},L_{\bar{1}}),\lsohat(R_{\bar{0}},R_{\bar{1}}))$ with $L_{\bar{0}} L_{\bar{1}} R_{\bar{0}} R_{\bar{1}}$ being odd, then we define the Fock space
    $$
    \F = \bigwedge(\W^{h,-})_{\rm {even}},
    $$ 
    that is, the subspace of $ \bigwedge(\W^{h,-})$ spanned by those  elements $x_1\wedge x_2\wedge\cdots\wedge x_{2n}$ with $n\in\mathbb{N}$ and $x_i\in\W^{h,-}$ for all $i$. 
    \item If $(\lghat,\lgdhat)=(\lglhat(l),\lslhat(r))$, where  $L_{\bar{0}}=2l, R_{\bar{0}}=2r,L_{\bar{1}}=R_{\bar{1}}=0$, then we take a subspace $\bar{\W}=(\W^+_{2l} \otimes \W^+_{2r}) \oplus (\W^-_{2l} \otimes \W^-_{2r})$ of $\W$. Let $\bar{\W}^{h,-}=(\bar{\W}\otimes t^{h}\mathbb{C}[t^{\pm}])\cap \W^{h,-}$, and define 
     $$\F=\bigwedge(\bar{\W}^{h,-}).$$ 
    \item For all other cases, we define 
    $$\F=\bigwedge(\W^{h,-})$$
    with $h=0$ when $(\lghat,\lgdhat)=(\lglhat{}^{(2)}(L), \lslhat{}^{(2)}(R))$. 
\end{itemize}
\end{Def}
 
In each case, via restriction, the actions of $\lghat$ and $\lgdhat$ on $\F$ are of the forms:
\[\begin{array}{cc}
\begin{array}{rcl}
     \pi:  \lghat & \to & \End (\F) \\
     X_{ij}(n) & \mapsto & \sum\limits_{(s, \gamma)} \pm\nord{\psi^{i,s}(n-\gamma) \psi^{-j, -s}(\gamma)} \\
     c_{\lleft} &\mapsto&  \levelR;\\
\end{array}
& 
\begin{array}{rcl}
  \dot{\pi}:  \lgdhat & \to & \End (\F) \\
     X_{pq}(n) & \mapsto & \sum\limits_{(k, \gamma)} \pm\nord{\psi^{k,p}(n-\gamma) \psi^{-k, -q}(\gamma)} \\
    \dot{c} & \mapsto & \levelL,\\
\end{array}
\end{array}\]
with specific choices of signs depending on types.
Detailed action rules are listed in Appendix \ref{appendix-aff-aff}.

\subsubsection{Orthogonal case}
\label{liepair-ortho}
Let $(\lieg, \liegd) = (\lso(L_{\bar{0}}, L_{\bar{1}}), \lso(R_{\bar{0}}, R_{\bar{1}}))$. Then $\lieg$ and $\liegd$ naturally act on $\W_L$ and $\W_R$, respectively, and 
 the tensor product gives rise to a natural embedding:
\[
\begin{matrix}
    \lso(L_{\bar{0}}, L_{\bar{1}}) & \oplus & \lso(R_{\bar{0}}, R_{\bar{1}}) & \hookrightarrow & \lieg_{\bigalg}:=\lso(L_{\bar{0}} R_{\bar{0}} + L_{\bar{1}} R_{\bar{1}}, L_{\bar{1}} R_{\bar{0}} + L_{\bar{0}} R_{\bar{1}}) \\
    \curvearrowright &  & \curvearrowright  &   & \curvearrowright \\
    (\grws{\W}{L}) & \otimes & (\grws{\W}{R}) & \cong & \W=\W_{\bar{0}}\oplus \W_{\bar{1}}.
\end{matrix}
\]
Note that this embedding is compatible with the $\mathbb{Z}_2$-gradings, thus we can extend to the affine algebra:
\begin{align*}
    \lghat_{\lleft} \oplus \lgdhat &\to \lghat_{\bigalg} \\
    x(n) + y(m) & \mapsto  (x \otimes I)(n) + (I \otimes y)(m),\\
    \eta c_{\lleft} + \eta' \dot{c} & \mapsto  (R\eta + L\eta') c. 
\end{align*}

\begin{Remark}
\label{rem-Fock-D2D2}
For type $(D^{(2)},D^{(2)})$, that is, $L_{\bar{0}} L_{\bar{1}} R_{\bar{0}} R_{\bar{1}}$ odd, the action of $\lso(L_{\bar{0}}, L_{\bar{1}}) \oplus \lso(R_{\bar{0}}, R_{\bar{1}})$ fails to cover $\W$ 
in the sense that the base elements $\psi^{\pm e}$, $\psi^{\pm \gh{e}}$ are of weights $0$ both for $\lso(L_{\bar{0}}, L_{\bar{1}})$ and $\lso(R_{\bar{0}}, R_{\bar{1}})$.

Recall that as a $\lsohat(L_{\bar{0}} R_{\bar{0}} + L_{\bar{1}} R_{\bar{1}}, L_{\bar{1}} R_{\bar{0}} + L_{\bar{0}} R_{\bar{1}})$-module, the space $\bigwedge (\W^{0,-})$ decomposes as two irreducible modules with highest weight vectors $1$ and $\psi^{-e}(0)$.
It happens that these two irreducible $\lghat_{\bigalg}$-modules are isomorphic as $\lghat_{\lleft} \oplus \lgdhat$-modules, 
thus in the decomposition of $\bigwedge (\W^{0,-})$ as $\lghat_{\lleft} \oplus \lgdhat$-modules every irreducible components will have two isomorphic copies coming from $1$ and $\psi^{-e}(0)$.
The case $h = \frac{1}{2}$ becomes analogously with $\psi^{-e}(0)$ replaced by $\psi^{-\gh{e}}(0)$.

Consequently, to prevent the redundant multiplicities, here the Fock space $\F = \bigwedge(\W^{h,-})_{\text{even}}$ is taken to be one of the irreducible direct summands of $\bigwedge (\W^{h,-})$(cf. Definition \ref{def-Fock-space}).
\end{Remark}

Let us consider the action of diagram automorphisms on the Fock space. Suppose $L_{\bar{0}} = 2 l_{\bar{0}}$ and $\lfloor \frac{L_{\bar{1}}}{2} \rfloor = l_{\bar{1}}$. 
Analogously as in level-$1$ case in \S\ref{orthL1}, consider $\lgo_{\lleft}-$action on $\F$ we have the decomposition:
\[
\F \cong \bigwedge(\mathbb{C}^{l_{\bar{0}} + l_{\bar{1}}} \otimes \mathbb{C}^{R_{\overline{2h}}}) \otimes \bigwedge(\mathbb{C}^{L_{\bar{0}} + 2l_{\bar{1}}} \otimes \mathbb{C}^{\mathbb{N}}) \otimes \mathbb{C}^{\mathbb{N}},
\]
which by the same reason lifts to ${Pin}(L_{\bar{0}} + 2l_{\bar{1}})$, allowing us to define the action of $\sigma^{\lleft}_{l_{\bar{0}}, l_{\bar{0}}-1}$.
Write the monomial basis of $\F$ by $\prod\limits_{p \in S} \psi^{-l_{\bar{0}}, p}(0) \cdot \prod\limits_{k} \psi^{i_k, p_k}(\gamma_k)$ 
with $S\subset I_{R_{\bar{0}}}$ if $h = 0$ or $S\subset \gh{I}_{R_{\bar{1}}}$ if $h = \frac{1}{2}$ and $(i_k, \gamma_k) \neq (-l_{\bar{0}}, 0)$, we can concretely describe the action of $\sigma^{\lleft}_{l_{\bar{0}}, l_{\bar{0}}-1}$(up to $\pm$):
\[
\sigma^{\lleft}_{l_{\bar{0}}, l_{\bar{0}}-1}: \prod\limits_{p \in S} \psi^{-l_{\bar{0}}, p}(0) \cdot \prod\limits_{k} \psi^{i_k, p_k}(\gamma_k) \longleftrightarrow \prod\limits_{p \in S^c} \psi^{-l_{\bar{0}}, -p}(0) \cdot \prod\limits_{k} \psi^{i_k, p_k}(\gamma_k).
\]
Similar action can be defined for $\sigma^{\lleft}_{\gh{l}_{\bar{1}}, \gh{(l_{\bar{1}}-1)}}$ and $\lgdhat$ counterparts. These define a $(\Sigma_{\lleft} \ltimes \lghat_{\lleft}, \dot{\Sigma} \ltimes \lgdhat)$-action on $\F$.
More intuitive explanations and examples of these actions will be given in \S\ref{graphic} below.

\subsubsection{General linear case}
\label{liepair-gl}
Let $(\lghat_{\lleft}, \lgdhat) = (\lglhat{}^{(2)}(L), \lslhat{}^{(2)}(R))$, $\lghat_{\bigalg} = \lglhat{}^{(2)}(LR)$ and fix $h=0$.
Observe that on the tensor space $\W = \W_L \otimes \W_R$, the tensor product of matrices $J_L \otimes J_R$ coincides with $J_{LR}$. 
This means the natural embedding $\lglhat(L) \oplus \lslhat(R) \hookrightarrow \lglhat(LR)$ is compatible with automorphisms on each algebra.
That is, it induces a Lie algebra embedding:
\begin{align*}
    \lghat_{\lleft} \oplus \lgdhat &\to \lghat_{\bigalg} \\
    x(n) + y(m) & \mapsto  (x \otimes I)(n) + (I \otimes y)(m),\\
    \eta c_{\lleft} + \eta' \dot{c} & \mapsto  (R\eta + L\eta') c. 
\end{align*}
 
Similarly as in the orthogonal cases, the actions of the diagram automorphisms are again defined by lifting to corresponding ${Pin}$ group.

\subsubsection{Types $A^{(1)}$ and $C^{(1)}$}
\label{liepair-AC}
For completeness we also state the known result of types $A^{(1)}$ and $C^{(1)}$, following \cite[\S 4]{Hasegawa1989}. In this part, we assume that $L_{\bar{1}}=R_{\bar{1}}=0$, $L=L_{\bar{0}}=2l$ and $R=R_{\bar{0}}=2r$.

For type $C^{(1)}$, let $(\lghat_{\lleft}, \lgdhat) = (\lsphat(2l), \lsphat(2r))$ and $\lghat_{\bigalg} = \lsohat(4lr)$ acting on $\W = \W_{2l} \otimes \W_{2r} = \W_{4lr}$.
The embedding $\lghat_{\lleft} \oplus \lgdhat\to \lghat_{\bigalg}$ can be described explicitly as follows. 
\begin{align*}
    \lghat_{\lleft} \oplus \lgdhat &\to \lghat_{\bigalg} \\
    C_{ij}(n) & \mapsto
    \begin{cases}
    (C_{ij} \otimes I_{2r})(n), & ij>0;\\
    (C_{ij} \otimes {\rm diag}(I_{r}, I_{-r}))(n), & ij<0;\\
    \end{cases}\\
    C_{pq}(m) & \mapsto
    \begin{cases}
    (I_{2l} \otimes C_{pq})(m),& pq>0;\\
    ({\rm diag}(I_{l}, I_{-l}) \otimes C_{pq})(m),& pq<0;\\
    \end{cases} \\
    \eta c_{\lleft} + \eta' \dot{c} & \mapsto  (R\eta + L\eta') c. 
\end{align*}

For type $A^{(1)}$, recall from  Definition \ref{def-Fock-space} that $(\lghat,\lgdhat)=(\lglhat(l),\lslhat(r))$, and the Fock space is  $\F=\bigwedge(\bar{\W}^{h,-})$, where  $\bar{\W}=(\W^+_{2l} \otimes \W^+_{2r}) \oplus (\W^-_{2l} \otimes \W^-_{2r})$ and $\bar{\W}^{h,-}=(\bar{\W}\otimes t^{h}\mathbb{C}[t^{\pm}])\cap \W^{h,-}$. Then $\F$ is a module over $\lghat_{\bigalg}=\lsohat(\bar{W})=\lsohat(2lr)$. 

Since $\lieg_{\lleft} = \lgl(l)$ and $\liegd = \lsl(r)$ act naturally on $\W_{2l}^+$ and $\W_{2r}^{+}$, there is an embedding
$\lgl(l) \oplus \lsl(r) \hookrightarrow \lgl(lr)$, and $\lgl(lr)$ acts naturally on $\W^+_{2l} \otimes \W^+_{2r}$.
Identify the dual space $(\W^+_{2l} \otimes \W^+_{2r})^* \cong \W^-_{2l} \otimes \W^-_{2r}$ with respect to the bilinear form. 
Then $\lgl(lr)$ naturally acts on $\bar{\W}$ by $\lgl(lr) \hookrightarrow \lso(\bar{\W}), x \mapsto x - J x J$. 
The composition $\lieg\oplus\liegd\hookrightarrow \lgl(lr)\hookrightarrow\lso(2lr)$ then gives an embedding of the corresponding affine algebras. 
\[
\lghat_{\lleft} \oplus \lgdhat \to \lghat_{\bigalg}.
\]

Consider the action of diagram automorphisms. In this case we take $\Sigma_{\lleft} = \langle \sigma^{\#}_{cyc} \rangle$ and $\dot{\Sigma} = \{ {\rm id} \}$(ref. \cite[\S 4.2]{Hasegawa1989}).
Later we will give a graphical explanation for the action of $\sigma^{\#}_{cyc}$ on Fock space using abacus configuration, see \S\ref{diag-auto-on-abacus} below.

\medskip 

In summary, for all cases in \S\ref{liepair-ortho}, \S\ref{liepair-gl} and \S\ref{liepair-AC}, the Fock space $\F$ becomes a restricted $\lghat_{\lleft}$-module of level $\levelR$ as well as a restricted $\lgdhat$-module of level $\levelL$, where $\levelR = r, \levelL = l$ for type $(A^{(1)},A^{(1)})$ and $\levelR = R, \levelL = L$ otherwise.
In particular, the actions of $\lghat_{\lleft}, \lgdhat$ and $\lghat_{\bigalg}$ respectively induce three actions $D^{\lghat_{\lleft}}, D^{\lgdhat}$ and $D^{\lghat_{\bigalg}}$ of $\Vir$ on $\F$.
In fact, by formula (\ref{eq-central-charge}), it is straightforward to verify the following proposition by  comparing central charges.
\begin{Prop} [{cf. \cite[\S 4.1]{Hasegawa1989}}] \label{VirEv}
We have an equality
    $D^{\lghat_{\lleft}}+D^{\lgdhat} = D^{\lghat_{\bigalg}}$ as  representations of $\Vir$ on $\F$. In particular, if $v\in \F$ is a $(\lghat_{\lleft}, \lgdhat)$-highest weight vector of classical weight $(\lambda_{\lleft}, \dot{\lambda})$, then $v$ is a common eigenvector of  $D^{\lghat_{\lleft}}(0),D^{\lgdhat}(0)$ and $D^{\lghat_{\bigalg}}(0)$, and the eigenvalues satisfy  
    \[
    {\sf d}^{\lghat_{\lleft}}_{\lambda_{\lleft};\levelR} + {\sf d}^{\lgdhat}_{\dot{\lambda};\levelL} = 
    \begin{cases}
        \frac{1}{2}{\sf d}_v + \frac{LR}{32}, &\mbox{ for type } (A^{(2)},A^{(2)});\\
        {\sf d}_v, &\mbox{ otherwise}.
    \end{cases}
    \]
\end{Prop}

A big task is to decompose $\F$ into the sum of  irreducible $(\lghat,\lgdhat)$-bimodules, which give the level-rank duality theory. Before that, we shall give a  combinatorial model of $\F$ in the next section.

\section{Graphical configurations: Maya diagrams and abaci} \label{graphic}

It is well-known \cite{U} for type $A^{(1)}$, the level-rank duality on semi-infinite wedge spaces are closely related to multipartitions. In this section we will give, for all classical types, a uniform graphical configuration which draws Maya diagrams on infinitely many pieces of ``chessboards", and interpret the action of Lie algebras on this model.
These two types of diagrammatic models are closely related via certain graphical operations.

\subsection{Maya diagram models and moves}\label{Maya-diagrams}

We introduce a graphical interpretation of the Fock space $\F$.

First we consider the cases not of type $(A^{(1)},A^{(1)})$. Recall that
\[\W^h=(\W_{\bar{0}}\otimes t^h\mathbb{C}[t^{\pm}])\oplus (\W_{\bar{1}}\otimes t^{h'}\mathbb{C}[t^{\pm}])\]
has a basis $\psi^{i,p}(\gamma)$ indexed by 
\[{\bf I}^h := ((({\bf I}_{L_{\bar{0}}} \times {\bf I}_{R_{\bar{0}}}) \sqcup (\gh{\bf I}_{{L}_{\bar{1}}} \times \gh{\bf I}_{{R}_{\bar{1}}})) \times (h + \mathbb{Z})) \sqcup (({\bf I}_{L_{\bar{0}}} \times \gh{\bf I}_{R_{\bar{1}}}) \sqcup (\gh{\bf I}_{{L}_{\bar{1}}} \times {\bf I}_{{R}_{\bar{0}}}) \times (h' + \mathbb{Z})).\]
The set ${\bf I}^h$ is endowed with lexicographical order. This can be done as follows. We always assume that the elements in $\gh{\bf I}_M$ are  bigger than the elements in ${\bf I}_N$ for each pair of positive integers  $M,N$. Together with the order on ${\bf I}_M$ and $\gh{\bf I}_N$ given in \S \ref{notation-classic-lie}, 
for each $(i,p,\gamma)\in {\bf I}^h$, the lexicographical order is respect to $\gamma, i$ and $p$.

For each homogeneous component $\W_{L_a}\otimes \W_{R_b}\otimes \mathbb{C}t^{\gamma}$ of $\W^{h}$ with  $a,b\in\{\bar{0},\bar{1}\}$, we draw a grided board with $L_a$ columns and $R_b$ rows labelled by the corresponding subset of ${\bf I}^h$. For instance, assume that $L_{\bar{0}}=4$ and $R_{\bar{1}}=3$ for $\W_{L_{\bar{0}}}\otimes \W_{R_{\bar{1}}}\otimes \mathbb{C}t^{-\frac{1}{2}}$ we draw a grided board of degree $-\frac{1}{2}$ with columns labelled by ${\bf I}_{L_{\bar{0}}}={\bf I}_4$ and rows labelled by $\gh{\bf I}_{R_{\bar{1}}}=\gh{\bf I}_3$. 
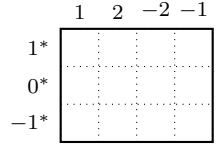
\begin{center}
\begin{tikzpicture}[scale=0.5]
\draw[dotted] (0,0) grid[step=1] (4,-3);
\draw[thick] (0,0) rectangle (4,-3);
\node[above] at (0.5,0) {${\scriptstyle 1}$};
\node[above] at (1.5,0) {${\scriptstyle 2}$};
\node[above] at (2.5,0) {${\scriptstyle -2}$};
\node[above] at (3.5,0) {${\scriptstyle -1}$};
\node[left] at (0,-0.5) {${\scriptstyle \gh{1}}$};
\node[left] at (0,-1.5) {${\scriptstyle \gh{0}}$};
\node[left] at (0,-2.5) {${\scriptstyle \gh{-1}}$};
    \end{tikzpicture}
\end{center}
A cell in $i$-th column and $p$-th row are labelled by $(i, p, \gamma)\in {\bf I}^{h}$ and corresponds to $\psi^{i, p}(\gamma) \in \W^{h}$. We arrange the ``initial'' boards as follows.  The arranged boards is denoted by ${\bf F}(h)$.
\begin{center}
    \begin{tikzpicture}[scale=0.5]
        \node[left] at (-2,2.5) {${\bf F}(h):$};
        \draw[thick] (0,0) rectangle (7,5);
        \draw[thick] (0,2)--(7,2);
        \draw[thick] (3,0)--(3,5);
        \node at (1.5,1) {$h+\frac{1}{2}$};
        \node at (5,3.5) {$h+\frac{1}{2}$};
        \node at (1.5,3.5) {$h+1$};
        \node at (5,1) {$h$};
        \node[left] at (0,1) {${\bf I}_{R_{\bar{0}}}$};
        \node[left] at (0,3.5) {$\gh{\bf I}_{R_{\bar{1}}}$};
        \node[above] at (1.5,5) {$\gh{\bf I}_{L_{\bar{1}}}$};
        \node[above] at (5,5) {${\bf I}_{L_{\bar{0}}}$};
    \end{tikzpicture}
\end{center}
 Similarly, one can draw ${\bf F}(h-k)$ for all $k\in\mathbb{Z}$. We can arrange all possible boards horizontally or vertically as follows. 
\[\cdots {\bf F}(h)\, {\bf F}(h-1)\, \cdots \quad\quad \begin{array}{c}
    \vdots\\
    {\bf F}(h)\\
    {\bf F}(h-1)\\
    \vdots
\end{array}\]
Denote them by ${\bf F}_{\rm hor}$ and ${\bf F}_{\rm ver}$ respectively. Now we remove simultaneously in ${\bf F}_{\rm hor}$ and ${\bf F}_{\rm ver}$ the cells indexed by $(i,p,\gamma)$ for which $\psi^{i,p}(\gamma)$ is not in $\W^{h,-}$. The resulting graph is denoted by $\overline{\bf F}_{\rm hor}$ and $\overline{\bf F}_{\rm ver}$, respectively. Also, the remaining part of ${\bf F}(h-k)$ is denoted by $\overline{\bf F}(h-k)$ for all $k\in\mathbb{Z}$.  For instance, set 
\[h=\frac{1}{2}, L_{\bar{1}}=R_{\bar{1}}=1, L_{\bar{0}}=7,R_{\bar{0}}=5.\]
The horizontal board $\overline{\bf F}_{\rm hor}$ is as follows. 
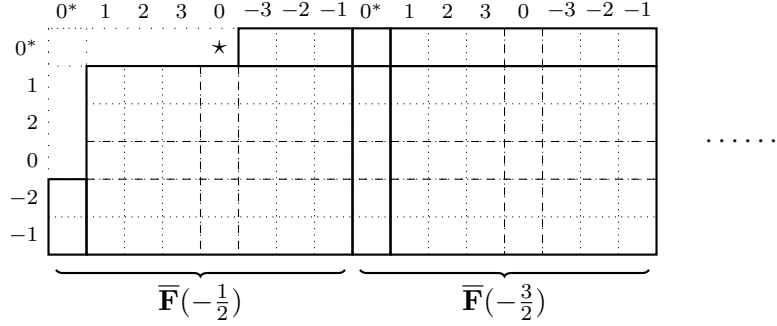
\begin{center}
\begin{tikzpicture}[scale=0.5]
    \begin{scope}
        \draw[dotted] (5,0) grid[step=1] (8,-1);
        \draw[thick] (5,0) rectangle (8,-1);
        \draw[dotted] (0,-4) grid[step=1] (1,-6);
        \draw[thick] (0,-4) rectangle (1,-6);
        \draw[dotted] (1,-1) grid[step=1] (8,-6);
        \draw[thick] (1,-1) rectangle (8,-6);
        \draw[loosely dotted] (0,0) rectangle (5,-1);
        \draw[loosely dotted] (0,0) rectangle (1,-4);
        \node[above] at (0.5,0) {$\scriptstyle \gh{0}$};
        \node[above] at (1.5,0) {$\scriptstyle {1}$};
        \node[above] at (2.5,0) {$\scriptstyle {2}$};
        \node[above] at (3.5,0) {$\scriptstyle {3}$};
        \node[above] at (4.5,0) {$\scriptstyle {0}$};
        \node[above] at (5.5,0) {$\scriptstyle {-3}$};
        \node[above] at (6.5,0) {$\scriptstyle {-2}$};
        \node[above] at (7.5,0) {$\scriptstyle {-1}$};
        \node[left] at (0,-0.5) {$\scriptstyle \gh{0}$};
        \node[left] at (0,-1.5) {$\scriptstyle {1}$};
        \node[left] at (0,-2.5) {$\scriptstyle {2}$};
        \node[left] at (0,-3.5) {$\scriptstyle {0}$};
        \node[left] at (0,-4.5) {$\scriptstyle {-2}$};
        \node[left] at (0,-5.5) {$\scriptstyle {-1}$};
        \node at (4.5,-0.5) {$\star$};
        \draw[thick, decoration={
                brace,
                mirror,
                raise=0.2cm
            }, decorate] (0.2,-6) -- (7.8,-6) node [pos=0.5,anchor=north,yshift=-0.25cm] {$\overline{\bf F}(-\frac{1}{2})$};
    \end{scope}
    \begin{scope}[shift={(8,0)}]
        \draw[dotted] (1,0) grid[step=1] (8,-1);
        \draw[thick] (0,0) rectangle (1,-1);
        \draw[thick] (1,0) rectangle (8,-1);
        \draw[dotted] (0,-1) grid[step=1] (1,-6);
        \draw[thick] (0,-1) rectangle (1,-6);
        \draw[dotted] (1,-1) grid[step=1] (8,-6);
        \draw[thick] (1,-1) rectangle (8,-6);
        \draw[thick, decoration={
                brace,
                mirror,
                raise=0.2cm
            }, decorate] (0.2,-6) -- (7.8,-6) node [pos=0.5,anchor=north,yshift=-0.25cm] {$\overline{\bf F}(-\frac{3}{2})$};
        \node[above] at (0.5,0) {$\scriptstyle \gh{0}$};
        \node[above] at (1.5,0) {$\scriptstyle {1}$};
        \node[above] at (2.5,0) {$\scriptstyle {2}$};
        \node[above] at (3.5,0) {$\scriptstyle {3}$};
        \node[above] at (4.5,0) {$\scriptstyle {0}$};
        \node[above] at (5.5,0) {$\scriptstyle {-3}$};
        \node[above] at (6.5,0) {$\scriptstyle {-2}$};
        \node[above] at (7.5,0) {$\scriptstyle {-1}$};
        \node[right, shift={(0.5,0)}] at (8,-3) {$\cdots\cdots$};
    \end{scope}    

    \draw[dashed] (4, -1) -- (4, -6);
    \draw[dashed] (5, -1) -- (5, -6);
    \draw[dashed] (12, 0) -- (12, -6);
    \draw[dashed] (13, 0) -- (13, -6);
    \draw[dashed] (1, -3) -- (16, -3);
    \draw[dashed] (1, -4) -- (16, -4);

\end{tikzpicture}
\end{center}
All cells in ${\bf F}(\frac{1}{2}-k)$ with $k<0$ and the cells with dotted frames in ${\bf F}(\frac{1}{2})$ are removed, since they correspond to indexes $(i,p,\gamma)$ such that $\psi^{i,p}(\gamma)\notin\W^{\frac{1}{2},-}$. 
For instance, the cell marked with $\star$ is removed since it corresponds to $\psi^{{0},\gh{0}}(0)$ which does not belong to $\W^{\frac{1}{2},-}$.

A horizontal (or vertical, resp.) {\bf Maya diagram} $\M$ of shape $\W^{h, -}$ is an assignment that coloring finitely many cells black in  $\overline{\bf F}_{\rm hor}$ (or $\overline{\bf F}_{\rm ver}$, resp.). 
 We may say an uncolored cell white.
To emphasize the shape of the board, we will sometimes denote by $\dot{\M}$ the vertical Maya diagram, which is identically the same as $\M$ via the natural correspondence between $\overline{\bf F}_{\rm hor}$ and $\overline{\bf F}_{\rm ver}$.

For each Maya diagram $\M$ over $\overline{\bf F}_{\rm hor}$ such that the black cells are $(i_1,p_1,\gamma_1)$, $(i_2,p_2,\gamma_2),\ldots$, $(i_n,p_n,\gamma_n)$, in a descending order with respect to the lexicographical order on ${\bf I}^h$, we can construct a vector 
\[v(\M)=\begin{cases}
    \psi^{-e}(0)\psi^{i_1, p_1}(\gamma_1) \cdots \psi^{i_n, p_n}(\gamma_n), &  L_{\bar{0}}L_{\bar{1}}R_{\bar{0}}R_{\bar{1}}n \mbox{ is odd}, h=0;\\
    \psi^{-\gh{e}}(0)\psi^{i_1, p_1}(\gamma_1) \cdots \psi^{i_n, p_n}(\gamma_n), &  L_{\bar{0}}L_{\bar{1}}R_{\bar{0}}R_{\bar{1}}n \mbox{ is odd}, h=\frac{1}{2};\\
    \psi^{i_1, p_1}(\gamma_1) \cdots \psi^{i_n, p_n}(\gamma_n), & \mbox{otherwise}.
\end{cases}\]
The set 
\[\{v(\M)\mid \M\mbox{ is a Maya diagram over }\overline{\bf F}_{\rm hor}\}\]
is then a monomial basis of the Fock space $\F$. Identically, we can define $v(\dot{\M})$ for each Maya diagram over $\overline{\bf F}_{\rm ver}$. In this way we can use Maya diagrams to display the monomial basis of $\F$ and to study, up to scalar, the action of Lie algebras graphically. For example, the operator $\nord{\psi^{i, p}(\gamma)\psi^{-j, -q}(-\eta)}$ with $\gamma, \eta < 0$ acts as moving the black cell at $(j, q, \eta)$ to position $(i, p, \gamma)$, 
or $0$ if the source cell is not black or the target cell is already black. 
While in the case both $\psi^{i, p}(\gamma)$ and $\psi^{-j, -q}(-\eta)$ being annihilation (resp. creation) operators, $\nord{\psi^{i, p}(\gamma)\psi^{-j, -q}(-\eta)}$ with $\gamma, \eta < 0$ acts as simultaneously adding (resp. removing) black cells at $(j, q, \eta)$ and $(i, p, \gamma)$.

Let us consider the actions of $(\lghat_{\lleft}, \lgdhat)$ on Maya diagrams. Obviously the action of $\lghat_{\lleft}$ goes horizontally while $\lgdhat$ goes vertically.
Here we discuss $\pi_{\lleft}: \lghat_{\lleft} \to \End (\F)$. Take (non-zero) $X_{ij}(n)\in \lghat_{\lleft}$, where $X = A, B, C, D$ depends on type of $\lghat_{\lleft}$, and a Maya diagram $\M$, the element $X_{ij}(n).v(\M)$ is a linear combination of monomials $v(\M')$, where $\M'$ is a Maya diagram obtained in one of the following patterns. 
\begin{itemize}
    \item Move a black cell $(j,p,\gamma)$ in $\M$ to the white cell $(i, p, n+\gamma)$ for some $p,\gamma$; 
    \item Move a black cell $(-i,p,\gamma)$ in $\M$ to the white cell $(-j, p, n+\gamma)$ for some $p,\gamma$; 
    \item  Suppose $X_{ij}(n)\in \lghat_{\lleft}{}^-$. Reverse the color of two white cells $(i, p, n-\gamma)$ and $(-j, -p, \gamma)$ in $\M$ for some $p,\gamma$; 
    \item Suppose $X_{ij}(n)\in \lghat_{\lleft}{}^+$. Reverse the color of two black cells
     $(-i, -p, \gamma-n)$ and $(j, p, -\gamma)$ in $\M$ for some $p,\gamma$.
\end{itemize}
The case for $\dot{\pi}$ goes in a similar pattern. Drawing the Maya diagram horizontally(resp. vertically), we find that the action of $\lghat_{\lleft}{}^+$(resp. $\lgdhat{}^+$) carries black cells to leftward(resp. upward) white cells and cancels pair of black cells appearing in rows(columns) with opposite label $\pm p$(resp. $\pm i$).
This observation gives us an intuitive understanding of $(\lghat_{\lleft}, \lgdhat)$ joint highest weight vectors. 

\begin{Exam}
    Here we introduce the Maya diagram corresponding to a Young diagram, which will be frequently used in the following sections.
    Let $Y = (y_1,y_2,\ldots,y_l) \in \P^l_{\infty}$ a partition with at most $l$ parts. Draw the following Maya diagrams on $\overline{\bf{F}}_{\rm ver}$:
    \begin{itemize}
        \item For each column labelled by $i = 1,2,\ldots,l$, fill $y_i$ many cells at the top of the column. Denote this Maya diagram by $\M^+(Y)$;
        \item For each column labelled by $i = -l, -l+1, \ldots, -1$, fill $y_{i+l+1}$ many cells at the top of the column. Denote this Maya diagram by $\M^-(Y)$.
    \end{itemize}

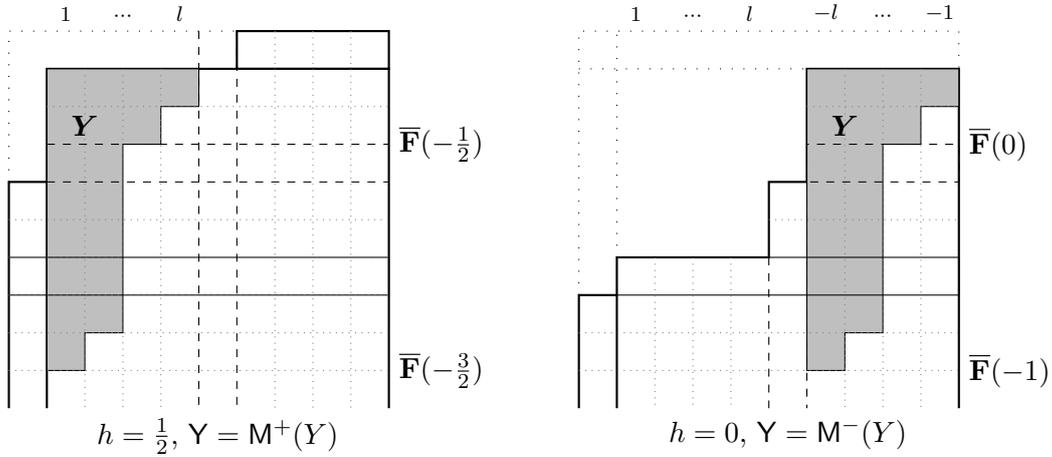
    \begin{center}
        \begin{tikzpicture}[scale=0.5]
        \begin{scope}
            \draw[thick] (0,-9) -- (0,0) -- (9,0) -- (9,-9);
            \draw[thick] (5,0) rectangle (9,1);
            \draw[thick] (0,-3) -- (-1,-3) -- (-1,-9);
            \draw[loosely dotted] (-1,-3) -- (-1,1) -- (5,1);
            \draw[fill=lightgray] (0, 0) 
            -- (0, -8) -- (1, -8)
            -- (1, -7) -- (2, -7)
            -- (2, -2) -- (3, -2)
            -- (3, -1) -- (4, -1)
            -- (4, 0) -- (0, 0) -- cycle;
            
            \draw[dashed] (0, -2) -- (9, -2);
            \draw[dashed] (0, -3) -- (9, -3);
            \draw (-1,-5) rectangle (9,-6);
            \draw[dashed] (4, 1) -- (4, -9);
            \draw[dashed] (5, 1) -- (5, -9);

            \node[above] at (0.5,1) {$\scriptstyle {1}$};
            \node[above] at (2,1) {$\scriptstyle {\cdots}$};
            \node[above] at (3.5,1) {$\scriptstyle {l}$};
            \node[below] at (4.5, -9) {$h = \frac{1}{2}$, $\Y = \M^+(Y)$};
            \foreach \x in {1,2,...,3}
                \draw[gray, dotted] (\x,0) -- (\x,-9);
            \foreach \x in {6,7,...,8}
                \draw[gray, dotted] (\x,1) -- (\x,-9);
            \draw[gray, dotted] (0,-1) -- (9,-1);
            \foreach \y in {4,5,...,8}
                \draw[gray, dotted] (-1,-\y) -- (9,-\y);
            \node[font=\bfseries\boldmath] at (1, -1.5) {$Y$};
            \node[right] at (9, -2) {$\overline{\bf F}(-\frac{1}{2})$};
            \node[right] at (9, -8) {$\overline{\bf F}(-\frac{3}{2})$};
        \end{scope}
        \begin{scope}[shift={(15,0)}]
            \draw[thick] (0,-9) -- (0,-5) -- (4,-5) -- (4,-3) -- (5,-3) -- (5,0) -- (9,0) -- (9,-9);
            \draw[thick] (0,-6) -- (-1,-6) -- (-1,-9);
            \draw[loosely dotted] (-1,-6) -- (-1,1) -- (9,1) -- (9,0);
            \draw[loosely dotted] (0,-5) -- (0,1);
            \draw[loosely dotted] (-1,0) -- (5,0);
            \draw[fill=lightgray] (5, 0) 
            -- (5, -8) -- (6, -8)
            -- (6, -7) -- (7, -7)
            -- (7, -2) -- (8, -2)
            -- (8, -1) -- (9, -1)
            -- (9, 0) -- (5, 0) -- cycle;
            \node[font=\bfseries\boldmath] at (6, -1.5) {$Y$};

            \draw (0,-5) rectangle (9,-6);
            \draw[dashed] (4, -5) -- (4, -9);
            \draw[dashed] (5, -3) -- (5, -9);
            \draw[dashed] (5, -3) rectangle (9, -2);
            \node[above] at (0.5,1) {$\scriptstyle {1}$};
            \node[above] at (2,1) {$\scriptstyle {\cdots}$};
            \node[above] at (3.5,1) {$\scriptstyle {l}$};
            \node[above] at (5.5,1) {$\scriptstyle {-l}$};
            \node[above] at (7,1) {$\scriptstyle {\cdots}$};
            \node[above] at (8.5,1) {$\scriptstyle {-1}$};
            \node[below] at (4.5, -9) {$h = 0$, $\Y = \M^-(Y)$};

            \foreach \x in {1,2,...,3}
                \draw[gray, dotted] (\x,-5) -- (\x,-9);
            \foreach \x in {6,7,...,8}
                \draw[gray, dotted] (\x,0) -- (\x,-9);
            \draw[gray, dotted] (5,-1) -- (9,-1);
            \draw[gray, dotted] (4,-4) -- (9,-4);
            \foreach \y in {7,8}
                \draw[gray, dotted] (-1,-\y) -- (9,-\y);
            \node[right] at (9, -2) {$\overline{\bf F}(0)$};
            \node[right] at (9, -8) {$\overline{\bf F}(-1)$};
        \end{scope}    
        \end{tikzpicture}
    \end{center}

    We uniformly write $\Y = \M^+(Y)$ when $h = \frac{1}{2}$ or $\Y = \M^-(Y)$ when $h = 0$. In particular, if $Y\in \P^l_r$, then graphically the Maya diagram $\Y$ lies in the top-left corner of the board.
    It is routine to check that none of the above $\lghat{}^+$ and $\lgdhat{}^+$ operation on $\Y$ can be done and the monomial $v(\Y)$ is a $(\lghat_{\lleft}, \lgdhat)$ joint highest weight vector. 
    In fact, all $(\lghat_{\lleft}, \lgdhat)$ joint highest weight vectors can be obtained in this way up to scalar, which is crucial in the proof of Theorem \ref{mainDualThmIntro}. 
\end{Exam}

It is clear that $v(\M)$ is a weight vector for both $\lghat$ and $\lgdhat$. The question is how to read the  weights of $v(\M)$ from the Maya diagram $\M$. 
In the spirit of Remark \ref{ortho-reduction}, from now on we assume that $L_{\bar{1}}, R_{\bar{1}} \le 1$, and set $l = \lfloor \frac{L_{\bar{0}}}{2}\rfloor$, $r = \lfloor \frac{R_{\bar{0}}}{2}\rfloor$.
Write the standard bases of $\dlho_{\lleft}$ and $\mathring{\liehd}{}^*$ by $\epsilon_1, \ldots, \epsilon_l$ and $\dot{\epsilon}_1, \ldots, \dot{\epsilon}_r$, respectively.
Given a Maya diagram $\M$, the  weights $(\lambda_{\M}, \dot{\lambda}_{\M}) \in \dlho_{\lleft} \oplus \mathring{\liehd}{}^*$ of $v(\M)$ can be read in the following way. 
\begin{itemize}
    \item In the corresponding vertical Maya diagram $\dot{\M}$ of $\M$, for $i$-th column, define 
    \[b_i=\#\{(p,\gamma)\mid \mbox{The cell } (i,p,\gamma)\mbox{ in }\dot{\M} \mbox{ is black}\}.\]
    Then the horizontal weight of $v(\M)$ is given by 
     \[\lambda_{\M} = \sum\limits_{i = 1}^l \left(b_i - b_{-i} + \frac{R_{\overline{2h}}}{2}\right) \epsilon_i.\]
    \item In $\M$, for $p$-th row, define 
     \[\dot{b}_p=\#\{(i,\gamma)\mid \mbox{The cell } (i,p,\gamma)\mbox{ in } \M \mbox{ is black}\}.\]
    Then the vertical weight of $v(\M)$ is given by 
    \[\dot{\lambda}_{\M} = \sum\limits_{p = 1}^r \left(\dot{b}_p - \dot{b}_{-p} + \frac{L_{\overline{2h}}\pmod 2}{2} \right) \dot{\epsilon}_p.\] 
\end{itemize}
It is remarkable here that from the definitions we have each pair of $b_i, \dot{b}_p \in \mathbb{Z}$. In particular, in a fixed Fock space, every monomial basis vector $v(\M)$ is of weight $\lambda_{\M}$ with coordinates lying in $\frac{R_{\overline{2h}}}{2}\mathbf{1} + \mathbb{Z}^l$, and similarly for $\dot{\lambda}_{\M}$.
These facts are helpful when we consider all possible dominant weights in $\F$. 
For the diagram $\Y$ corresponding to $Y\in\P^l_r$, we can read the coordinates of the weights as follows:

\begin{equation}
(\lambda_{\Y}, \dot{\lambda}_{\Y}) = 
\begin{cases}
    (Y^c + \frac{R_{\bar{0}}\pmod 2}{2}\mathbf{1}, Y^t + \frac{L_{\bar{0}}\pmod 2}{2}\mathbf{1}),& h = 0;\\
    (Y + \frac{R_{\bar{1}}}{2}\mathbf{1}, Y^t +  \frac{L_{\bar{1}}}{2}\mathbf{1}),& h = \frac{1}{2}.
\end{cases}
\tag{4.1.1}\label{weightY}
\end{equation}
One can easily see that the above weights are always dominant. 
In fact, recall the parity decomposition given in \S \ref{Parity}, we have $\lambda_{\Y} \in \sigdom(\lLevel)_{\bar{n}_h, \bar{n}_t}$, where $n_h \equiv_2 R_{\overline{2h'}}$, $n_t \equiv_2 R_{\overline{2h}}$, and similarly for $\dot{\lambda}_{\Y}$.

To recover the weight information of $\lgtilde_{\lleft}$ and $\lgdtilde$, we need to consider eigenvalue of $D^{\lghat}(0)$ and $D^{\lgdhat}(0)$. 
Unfortunately, the monomial basis elements $v(\M)$ are not always eigenvectors of these operators.
Therefore instead of studying $D^{\lghat}(0)$ and $D^{\lgdhat}(0)$ directly, we can recover the degree of $v(\M)$ by ${\sf d}_{v(\M)}$, the eigenvalue under $D^{\Cl}(0)$, thanks to Proposition \ref{VirEv}.
Define the weight lattice of $\F$ by
\[
P(\F):= \{ (\lambda, \dot{\mu}; n) \in \dlho \times \mathring{\liehd}{}^* \times \mathbb{C}\},
\]
with weight spaces 
$$
\F_{(\lambda, \dot{\mu}; n)} = \{v\in \F \mid -D^{\Cl}(0).v = nv\text{ and }  v \text{ is of weight } (\lambda, \dot{\mu}) \text{ with respect to } \lieg_{\lleft} \oplus \liegd \}.
$$
For $v\in \F_{(\lambda, \dot{\mu}; n)}$, we say the weight of $v$ in $\F$ is the triple $(\lambda, \dot{\mu}; n)$.
Certainly a monomial basis vector $v(\M)$ is of weight $(\lambda_{\M}, \dot{\lambda}_{\M}; -{\sf d}_{v(\M)})$. 
In particular, if $\M = \Y$ is given by a Young diagram, then $v({\Y})$ is of weight $(\lambda_{\Y}, \dot{\lambda}_{\Y}; -{\sf d}_{\Y})$, with 
\[{\sf d}_{\Y} = h\cdot \size(Y)+ \frac{1}{16}\dim \W^h(0).\]
The weight information of $\lgtilde_{\lleft}$ and $\lgdtilde$ can be partially read via the surjection:
\[\begin{array}{rcl}
    \tilde{\lieh}{}^* \times \lhdtilde{}^* &\to& \dlho_{\lleft} \times \mathring{\liehd}{}^* \times \mathbb{C} \\
    (\lambda + \levelR \Lambda_0 + k\delta,\; \dot{\mu} + \levelL \dot{\Lambda}_0 + \dot{k} \dot{\delta}) &\mapsto & (\lambda, \dot{\mu};\; k + \dot{k}).
\end{array}\]
Let us remind that in type $D^{(2)}$ the indices are $(2,2,\ldots,2)$ and $\delta=2(\sum_{i}\alpha_i)$ in our setting, see \S \ref{Parity}. In practice, if we  deal with horizontal and vertical moves step by step, leaving either $k$ or $\dot{k}$ unchanged in each step, we are then able to get the coefficient of $\delta$ or $\dot{\delta}$ by the difference of ${\sf d}_{v(\M)}$. 

Note that on $\tilde{\lieh}{}^*_{\lleft} \times \lhdtilde{}^*$ we can consider the action of the Weyl group $\mathcal{W}\times \dot{\mathcal{W}}$ as well as the group of diagram automorphisms $\Sigma \times \dot{\Sigma}$.
Since $\delta$ and $\dot{\delta}$ are invariant under the action of $\mathcal{W}\times \dot{\mathcal{W}}$, pushing back through the above surjection gives a well-defined action of $\mathcal{W}\times \dot{\mathcal{W}}$ on $\dlho_{\lleft} \times \mathring{\liehd}{}^* \times \mathbb{C}$. 
Similarly, we can also define the action of $\Sigma \times \dot{\Sigma}$ on $\dlho_{\lleft} \times \mathring{\liehd}{}^* \times \mathbb{C}$. In particular, for any $w \in \mathcal{W}, \sigma \in \Sigma$, by the property of weight spaces we have $\F_{w(\lambda, \dot{\mu}; n)} \cong \F_{(\lambda, \dot{\mu}; n)} \cong \F_{\sigma(\lambda, \dot{\mu}; n)}$, and similarly for $\dot{\mathcal{W}}, \dot{\Sigma}$.

\begin{Exam}
    Set $\lghat_{\lleft} = \lsohat(8)$, $\lgdhat = \lsohat(7)$ and $h = 0$. In this case $L_{\bar{1}} = R_{\bar{1}} = 0, L_{\bar{0}} = 8, R_{\bar{0}} = 7, l = 4$, and $r = 3$.
    Take  $Y = (2, 1, 1, 0) \in \P^4_3$, and draw it in the top-left $4 \times 3$-rectangle of $\overline{\bf F}_{\rm hor}$: 

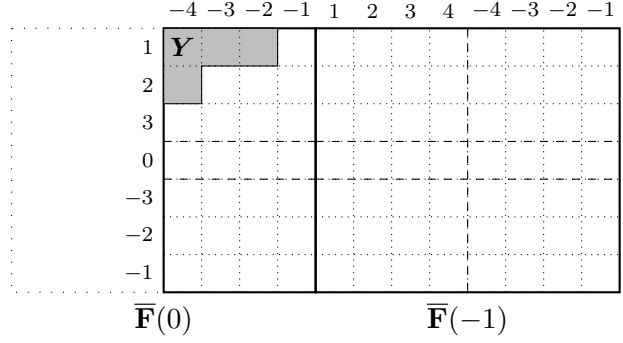
\begin{center}
\begin{tikzpicture}[scale=0.5]
    \draw[thick] (0,0) rectangle (4,-7);
    \draw[fill=lightgray] (0, 0)
    -- (0, -2) -- (1, -2)
    -- (1, -1) -- (3, -1)
    -- (3, 0) -- (0, 0) -- cycle;
    \node[font=\small\bfseries\boldmath] at (0.5, -0.5) {$Y$};
    \draw[dashed] (0, -3) -- (4, -3);
    \draw[dashed] (0, -4) -- (4, -4);
    \node[above] at (0.5,0) {$\scriptstyle {-4}$};
    \node[above] at (1.5,0) {$\scriptstyle {-3}$};
    \node[above] at (2.5,0) {$\scriptstyle {-2}$};
    \node[above] at (3.5,0) {$\scriptstyle {-1}$};
    \node[left] at (0,-0.5) {$\scriptstyle {1}$};
    \node[left] at (0,-1.5) {$\scriptstyle {2}$};
    \node[left] at (0,-2.5) {$\scriptstyle {3}$};
    \node[left] at (0,-3.5) {$\scriptstyle {0}$};
    \node[left] at (0,-6.5) {$\scriptstyle {-1}$};
    \node[left] at (0,-5.5) {$\scriptstyle {-2}$};
    \node[left] at (0,-4.5) {$\scriptstyle {-3}$};
    \node[below] at (0, -7) {$\overline{\bf F}(0)$};

    \draw[thick] (4,0) rectangle (12,-7);
    \draw[dashed] (4, -3) -- (12, -3);
    \draw[dashed] (4, -4) -- (12, -4);
    \node[above] at (4.5,0) {$\scriptstyle {1}$};
    \node[above] at (5.5,0) {$\scriptstyle {2}$};
    \node[above] at (6.5,0) {$\scriptstyle {3}$};
    \node[above] at (7.5,0) {$\scriptstyle {4}$};
    \draw[dashed] (8, 0) -- (8, -7);
    \node[above] at (8.5,0) {$\scriptstyle {-4}$};
    \node[above] at (9.5,0) {$\scriptstyle {-3}$};
    \node[above] at (10.5,0) {$\scriptstyle {-2}$};
    \node[above] at (11.5,0) {$\scriptstyle {-1}$};
    \node[below] at (8, -7) {$\overline{\bf F}(-1)$};
    \draw[dotted] (4,0) grid[step=1] (12,-7);
    \draw[dotted] (0,0) grid[step=1] (4,-7);
    \draw[loosely dotted] (-4, 0) rectangle (0,-7);
\end{tikzpicture}
\end{center}
The resulting Maya diagram is $\Y = \M^-(Y)$. Recall the convention that we identify a weight with its coordinates vector with respect to the basis $\epsilon_i$ or $\dot{\epsilon}_i$. 
The pair of classical weights of this diagram can be read as $\lambda_{\Y}=\left(\frac{7}{2}, \frac{5}{2}, \frac{5}{2}, \frac{3}{2}\right) = Y^c + \frac{1}{2}\mathbf{1}$ and $\dot{\lambda}_{\Y} = (3, 1, 0) = Y^t$. 
They correspond to the affine weights $\Lambda = \lambda_{\Y} + 7 \Lambda_0 = 4 \Lambda_4 + \Lambda_3 + \Lambda_1$, $\dot{\Lambda} = \dot{\lambda}_{\Y} + 8 \dot{\Lambda}_0 = 4\dot{\Lambda}_0 + 2\dot{\Lambda}_1 + \dot{\Lambda}_2$.
The monomial basis vector $v(\Y)$   is of weights $\left(\left(\frac{7}{2}, \frac{5}{2}, \frac{5}{2}, \frac{3}{2}\right), (3, 1, 0); -\frac{7}{2}\right)$, and is a joint highest weight vector.

Now consider the vector 
\[v=\psi^{-2, 1}(0)\psi^{-3, 2}(0)\psi^{-4, -2}(0)\psi^{-4, 2}(0)\psi^{-4, 1}(0)\psi^{1, 1}(-1) \in \F,\] 
corresponding to the following Maya diagram $\M(v)$.
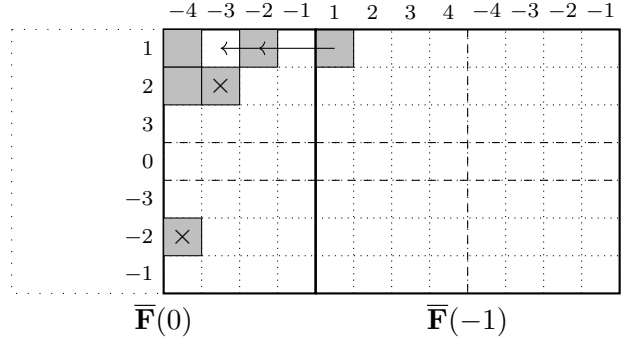
\begin{center}
\begin{tikzpicture}[scale=0.5]
    \draw[thick] (0,0) rectangle (4,-7);
    \draw[dotted] (0,0) grid[step=1] (4,-7);
    \draw[loosely dotted] (-4, 0) rectangle (0,-7);

    \draw[fill=lightgray] (0,0) rectangle (1,-1);
    \draw[fill=lightgray] (0,-1) rectangle (1,-2);
    
    \draw[fill=lightgray] (2,0) rectangle (3,-1);
    \draw[->] (2.5,-0.5) -- (1.5,-0.5);
    \draw[fill=lightgray] (4,0) rectangle (5,-1);
    \draw[->] (4.5,-0.5) -- (2.5,-0.5);
    
    \draw[fill=lightgray] (0,-5) rectangle (1,-6);
    \node at (0.5, -5.5) {$\times$};
    \draw[fill=lightgray] (1,-1) rectangle (2,-2);
    \node at (1.5, -1.5) {$\times$};

    \draw[dashed] (0, -3) -- (4, -3);
    \draw[dashed] (0, -4) -- (4, -4);
    \node[above] at (0.5,0) {$\scriptstyle {-4}$};
    \node[above] at (1.5,0) {$\scriptstyle {-3}$};
    \node[above] at (2.5,0) {$\scriptstyle {-2}$};
    \node[above] at (3.5,0) {$\scriptstyle {-1}$};
    \node[left] at (0,-0.5) {$\scriptstyle {1}$};
    \node[left] at (0,-1.5) {$\scriptstyle {2}$};
    \node[left] at (0,-2.5) {$\scriptstyle {3}$};
    \node[left] at (0,-3.5) {$\scriptstyle {0}$};
    \node[left] at (0,-6.5) {$\scriptstyle {-1}$};
    \node[left] at (0,-5.5) {$\scriptstyle {-2}$};
    \node[left] at (0,-4.5) {$\scriptstyle {-3}$};
    \node[below] at (0, -7) {$\overline{\bf F}(0)$};

    \draw[thick] (4,0) rectangle (12,-7);
    \draw[dotted] (4,0) grid[step=1] (12,-7);
    \draw[dashed] (4, -3) -- (12, -3);
    \draw[dashed] (4, -4) -- (12, -4);
    \node[above] at (4.5,0) {$\scriptstyle {1}$};
    \node[above] at (5.5,0) {$\scriptstyle {2}$};
    \node[above] at (6.5,0) {$\scriptstyle {3}$};
    \node[above] at (7.5,0) {$\scriptstyle {4}$};
    \draw[dashed] (8, 0) -- (8, -7);
    \node[above] at (8.5,0) {$\scriptstyle {-4}$};
    \node[above] at (9.5,0) {$\scriptstyle {-3}$};
    \node[above] at (10.5,0) {$\scriptstyle {-2}$};
    \node[above] at (11.5,0) {$\scriptstyle {-1}$};
    \node[below] at (8, -7) {$\overline{\bf F}(-1)$};
\end{tikzpicture}
\end{center}
The pair of classical weights of $v$ can be read as 
$$\lambda_v = (1, -1, -1, -3) + \frac{7}{2}\mathbf{1} = \left(\frac{9}{2}, \frac{5}{2}, \frac{5}{2}, \frac{1}{2}\right),$$
 $\dot{\lambda}_v = (3, 1, 0)$ and ${\sf d}_v = -(-1)+\frac{56}{16} = \frac{9}{2}$.
    Here $v$ is not a highest weight vector, since we can apply operators in $\lghat_{\lleft}^+$ to push the black cells leftwards.
    Applying $e_2 = D_{2,3}(0)$ moves the black cell at $(-2, 1, 0)$ to $(-3, 1, 0)$, while subsequently applying $e_0 = D_{-1,2}(1)$ moves the black cell at $(1, 1, -1)$ to $(-2, 1, 0)$.
    Since $e_4 = D_{3,-4}(0)$, applying it afterwards erases a pair of black cells at $(-3, 2, 0)$ and $(-4, -2, 0)$.
    After successive action of these three operators we bring the diagram $\M$ to $\Y$, which means up to a scalar we have $e_4 e_0 e_2. v = v(\Y)$.
    The change of weight writes $\alpha_0 + \alpha_2 + \alpha_4 = \left(-1, 0, 0, 1\right) + \delta$,
    corresponding to the difference $\left(\left(\frac{7}{2}, \frac{5}{2}, \frac{5}{2}, \frac{3}{2}\right), (3, 1, 0); -\frac{7}{2}\right) - \left(\left(\frac{9}{2}, \frac{5}{2}, \frac{5}{2}, \frac{1}{2}\right), (3, 1, 0); -\frac{9}{2}\right) = (\left(-1, 0, 0, 1\right), 0; 1)$.
\end{Exam}

Finally we return to the $(A^{(1)},A^{(1)})$ case, in which $\bar{\W}= (\W^+_{2l} \otimes \W^+_{2r}) \oplus (\W^-_{2l} \otimes \W^-_{2r})$.
Embedding this space into $\W_{2l} \otimes \W_{2r}$, we can draw the board corresponding to $\bar{\W}$ by removing cells $(i, p)$ with $ip < 0$.
Then we are able to routinely draw Maya diagrams as well as reading the weights  in other cases, with the convention $L_{\bar{0}} = l, R_{\bar{0}} = r$ in the weight formula. 
In this case, for the diagram $\Y$ corresponding to $Y\in \P^l_r$, the weights are as follows:
\begin{equation}
(\lambda_{\Y}, \dot{\lambda}_{\Y}) = 
\begin{cases} 
    (Y^c - \frac{r}{2}\mathbf{1}, [Y^\dagger]),& h = 0;  \\ 
    (Y,[Y^t]),& h = \frac{1}{2}. 
\end{cases}
\tag{4.1.2}\label{weightY-A}
\end{equation}
As an example, for $(\lghat_{\lleft}, \lgdhat) = (\lglhat(4), \lslhat(3))$ and $h = \frac{1}{2}$, the action of $e_0 = E_{41}(1) \in \lglhat(4)$ on $\psi^{1,1}(-\frac{3}{2})$ displaying on the horizontal board $\overline{\bf F}_{\rm hor}$ is of the form:
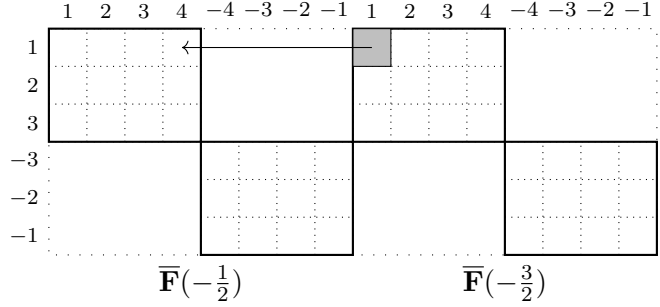
\begin{center}
\begin{tikzpicture}[scale=0.5]
    \draw[thick] (0,0) rectangle (4,-3);
    \draw[dotted] (0,0) grid[step=1] (4,-3);
    \draw[thick] (8,-6) rectangle (4,-3);
    \draw[dotted] (8,-6) grid[step=1] (4,-3);

    \draw[loosely dotted] (0, 0) rectangle (16,-6);

    \node[above] at (0.5,0) {$\scriptstyle {1}$};
    \node[above] at (1.5,0) {$\scriptstyle {2}$};
    \node[above] at (2.5,0) {$\scriptstyle {3}$};
    \node[above] at (3.5,0) {$\scriptstyle {4}$};
    \node[above] at (4.5,0) {$\scriptstyle {-4}$};
    \node[above] at (5.5,0) {$\scriptstyle {-3}$};
    \node[above] at (6.5,0) {$\scriptstyle {-2}$};
    \node[above] at (7.5,0) {$\scriptstyle {-1}$};

    \node[left] at (0,-0.5) {$\scriptstyle {1}$};
    \node[left] at (0,-1.5) {$\scriptstyle {2}$};
    \node[left] at (0,-2.5) {$\scriptstyle {3}$};
    \node[left] at (0,-5.5) {$\scriptstyle {-1}$};
    \node[left] at (0,-4.5) {$\scriptstyle {-2}$};
    \node[left] at (0,-3.5) {$\scriptstyle {-3}$};
    \node[below] at (4, -6) {$\overline{\bf F}(-\frac{1}{2})$};

    \draw[thick] (8,0) rectangle (12,-3);
    \draw[dotted] (8,0) grid[step=1] (12,-3);
    \draw[thick] (16,-6) rectangle (12,-3);
    \draw[dotted] (16,-6) grid[step=1] (12,-3);
    
    \node[above] at (8.5,0) {$\scriptstyle {1}$};
    \node[above] at (9.5,0) {$\scriptstyle {2}$};
    \node[above] at (10.5,0) {$\scriptstyle {3}$};
    \node[above] at (11.5,0) {$\scriptstyle {4}$};
    \node[above] at (12.5,0) {$\scriptstyle {-4}$};
    \node[above] at (13.5,0) {$\scriptstyle {-3}$};
    \node[above] at (14.5,0) {$\scriptstyle {-2}$};
    \node[above] at (15.5,0) {$\scriptstyle {-1}$};
    \node[below] at (12, -6) {$\overline{\bf F}(-\frac{3}{2})$};

    \draw[fill=lightgray] (8,0) rectangle (9,-1);
    \draw[->] (8.5,-0.5) -- (3.5,-0.5);
\end{tikzpicture}
\end{center}

\subsection{Abacus and Uglov map}\label{Uglov-map}

From now on we fix a dual pair $(\lghat_{\lleft}, \lgdhat)$ and $h \in \{0, \frac{1}{2}\}$, and assume that $L_{\bar{1}}, R_{\bar{1}} \le 1$, $l=\lfloor\frac{L_{\bar{0}}}{2}\rfloor$, $r=\lfloor\frac{R_{\bar{0}}}{2}\rfloor$. 
Starting from a horizontal Maya diagram $\M$, the corresponding row abacus ${\B}(\M)$ is the graph after applying the following successive operations on $\M$:
\begin{enumerate}[label={(\arabic{*})}]
    \item Rotate the rows with negative labels $-p = -1, \ldots, -r$ toward left, line up and piece together the $-p$ one with the row labelled by $p$;
    \item Reverse the color of all cells coming from negative rows;
    \item Replace all black cells by beads  and all white cells by empty positions. 
\end{enumerate}

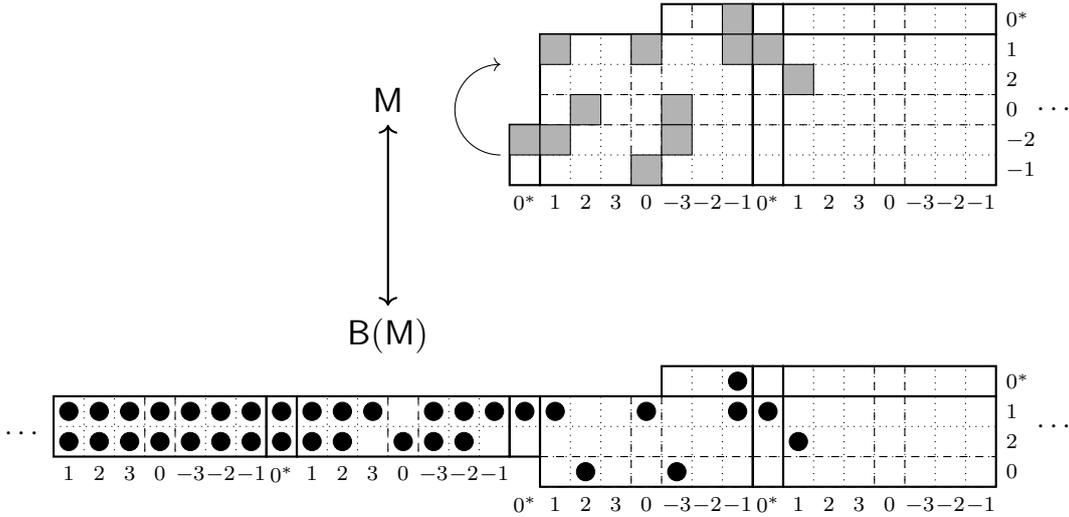
\begin{center}
    \begin{tikzpicture}[scale=0.4]

    \begin{scope}[shift={(15,-18)}]
    \draw[->] (-1.3,-4) arc (270:90:1.5);
    \draw[dotted] (0,0) grid[step=1] (7,-5);
    \draw[thick] (0,0) rectangle (7,-5);
    \draw[dashed] (3,0) --++(0,-5);
    \draw[dashed] (4,0) --++(0,-5);
    \draw[dashed] (0,-2) --++(7,0);
    \draw[dashed] (0,-3) --++(7,0);
    \draw[dotted] (-1,-3) grid[step=1] (0,-5);
    \draw[thick] (-1,-3) rectangle (0,-5);
    \draw[dotted] (4,1) grid[step=1] (7,0);
    \draw[thick] (4,1) rectangle (7,0);
    \draw[dashed] (5,1) --++(0,-1);
    \draw[dashed] (6,1) --++(0,-1);
    \draw[fill=black!30] (6,1) rectangle (7,0);
    \draw[fill=black!30] (0,0) rectangle (1,-1);
    \draw[fill=black!30] (3,0) rectangle (4,-1);
    \draw[fill=black!30] (6,0) rectangle (7,-1);
    \draw[fill=black!30] (1,-2) rectangle (2,-3);
    \draw[fill=black!30] (4,-2) rectangle (5,-3);
    \draw[fill=black!30] (-1,-3) rectangle (0,-4);
    \draw[fill=black!30] (0,-3) rectangle (1,-4);
    \draw[fill=black!30] (4,-3) rectangle (5,-4);
    \draw[fill=black!30] (3,-4) rectangle (4,-5);
    \node[below] at (-0.5,-5) {${\scriptstyle 0^*}$};
    \node[below] at (0.5,-5) {${\scriptstyle 1}$};
    \node[below] at (1.5,-5) {${\scriptstyle 2}$};
    \node[below] at (2.5,-5) {${\scriptstyle 3}$};
    \node[below] at (3.5,-5) {${\scriptstyle 0}$};
    \node[below] at (4.5,-5) {${\scriptstyle -3}$};
    \node[below] at (5.5,-5) {${\scriptstyle -2}$};
    \node[below] at (6.5,-5) {${\scriptstyle -1}$};
    \end{scope}
    \begin{scope}[shift={(23,-18)}]
    \draw[dotted] (0,0) grid[step=1] (7,-5);
    \draw[thick] (0,0) rectangle (7,-5);
    \draw[dashed] (3,0) --++(0,-5);
    \draw[dashed] (4,0) --++(0,-5);
    \draw[dashed] (0,-2) --++(7,0);
    \draw[dashed] (0,-3) --++(7,0);
    \draw[dotted] (-1,1) grid[step=1] (0,0);
    \draw[thick] (-1,1) rectangle (0,0);
    \draw[dotted] (-1,0) grid[step=1] (0,-5);
    \draw[thick] (-1,0) rectangle (0,-5);
    \draw[dashed] (-1,-2) --++(1,0);
    \draw[dashed] (-1,-3) --++(1,0);
    \draw[dotted] (0,1) grid[step=1] (7,0);
    \draw[thick] (0,1) rectangle (7,0);
    \draw[dashed] (3,1) --++(0,-1);
    \draw[dashed] (4,1) --++(0,-1);
    \draw[fill=black!30] (-1,0) rectangle (0,-1);
    \draw[fill=black!30] (0,-1) rectangle (1,-2);
    \node[below] at (-0.5,-5) {${\scriptstyle 0^*}$};
    \node[below] at (0.5,-5) {${\scriptstyle 1}$};
    \node[below] at (1.5,-5) {${\scriptstyle 2}$};
    \node[below] at (2.5,-5) {${\scriptstyle 3}$};
    \node[below] at (3.5,-5) {${\scriptstyle 0}$};
    \node[below] at (4.5,-5) {${\scriptstyle -3}$};
    \node[below] at (5.5,-5) {${\scriptstyle -2}$};
    \node[below] at (6.5,-5) {${\scriptstyle -1}$};
    \node[right] at (7,0.5) {${\scriptstyle 0^*}$};
    \node[right] at (7,-0.5) {${\scriptstyle 1}$};
    \node[right] at (7,-1.5) {${\scriptstyle 2}$};
    \node[right] at (7,-2.5) {${\scriptstyle 0}$};
    \node[right] at (7,-3.5) {${\scriptstyle -2}$};
    \node[right] at (7,-4.5) {${\scriptstyle -1}$};
    \node[right] at (8,-2.5){$\cdots$};
    \end{scope}

    \draw[<->, thick] (10, -21) node[above, scale=1.2] {$\M$} -- (10, -27) node[below,scale=1.2]{$\B(\M)$};

    \begin{scope}[shift={(15,-30)}]
    \draw[dotted] (0,0) grid[step=1] (7,-3);
    \draw[thick] (0,0) rectangle (7,-3);
    \draw[dashed] (3,0) --++(0,-3);
    \draw[dashed] (4,0) --++(0,-3);
    \draw[dashed] (0,-2) --++(7,0);
    \draw[dashed] (0,-3) --++(7,0);
    \draw[dotted] (4,1) grid[step=1] (7,0);
    \draw[thick] (4,1) rectangle (7,0);
    \node[below] at (-0.5,-3) {${\scriptstyle 0^*}$};
    \node[below] at (0.5,-3) {${\scriptstyle 1}$};
    \node[below] at (1.5,-3) {${\scriptstyle 2}$};
    \node[below] at (2.5,-3) {${\scriptstyle 3}$};
    \node[below] at (3.5,-3) {${\scriptstyle 0}$};
    \node[below] at (4.5,-3) {${\scriptstyle -3}$};
    \node[below] at (5.5,-3) {${\scriptstyle -2}$};
    \node[below] at (6.5,-3) {${\scriptstyle -1}$};
    \draw[fill=black] (6.5,0.5) circle (0.3);
    \draw[fill=black] (0.5,-0.5) circle (0.3);
    \draw[fill=black] (3.5,-0.5) circle (0.3);
    \draw[fill=black] (6.5,-0.5) circle (0.3);
    \draw[fill=black] (1.5,-2.5) circle (0.3);
    \draw[fill=black] (4.5,-2.5) circle (0.3);
    \end{scope}
    \begin{scope}[shift={(7,-30)}]
    \draw[dotted] (0,0) grid[step=1] (7,-2);
    \draw[thick] (0,0) rectangle (7,-2);
    \draw[dashed] (3,0) --++(0,-2);
    \draw[dashed] (4,0) --++(0,-2);
    \draw[dotted] (7,0) grid[step=1] (8,-2);
    \draw[thick] (7,0) rectangle (8,-2);
    \node[below] at (0.5,-2) {${\scriptstyle 1}$};
    \node[below] at (1.5,-2) {${\scriptstyle 2}$};
    \node[below] at (2.5,-2) {${\scriptstyle 3}$};
    \node[below] at (3.5,-2) {${\scriptstyle 0}$};
    \node[below] at (4.5,-2) {${\scriptstyle -3}$};
    \node[below] at (5.5,-2) {${\scriptstyle -2}$};
    \node[below] at (6.5,-2) {${\scriptstyle -1}$};
    \draw[fill=black] (0.5,-0.5) circle (0.3);
    \draw[fill=black] (1.5,-0.5) circle (0.3);
    \draw[fill=black] (2.5,-0.5) circle (0.3);
    \draw[fill=black] (4.5,-0.5) circle (0.3);
    \draw[fill=black] (5.5,-0.5) circle (0.3);
    \draw[fill=black] (6.5,-0.5) circle (0.3);
    \draw[fill=black] (7.5,-0.5) circle (0.3);
    \draw[fill=black] (0.5,-1.5) circle (0.3);
    \draw[fill=black] (1.5,-1.5) circle (0.3);
    \draw[fill=black] (3.5,-1.5) circle (0.3);
    \draw[fill=black] (4.5,-1.5) circle (0.3);
    \draw[fill=black] (5.5,-1.5) circle (0.3);
    \end{scope}
    \begin{scope}[shift={(23,-30)}]
    \draw[dotted] (0,0) grid[step=1] (7,-3);
    \draw[thick] (0,0) rectangle (7,-3);
    \draw[dashed] (3,0) --++(0,-3);
    \draw[dashed] (4,0) --++(0,-3);
    \draw[dashed] (0,-2) --++(7,0);
    \draw[dashed] (0,-3) --++(7,0);
    \draw[dotted] (-1,1) grid[step=1] (0,0);
    \draw[thick] (-1,1) rectangle (0,0);
    \draw[dotted] (-1,0) grid[step=1] (0,-3);
    \draw[thick] (-1,0) rectangle (0,-3);
    \draw[dashed] (-1,-2) --++(1,0);
    \draw[dashed] (-1,-3) --++(1,0);
    \draw[dotted] (0,1) grid[step=1] (7,0);
    \draw[thick] (0,1) rectangle (7,0);
    \draw[dashed] (3,1) --++(0,-1);
    \draw[dashed] (4,1) --++(0,-1);
    \node[below] at (-0.5,-3) {${\scriptstyle 0^*}$};
    \node[below] at (0.5,-3) {${\scriptstyle 1}$};
    \node[below] at (1.5,-3) {${\scriptstyle 2}$};
    \node[below] at (2.5,-3) {${\scriptstyle 3}$};
    \node[below] at (3.5,-3) {${\scriptstyle 0}$};
    \node[below] at (4.5,-3) {${\scriptstyle -3}$};
    \node[below] at (5.5,-3) {${\scriptstyle -2}$};
    \node[below] at (6.5,-3) {${\scriptstyle -1}$};
    \draw[fill=black] (-0.5,-0.5) circle (0.3);
    \draw[fill=black] (0.5,-1.5) circle (0.3);
    \node[right] at (7,0.5) {${\scriptstyle 0^*}$};
    \node[right] at (7,-0.5) {${\scriptstyle 1}$};
    \node[right] at (7,-1.5) {${\scriptstyle 2}$};
    \node[right] at (7,-2.5) {${\scriptstyle 0}$};
    \node[right] at (8,-1){$\cdots$};
    \end{scope}
    \begin{scope}[shift={(-1,-30)}]
    \draw[dotted] (0,0) grid[step=1] (7,-2);
    \draw[thick] (0,0) rectangle (7,-2);
    \draw[dashed] (3,0) --++(0,-2);
    \draw[dashed] (4,0) --++(0,-2);
    \draw[dotted] (7,0) grid[step=1] (8,-2);
    \draw[thick] (7,0) rectangle (8,-2);
    \node[below] at (0.5,-2) {${\scriptstyle 1}$};
    \node[below] at (1.5,-2) {${\scriptstyle 2}$};
    \node[below] at (2.5,-2) {${\scriptstyle 3}$};
    \node[below] at (3.5,-2) {${\scriptstyle 0}$};
    \node[below] at (4.5,-2) {${\scriptstyle -3}$};
    \node[below] at (5.5,-2) {${\scriptstyle -2}$};
    \node[below] at (6.5,-2) {${\scriptstyle -1}$};
    \node[below] at (7.5,-2) {${\scriptstyle 0^*}$};
    \draw[fill=black] (0.5,-0.5) circle (0.3);
    \draw[fill=black] (1.5,-0.5) circle (0.3);
    \draw[fill=black] (2.5,-0.5) circle (0.3);
    \draw[fill=black] (3.5,-0.5) circle (0.3);
    \draw[fill=black] (4.5,-0.5) circle (0.3);
    \draw[fill=black] (5.5,-0.5) circle (0.3);
    \draw[fill=black] (6.5,-0.5) circle (0.3);
    \draw[fill=black] (7.5,-0.5) circle (0.3);
    \draw[fill=black] (0.5,-1.5) circle (0.3);
    \draw[fill=black] (1.5,-1.5) circle (0.3);
    \draw[fill=black] (2.5,-1.5) circle (0.3);
    \draw[fill=black] (3.5,-1.5) circle (0.3);
    \draw[fill=black] (4.5,-1.5) circle (0.3);
    \draw[fill=black] (5.5,-1.5) circle (0.3);
    \draw[fill=black] (6.5,-1.5) circle (0.3);
    \draw[fill=black] (7.5,-1.5) circle (0.3);
    \node[left] at (0,-1.25){$\cdots$};
    \end{scope}
    \end{tikzpicture}
\end{center}
\begin{Remark}\label{def-abacus}
    For the completeness of the discussion, we briefly recall the original definition of $\beta$-sets and level-$1$ abacus, see, for instance, \cite{LQ} for details. A $\beta$-set $\B$ is a subset of $\mathbb{Z}$ such that both $\max({\B})$ and $\min(\mathbb{Z} \backslash \B)$ exist.
    Its charge $s(\B)$ is defined as $s(\B) := |\B \cap \mathbb{Z}_{\geq 0}| - |\mathbb{Z}_{<0} \setminus \B|$.
    Given a horizontal line with positions labelled by $\mathbb{Z}$ in an ascending order going from left to right, we can obtain the abacus display of a $\beta$-set by putting a bead at the position $x$ for each $x \in \B$. 
    There is a bijection between $\beta$-sets and charged partitions, sending $\B$ to a partition with charge $s(\B)$.

    Note that moving the beads left or right on an abacus does not change its charge.
    In particular, given an abacus with charge $s$, we can push all the beads as left as possible, so that all the beads are arranged adjacently in order, and get the abacus corresponding to $\B = \mathbb{Z}_{ < s}$, which we call the vacuum abacus. 
    Under the correspondence between abacus and charged partitions, such an vacuum abacus always corresponds to the empty partition.
\end{Remark}
 
We say a row in an abacus is a full row if it goes infinitely in both directions, and a half row if it is bounded on the left. In this way we get an abacus with $r$ full rows labelled by $1,2,\ldots,r$ from top to bottom, with possibly a half row labelled by $0$ at bottom as well as a half row labelled by $\gh{0}$ on the top, depending on the type of $\lgdhat$.

Graphically, each full row in the abacus corresponds to a $\beta$-set in an obvious way.
In the cases $X^{(\r)} = A^{(1)}$ or $R_{\bar{0}}$ even, there is no half row in ${\B}(\M)$ for any Maya diagram $\M$. The abacus can be interpreted to a charged $r$-multipartition.
In particular, in case that $(\lghat_{\lleft}, \lgdhat)$ is of type $(A^{(1)},A^{(1)})$, the abacus defined in this way coincides with the classical one.

The charge of a $\beta$-set can be read by counting certain beads on an abacus.
Here we similarly define the charge of the $p$-th row, $p = 1,2,\ldots, r$, to be $\dot{b}_p - \dot{b}_{-p}$, see the notations given above the formulae (\ref{weightY}).
Compare with the formulae of $\dot{\lambda}_{\M}$ we note that here the charge sequence of an abacus is its weight with respect to $\mathring{\liehd}{}^*$ up to a shift by some fixed $\mathbf{1}$. 
In particular, we do not define the charge of the half rows $0$ or $\gh{0}$, which is consistent with the fact that the corresponding weight component is always zero.

Through the one-to-one correspondence between $\B(\M)$ and $\M$, we can interpret the Lie-actions on Maya diagrams into operations on abaci.
Since the two opposite half rows $\pm p$ in a Maya diagram are merged into a same full row $p > 0$ in an abacus, on each full row of $\B(\M)$, the action of $\lghat$ can be directly understood as left or right moves of beads.
Therefore, when studying the action of $\lghat$, the row abacus model is more intuitive and concise than the Maya diagram model, while in contrast, considering the action of $\lgdhat$ on the row abacus $\B(\M)$ becomes more complicated.

Back to Fock space, the correspondence between several combinatorial objects also induces an isomorphism between $\F$, the ``finite" wedge space, and certain semi-infinite wedge construction of Fock space, as given in \cite{U}. 

\begin{Remark}\label{double-half-row}
    The graphical operations obtaining the abacus from a Maya diagram can be understood in a more conceptual way, which is related to the particle and anti-particle picture of the uncut board ${\bf F}_{\rm hor}$. 
    Put the Maya diagram $\M$ back into ${\bf F}_{\rm hor}$ without removing any cells. For each cell $(i, p, \gamma)$ such that $\psi^{i,p}(\gamma) \in \W^{h,-}$, $p = 0, -1, \ldots, -r, \gh{0}$, we consider the cell $(-i, -p, -\gamma)$ in $\W^{h,+}$, which is paired with it with respect to the bilinear form on $\W^h$, and color it with the opposite color of $(i, p, \gamma)$. 
    Assigning each neutral particle with any color, we get the colored board ${\bf F}_{\rm hor}$. 
    It is remarkable that any choices of coloring for the neutral particles give the same weight data, thereby do not change any of the results in the subsequent sections.

    Then we cut the colored board ${\bf F}_{\rm hor}$ again as follows. We keep all the rows with positive labels $p > 0$, as well as the cells lying in $\W^{h,-}$ in the rows $0$ and $\gh{0}$. Removing all other cells, the resulting board is exactly the row abacus $\B(\M)$.

    On the other hand, we can also ``duplicate'' the half rows in the abacus to make them into full rows. That is, we keep all full rows labelled by $p = 0, 1, \ldots, r, \gh{0}$ in the colored board ${\bf F}_{\rm hor}$. Denote the resulting abacus by ${\sf DB}(\M)$.
    In some sense, this duplicated abacus ${\sf DB}(\M)$ with all rows full can be regarded as a weighted abacus with weight $\frac{1}{2}$ assigned to rows $0$ and $\gh{0}$. 
    
    The advantage of this viewpoint is that all the actions of $\lghat$ on the weighted abacus are given by left or right moves of beads, which is consistent with the original abacus model in type $A^{(1)}$. It is remarkable that when considering the $\lghat$-action on ${\sf DB}(\M)$, 
    each operator acts simultaneously on the positive and negative parts of the rows labelled by $0$ and $\gh{0}$, certifying the duplicating construction.

    The weighted abacus model allows us to treat half rows and full rows in a uniform way, which helps understanding the action of the diagram automorphism group $\Sigma$ on the abacus, as well as the ``half moves'' which appear when considering the moving vectors, see Remark \ref{remark-half-move} below.
\end{Remark}

\begin{Remark}
    Here the terms in our context are slightly different from the ones in physics literature. In physics, when considering $A^{(1)}$-type Lie algebras, the abacus defined here is called a ``Maya diagram'', see \cite{Tingley} for instance.
    However, when considering orthogonal types with neutral particles, there is no corresponding ``Maya diagram'' in physics, i.e., the abacus with half rows as defined here. Therefore, we ``flip back'' the anti-particles occupying positions in the half rows, resulting in the Maya diagram defined in this paper.
\end{Remark}

From a vertical Maya diagram $\dot{\M}$, one can define a column abacus $\dot{\sf B}(\dot{\M})$ similarly. Based on parallel discussions, on a column abacus, the action of $\lgdhat$ corresponds to vertical moves of beads, which is deeply related to the concept of moving vectors, see Section \ref{moving-vector} below. 

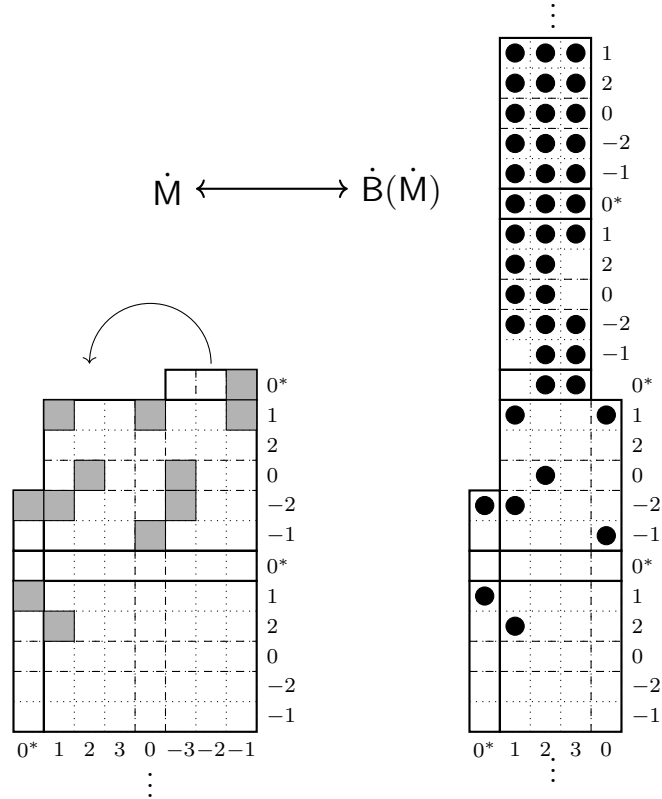
\begin{center}
\begin{tikzpicture}[scale=0.4]
	\draw[<->, thick] (15,-6) node[left, scale=1.2] {$\dot{\M}$} -- (20,-6) node[right, scale=1.2]{$\dot{\B}(\dot{\M})$};

	\begin{scope}[shift={(10,-13)}]

	\draw[->] (5.5,1.2) arc (0:180:2);
	\draw[dotted] (0,0) grid[step=1] (7,-5);
	\draw[thick] (0,0) rectangle (7,-5);
	\draw[dashed] (3,0) --++(0,-5);
	\draw[dashed] (4,0) --++(0,-5);
	\draw[dashed] (0,-2) --++(7,0);
	\draw[dashed] (0,-3) --++(7,0);
	\draw[dotted] (-1,-3) grid[step=1] (0,-5);
	\draw[thick] (-1,-3) rectangle (0,-5);
	\draw[dotted] (4,1) grid[step=1] (7,0);
	\draw[thick] (4,1) rectangle (7,0);
	\draw[dashed] (5,1) --++(0,-1);
	\draw[dashed] (6,1) --++(0,-1);
	\draw[fill=black!30] (6,1) rectangle (7,0);
	\draw[fill=black!30] (0,0) rectangle (1,-1);
	\draw[fill=black!30] (3,0) rectangle (4,-1);
	\draw[fill=black!30] (6,0) rectangle (7,-1);
	\draw[fill=black!30] (1,-2) rectangle (2,-3);
	\draw[fill=black!30] (4,-2) rectangle (5,-3);
	\draw[fill=black!30] (-1,-3) rectangle (0,-4);
	\draw[fill=black!30] (0,-3) rectangle (1,-4);
	\draw[fill=black!30] (4,-3) rectangle (5,-4);
	\draw[fill=black!30] (3,-4) rectangle (4,-5);
	\node[right] at (7,0.5) {${\scriptstyle 0^*}$};
	\node[right] at (7,-0.5) {${\scriptstyle 1}$};
	\node[right] at (7,-1.5) {${\scriptstyle 2}$};
	\node[right] at (7,-2.5) {${\scriptstyle 0}$};
	\node[right] at (7,-3.5) {${\scriptstyle -2}$};
	\node[right] at (7,-4.5) {${\scriptstyle -1}$};
	\end{scope}
	\begin{scope}[shift={(10,-19)}]
	\draw[dotted] (0,0) grid[step=1] (7,-5);
	\draw[thick] (0,0) rectangle (7,-5);
	\draw[dashed] (3,0) --++(0,-5);
	\draw[dashed] (4,0) --++(0,-5);
	\draw[dashed] (0,-2) --++(7,0);
	\draw[dashed] (0,-3) --++(7,0);
	\draw[dotted] (-1,1) grid[step=1] (0,0);
	\draw[thick] (-1,1) rectangle (0,0);
	\draw[dotted] (-1,0) grid[step=1] (0,-5);
	\draw[thick] (-1,0) rectangle (0,-5);
	\draw[dashed] (-1,-2) --++(1,0);
	\draw[dashed] (-1,-3) --++(1,0);
	\draw[dotted] (0,1) grid[step=1] (7,0);
	\draw[thick] (0,1) rectangle (7,0);
	\draw[dashed] (3,1) --++(0,-1);
	\draw[dashed] (4,1) --++(0,-1);
	\draw[fill=black!30] (-1,0) rectangle (0,-1);
	\draw[fill=black!30] (0,-1) rectangle (1,-2);
	\node[right] at (7,0.5) {${\scriptstyle 0^*}$};
	\node[right] at (7,-0.5) {${\scriptstyle 1}$};
	\node[right] at (7,-1.5) {${\scriptstyle 2}$};
	\node[right] at (7,-2.5) {${\scriptstyle 0}$};
	\node[right] at (7,-3.5) {${\scriptstyle -2}$};
	\node[right] at (7,-4.5) {${\scriptstyle -1}$};
	\node[below] at (-0.5,-5) {${\scriptstyle 0^*}$};
	\node[below] at (0.5,-5) {${\scriptstyle 1}$};
	\node[below] at (1.5,-5) {${\scriptstyle 2}$};
	\node[below] at (2.5,-5) {${\scriptstyle 3}$};
	\node[below] at (3.5,-5) {${\scriptstyle 0}$};
	\node[below] at (4.5,-5) {${\scriptstyle -3}$};
	\node[below] at (5.5,-5) {${\scriptstyle -2}$};
	\node[below] at (6.5,-5) {${\scriptstyle -1}$};
	\node[below] at (3.5,-5.5){$\vdots$};
	\end{scope}

	\begin{scope}[shift={(25,-13)}]
	\draw[dotted] (0,0) grid[step=1] (4,-5);
	\draw[thick] (0,0) rectangle (4,-5);
	\draw[dashed] (3,0) --++(0,-5);
	\draw[dashed] (4,0) --++(0,-5);
	\draw[dashed] (0,-2) --++(4,0);
	\draw[dashed] (0,-3) --++(4,0);
	\draw[dotted] (-1,-3) grid[step=1] (0,-5);
	\draw[thick] (-1,-3) rectangle (0,-5);
	\node[right] at (4,0.5) {${\scriptstyle 0^*}$};
	\node[right] at (4,-0.5) {${\scriptstyle 1}$};
	\node[right] at (4,-1.5) {${\scriptstyle 2}$};
	\node[right] at (4,-2.5) {${\scriptstyle 0}$};
	\node[right] at (4,-3.5) {${\scriptstyle -2}$};
	\node[right] at (4,-4.5) {${\scriptstyle -1}$};
	\draw[fill=black] (0.5,-0.5) circle (0.3);
	\draw[fill=black] (3.5,-0.5) circle (0.3);
	\draw[fill=black] (1.5,-2.5) circle (0.3);
	\draw[fill=black] (-0.5,-3.5) circle (0.3);
	\draw[fill=black] (0.5,-3.5) circle (0.3);
	\draw[fill=black] (3.5,-4.5) circle (0.3);
	\end{scope}
	\begin{scope}[shift={(25,-7)}]
	\draw[dotted] (0,0) grid[step=1] (3,-5);
	\draw[thick] (0,0) rectangle (3,-5);
	\draw[dashed] (0,-2) --++(3,0);
	\draw[dashed] (0,-3) --++(3,0);
	\draw[dotted] (0,-5) grid[step=1] (3,-6);
	\draw[thick] (0,-5) rectangle (3,-6);
	\draw[fill=black] (0.5,-0.5) circle (0.3);
	\draw[fill=black] (1.5,-0.5) circle (0.3);
	\draw[fill=black] (2.5,-0.5) circle (0.3);
	\draw[fill=black] (0.5,-1.5) circle (0.3);
	\draw[fill=black] (1.5,-1.5) circle (0.3);
	\draw[fill=black] (0.5,-2.5) circle (0.3);
	\draw[fill=black] (1.5,-2.5) circle (0.3);
	\draw[fill=black] (0.5,-3.5) circle (0.3);
	\draw[fill=black] (1.5,-3.5) circle (0.3);
	\draw[fill=black] (2.5,-3.5) circle (0.3);
	\draw[fill=black] (1.5,-4.5) circle (0.3);
	\draw[fill=black] (2.5,-4.5) circle (0.3);
	\draw[fill=black] (1.5,-5.5) circle (0.3);
	\draw[fill=black] (2.5,-5.5) circle (0.3);
	\node[right] at (3,-0.5) {${\scriptstyle 1}$};
	\node[right] at (3,-1.5) {${\scriptstyle 2}$};
	\node[right] at (3,-2.5) {${\scriptstyle 0}$};
	\node[right] at (3,-3.5) {${\scriptstyle -2}$};
	\node[right] at (3,-4.5) {${\scriptstyle -1}$};
	\end{scope}
	\begin{scope}[shift={(25,-19)}]
	\draw[dotted] (0,0) grid[step=1] (4,-5);
	\draw[thick] (0,0) rectangle (4,-5);
	\draw[dashed] (3,0) --++(0,-5);
	\draw[dashed] (4,0) --++(0,-5);
	\draw[dashed] (0,-2) --++(4,0);
	\draw[dashed] (0,-3) --++(4,0);
	\draw[dotted] (-1,1) grid[step=1] (0,0);
	\draw[thick] (-1,1) rectangle (0,0);
	\draw[dotted] (-1,0) grid[step=1] (0,-5);
	\draw[thick] (-1,0) rectangle (0,-5);
	\draw[dashed] (-1,-2) --++(1,0);
	\draw[dashed] (-1,-3) --++(1,0);
	\draw[dotted] (0,1) grid[step=1] (4,0);
	\draw[thick] (0,1) rectangle (4,0);
	\draw[dashed] (3,1) --++(0,-1);
	\draw[dashed] (4,1) --++(0,-1);
	\node[right] at (4,0.5) {${\scriptstyle 0^*}$};
	\node[right] at (4,-0.5) {${\scriptstyle 1}$};
	\node[right] at (4,-1.5) {${\scriptstyle 2}$};
	\node[right] at (4,-2.5) {${\scriptstyle 0}$};
	\node[right] at (4,-3.5) {${\scriptstyle -2}$};
	\node[right] at (4,-4.5) {${\scriptstyle -1}$};
	\node[below] at (-0.5,-5) {${\scriptstyle 0^*}$};
	\node[below] at (0.5,-5) {${\scriptstyle 1}$};
	\node[below] at (1.5,-5) {${\scriptstyle 2}$};
	\node[below] at (2.5,-5) {${\scriptstyle 3}$};
	\node[below] at (3.5,-5) {${\scriptstyle 0}$};
	\node[below] at (1.75,-5){$\vdots$};
	\draw[fill=black] (-0.5,-0.5) circle (0.3);
	\draw[fill=black] (0.5,-1.5) circle (0.3);
	\end{scope}
	\begin{scope}[shift={(25,-1)}]
	\draw[dotted] (0,0) grid[step=1] (3,-5);
	\draw[thick] (0,0) rectangle (3,-5);
	\draw[dashed] (0,-2) --++(3,0);
	\draw[dashed] (0,-3) --++(3,0);
	\draw[dotted] (0,-5) grid[step=1] (3,-6);
	\draw[thick] (0,-5) rectangle (3,-6);
	\draw[fill=black] (0.5,-0.5) circle (0.3);
	\draw[fill=black] (1.5,-0.5) circle (0.3);
	\draw[fill=black] (2.5,-0.5) circle (0.3);
	\draw[fill=black] (0.5,-1.5) circle (0.3);
	\draw[fill=black] (1.5,-1.5) circle (0.3);
	\draw[fill=black] (2.5,-1.5) circle (0.3);
	\draw[fill=black] (0.5,-2.5) circle (0.3);
	\draw[fill=black] (1.5,-2.5) circle (0.3);
	\draw[fill=black] (2.5,-2.5) circle (0.3);
	\draw[fill=black] (0.5,-3.5) circle (0.3);
	\draw[fill=black] (1.5,-3.5) circle (0.3);
	\draw[fill=black] (2.5,-3.5) circle (0.3);
	\draw[fill=black] (0.5,-4.5) circle (0.3);
	\draw[fill=black] (1.5,-4.5) circle (0.3);
	\draw[fill=black] (2.5,-4.5) circle (0.3);
	\draw[fill=black] (0.5,-5.5) circle (0.3);
	\draw[fill=black] (1.5,-5.5) circle (0.3);
	\draw[fill=black] (2.5,-5.5) circle (0.3);
	\node[right] at (3,-0.5) {${\scriptstyle 1}$};
	\node[right] at (3,-1.5) {${\scriptstyle 2}$};
	\node[right] at (3,-2.5) {${\scriptstyle 0}$};
	\node[right] at (3,-3.5) {${\scriptstyle -2}$};
	\node[right] at (3,-4.5) {${\scriptstyle -1}$};
	\node[right] at (3,-5.5) {${\scriptstyle 0^*}$};
	\node[above] at (1.75,0){$\vdots$};
	\end{scope}
\end{tikzpicture}
\end{center}

Now through the above two graphical operations, a Maya diagram corresponds to a row abacus as well as a column abacus, respectively. In this way, the Maya diagram serves as a bridge connecting the row and column abaci, which are respectively related to $\lghat$ and $\lgdhat$.
As a conclusion, there are $1-1$ correspondences:
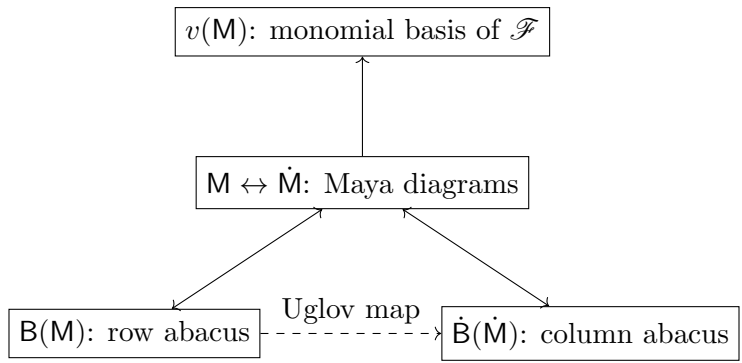
\begin{center}
	\begin{tikzpicture}
	\node[draw] (v) at (0, 2)  {$v(\M)$: monomial basis of $\F$};
    \node[draw] (m) at (0, 0)  {$\M\leftrightarrow \dot{\M}$: Maya diagrams};
    \node[draw] (ha) at (-3, -2)  {${{\B}}(\M)$: row abacus};
    \node[draw] (va) at (3, -2)  {${\dot{\B}}(\dot{\M})$: column abacus};
    
    \draw[<-] (v) -- (m);
    \draw[<->] (ha) -- (m);
    \draw[<->] (va) -- (m);
    \draw[dashed, ->] (ha) --node[midway, above]{Uglov map} (va);
	\end{tikzpicture}
\end{center}
In type $(A^{(1)}, A^{(1)})$, the dashed arrow gives a correspondence between charged $r$-multipartitions with charged $l$-multipartitions, which is precisely the Uglov bijection \cite{U}.
Therefore, by flipping negative rows and columns on Maya diagrams, we generalize the Uglov bijection to all types of classical affine algebras.

\begin{Remark}
    With specific choice of crystal action rules, one can endow an (abstract) $(\lghat, \lgdhat)$ bi-crystal structure on the set of all Maya diagrams (or all abaci).
    In fact, after quantization, the set of all abaci can be made to a $(U_{q}(\lghat), U_{-q^{-1}}(\lgdhat))$-crystal, as in \cite{U} for type $A^{(1)}$.
    Details of these will be given in subsequent on-going works.
    
    It is also worth mentioning that the crystal graphs of various classical affine Lie algebras have been given in \cite{HongKang} in the form of Young walls.
    In type $A^{(1)}$, these crystals are essentially a generalization of charged multipartitions.
    Therefore, the crystal graphs given by the abacus configurations will be equivalent to the Young walls, but at the same time reflecting more dual information.
\end{Remark}

\subsection{Action of reflection group revisit} \label{diag-auto-on-abacus}
Here we recall the action of $\Sigma$, the group of diagram automorphism, on the Fock space, viewed via Maya diagrams or abaci.

For the types whose Dynkin diagram has branching nodes, the group $\Sigma$ is generated by reflections. 
Interpreting the assignments of $\sigma_{l-1, l}$ on monomial basis elements given in \S\ref{subsection-liepair} to abacus configurations, we have the following graphical explanations for the action.
The reflection $\sigma_{l-1, l}$ acts on a row abacus by:
\begin{itemize}
    \item permuting each pair of adjacent columns labelled by $\pm l$;
    \item if there is a half row starting with $-l$, reverse the color of this starting bead.
\end{itemize}
When considering on a duplicated abacus, the second operation is again permuting a pair of adjacent columns and can be described uniformly as the first operation.

Similar action can be given for $\sigma_{0,1}$, with $\pm l$ replaced by $\pm 1$. 
In fact, one can give on row abaci the action of every simple reflections $s_1,s_2\ldots,s_{l-1}$ of the Weyl group in the same fashion: the simple reflection $s_i$ acts by permuting each pair of adjacent columns labelled by $i$ and $i+1$, as well as each pair $-i$ and $-i-1$.
Analogous interpretation of the action of $\dot{\Sigma}$ can be observe on column abaci.

\begin{Exam}\label{example-diag-auto-1}
    Take $(\lghat,\lgdhat) = (\lsohat(2l), \lsohat(2r))$ and $h = \frac{1}{2}$.
    Given $Y = (y_1,y_2,\ldots,y_l)\in \P^l_r$, consider the action of $\Sigma = \langle \sigma_{0,1}, \sigma_{l-1,l} \rangle$ on $\Y$.
    If $y_1<r$ or $y_l>0$, the action of $\Sigma$ on $\Y$ is non-trivial and can be present by the following Maya diagrams:
    
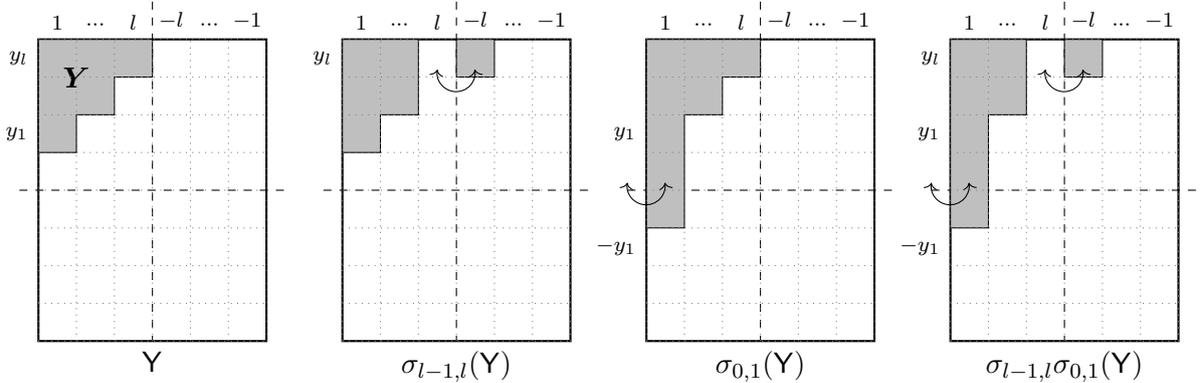
\begin{center}
    \begin{tikzpicture}[scale=0.5]
        \begin{scope}
            \draw[thick] (0,0) rectangle (6,-8);
            \draw[dashed] (3, 1) -- (3, -8);
            \draw[dashed] (-0.5,-4) -- (6.5,-4);
            \draw[fill=lightgray] (0, 0) 
            -- (0, -3) -- (1, -3) 
            -- (1, -2) -- (2, -2) 
            -- (2, -1) -- (3, -1) 
            -- (3, 0) -- (0, 0) -- cycle;
            
            \node[above] at (0.5,0) {$\scriptstyle {1}$};
            \node[above] at (1.5,0) {$\scriptstyle {\cdots}$};
            \node[above] at (2.5,0) {$\scriptstyle {l}$};
            \node[above] at (3.5,0) {$\scriptstyle {-l}$};
            \node[above] at (4.5,0) {$\scriptstyle {\cdots}$};
            \node[above] at (5.5,0) {$\scriptstyle {-1}$};
            \node[left] at (0,-0.5) {$\scriptstyle {y_l}$};
            \node[left] at (0,-2.5) {$\scriptstyle {y_1}$};
            \node[font=\bfseries\boldmath] at (1, -1) {$Y$};
            \node[below] at (3, -8) {$\Y$};
            \draw[gray, dotted] (0,0) grid[step=1] (6,-8);
        \end{scope}
        \begin{scope}[shift={(8,0)}]
            \draw[thick] (0,0) rectangle (6,-8);
            \draw[dashed] (3, 1) -- (3, -8);
            \draw[dashed] (-0.5,-4) -- (6.5,-4);
            \draw[fill=lightgray] (0, 0) 
            -- (0, -3) -- (1, -3)
            -- (1, -2) -- (2, -2)
            -- (2, 0) -- (0, 0) -- cycle;
            \draw[fill=lightgray] (3,0) rectangle (4,-1);
            
            \node[above] at (0.5,0) {$\scriptstyle {1}$};
            \node[above] at (1.5,0) {$\scriptstyle {\cdots}$};
            \node[above] at (2.5,0) {$\scriptstyle {l}$};
            \node[above] at (3.5,0) {$\scriptstyle {-l}$};
            \node[above] at (4.5,0) {$\scriptstyle {\cdots}$};
            \node[above] at (5.5,0) {$\scriptstyle {-1}$};
            \node[left] at (0,-0.5) {$\scriptstyle {y_l}$};
            \draw[<->] (2.5,-0.8) arc (170:370:0.5);

            \node[below] at (3, -8) {$\sigma_{l-1,l}(\Y)$};
            \draw[gray, dotted] (0,0) grid[step=1] (6,-8);
        \end{scope}    
        \begin{scope}[shift={(16,0)}]
            \draw[thick] (0,0) rectangle (6,-8);
            \draw[dashed] (3, 1) -- (3, -8);
            \draw[dashed] (-0.5,-4) -- (6.5,-4);
            \draw[fill=lightgray] (0, 0) 
            -- (0, -5) -- (1, -5) 
            -- (1, -2) -- (2, -2) 
            -- (2, -1) -- (3, -1) 
            -- (3, 0) -- (0, 0) -- cycle;
            
            \node[above] at (0.5,0) {$\scriptstyle {1}$};
            \node[above] at (1.5,0) {$\scriptstyle {\cdots}$};
            \node[above] at (2.5,0) {$\scriptstyle {l}$};
            \node[above] at (3.5,0) {$\scriptstyle {-l}$};
            \node[above] at (4.5,0) {$\scriptstyle {\cdots}$};
            \node[above] at (5.5,0) {$\scriptstyle {-1}$};
            \node[left] at (0,-2.5) {$\scriptstyle {y_1}$};
            \node[left] at (0,-5.5) {$\scriptstyle {-y_1}$};
            \node[below] at (3, -8) {$\sigma_{0,1}(\Y)$};
            \draw[<->] (-0.5,-3.8) arc (170:370:0.5);

            \draw[gray, dotted] (0,0) grid[step=1] (6,-8);
        \end{scope}    
        \begin{scope}[shift={(24,0)}]
            \draw[thick] (0,0) rectangle (6,-8);
            \draw[dashed] (3, 1) -- (3, -8);
            \draw[dashed] (-0.5,-4) -- (6.5,-4);
            \draw[fill=lightgray] (0, 0) 
            -- (0, -5) -- (1, -5) 
            -- (1, -2) -- (2, -2)
            -- (2, 0) -- (0, 0) -- cycle;
            \draw[fill=lightgray] (3,0) rectangle (4,-1);
            
            \node[above] at (0.5,0) {$\scriptstyle {1}$};
            \node[above] at (1.5,0) {$\scriptstyle {\cdots}$};
            \node[above] at (2.5,0) {$\scriptstyle {l}$};
            \node[above] at (3.5,0) {$\scriptstyle {-l}$};
            \node[above] at (4.5,0) {$\scriptstyle {\cdots}$};
            \node[above] at (5.5,0) {$\scriptstyle {-1}$};
            \node[left] at (0,-2.5) {$\scriptstyle {y_1}$};
            \node[left] at (0,-5.5) {$\scriptstyle {-y_1}$};
            \draw[<->] (-0.5,-3.8) arc (170:370:0.5);
            \node[left] at (0,-0.5) {$\scriptstyle {y_l}$};
            \draw[<->] (2.5,-0.8) arc (170:370:0.5);
            \node[below] at (3, -8) {$\sigma_{l-1,l}\sigma_{0,1}(\Y)$};
            \draw[gray, dotted] (0,0) grid[step=1] (6,-8);
        \end{scope}    
    \end{tikzpicture}
    \end{center}
    It is routine to check that each of these Maya diagrams $\M$ corresponds to a joint highest weight vector respectively.
    Also note that each $\M$ above is of the same vertical weight $\dot{\lambda}_{\M} = Y^t$.
    That is, in the Fock space, we may have at most 4 non-isomorphic irreducible summands $V_{\lghat}(\sigma(\lambda_{\Y})) \otimes V_{\lgdhat}(\dot{\lambda}_{\Y})$ as $(\lghat, \lgdhat)$-module.
    Piecing together these components, we get $V_{\Sigma\ltimes\lghat}(\lambda_{\Y}) \otimes V_{\lgdhat}(\dot{\lambda}_{\Y})$, an irreducible $(\Sigma\ltimes\lghat, \lgdhat)$-submodule of $\F$.
    In this way we can deal with those $(\lghat, \lgdhat)$-module which are not in 1-1 correspondence.
    See also Example \ref{example-diag-auto-2} below as a comparison versus the classical dual theory involving algebraic groups.
\end{Exam}

For type $A^{(1)}$, we define the action of $\sigma^{\#}_{cyc}$ on row abacus by shifting all beads to the right by one step.
For example, if $h = \frac{1}{2}$, let $\M$ be a Maya diagram of weight $(\lambda_{\M}, \dot{\lambda}_{\M}; -{\sf d}_{\M}) = \left(\sum\limits^l_{i=1} y_i \epsilon_i, \sum\limits^r_{p=1} \dot{y}_p \dot{\epsilon}_p; -{\sf d}_{\M}\right)$, then under this action, the weight of the resulting Maya diagram $\sigma^{\#}_{cyc}(\M)$ is given by
\[
    \left((y_l + r)\epsilon_1 + \sum^l_{i=2} y_{i-1} \epsilon_i,\; \sum^r_{p=1} \dot{y}_p \dot{\epsilon}_p + \mathbf{1}; \; -{\sf d}_{\M} - y_l - \frac{r}{2}\right).
\]
From the abacus viewpoint, this operation increases the charge of each row by $1$. Therefore, when we discuss the vacuum abacus corresponding to the joint highest weight vector, up to a diagram automorphism, we can always assume that its minimal charge is $0$, i.e., its corresponding Young diagram is taken from $\P^l_{r-1}$.

On the other hand, if we consider the pair $(\lslhat(l), \lglhat(r))$ with diagram automorphism $\dot{\sigma}^{\#}_{cyc}$ acting on column abacus, we have a similar operation by shifting all beads down by one step. 
As an operation on row abacus, the action of $\dot{\sigma}^{\#}_{cyc}$ on the charge set of the rows increases one element by $l$. 
In this way, up to a vertical diagram automorphism, we may assume that the charge set of the rows lies in $\{0,1,\ldots,l-1\}$, as considered in \cite{LQ}.

\section{Level-rank duality}

This section is devoted to proving Theorem \ref{mainDualThmIntro}. We given a remark for the duality data given in Table \ref{tbl-intro} before going into the proof.
\begin{Remark}\label{parity-in-Fock}
For type $(O^{(r)},O^{(r)})$, recall the parity decomposition of $\sigdomdel(R)$ in section \ref{Parity}. Comparing with the formulae (\ref{weightY}), we can see:
\begin{itemize}
    \item Each $\lambda_{\Y} \in \sigdom(R)_{\bar{0}, \bar{0}}$ shows up precisely in the case $\lgdhat = \lsohat(2r)$;
    \item Each $\lambda_{\Y} \in \sigdom(R)_{\bar{1}, \bar{0}}$ shows up precisely in the case $\lgdhat = \lsohat(2r+1)$, $h = \frac{1}{2}$;
    \item Each $\lambda_{\Y} \in \sigdom(R)_{\bar{0}, \bar{1}}$ shows up precisely in the case $\lgdhat = \lsohat(2r+1)$, $h = 0$;
    \item Each $\lambda_{\Y} \in \sigdom(R)_{\bar{1}, \bar{1}}$ shows up precisely in the case $\lgdhat = \lsohat(2r+1,1)$.
\end{itemize}
And vice versa for $\dot{\lambda} (Y)$ and $\lghat_{\lleft}$. Thus, we exhaust all irreducible integrable highest weight modules of orthogonal algebras in various spaces $\F$,
completing the result of \cite[Remark 4.2(i)]{Hasegawa1989}.
In particular, in any non-$A^{(1)}$ cases, all the dominant weights of $\lgtilde$ appearing in $\F$ lie in the same parity set.

For the dual pairs of other three types, it is remarkable that in either case of $h = 0$ or $h = \frac{1}{2}$, all possible dominant weights of given level appear in the dual space.
Thus for type $(A^{(1)}, A^{(1)})$ and $(C^{(1)}, C^{(1)})$, we can focus on the case $h = \frac{1}{2}$, in which there is a better model to consider. 
\end{Remark}

We briefly explain the idea of the proof.
The key point is to find all joint highest weight vectors in the Fock space. 
To do this, we shall first consider another dual pair on the same space, which consists of a finite dimensional subalgebra $\lgo_{\lleft}$ of $\lghat_{\lleft}$ and an infinite matrix algebra $\lgdhat{}^{\infty}$ containing $\lgdhat$:
\[
\begin{matrix}
    \lghat_{\lleft} & \oplus & \lgdhat & \curvearrowright & \F \\
    \begin{turn}{90} $\subset$ \end{turn} & & \begin{turn}{90} $\supset$ \end{turn} & & \begin{turn}{90} $=$ \end{turn} \\
    \lgo_{\lleft} & \oplus & \lgdhat{}^{\infty} & \curvearrowright & \F
\end{matrix}
\]
According to the general duality theory of finite dimensional Lie algebras (cf.\cite[\S 5.1]{Cheng2012}), we can easily find all joint highest weight vectors $v$ for the pair $(\lgo_{\lleft}, \lgdhat{}^{\infty})$ in the Fock space (see also \cite{Wang1999}).
Then the joint highest weight vectors we want must be of the form $\lgdhat{}^{\infty,-}.v$.
With the help of the identity in Proposition \ref{VirEv}, we shall check the eigenvalue of $D^{\Cl}$ on these vectors and determine all possible ones.

\subsection{Finite-infinite dual pairs}\label{sec-fin-inf} 
We refer the readers to \cite{Wang1999} and \cite[\S 5]{Cheng2012} for more detailed theory of this section. Keep the notations in Theorem \ref{mainDualThmIntro}. Fix a pair $(\lghat,\lgdhat)$ and $h$, 
recall that $\lgo$ is the subalgebra of $\lghat$ spanned by $X_{ij}(0)$, where $X=E,C,D$ or $A$ for types $A^{(1)}$, $C^{(1)}$, $O^{(\r)}$, or $A^{(2)}$ respectively. 
Via restriction, the Fock space $\F$ becomes a $\lgo$-module which is isomorphic to an infinite tensor space of spin or vector representations.
This action lifts to a suitable classical linear group $G$ whose type depending on the Fock space, satisfying that each irreducible $G$-module is also irreducible as a  $\sgmo \ltimes \lgo_{\lleft}$-module. Specific data of $G$ will be given later in Table \ref{table-fin-G}.

We recall from \cite[\S 5.1.4]{Cheng2012} a general duality theory. Let $G$  be a classical Lie group, and let $V$ be a rational representation of $G$ of countable dimension. Assume that $\mathscr{R}$ 
is an associative subalgebra of $\End(V)$ such that 
\begin{itemize}
    \item $V$ is an irreducible $\mathscr{R}$-module. 
    \item $\mathscr{R}$ is closed under the conjugation by $G$, that is, $\mathscr{R}$ is a $G$-module by conjugation.  
    \item As a $G$-module, $\mathscr{R}$ is a direct sum of finite dimensional irreducible modules. 
\end{itemize}
Then there is a $(G,\mathscr{R}^G)$-bimodule decomposition
\[V=\bigoplus_{\lambda} L(\lambda)\otimes M^{\lambda}\]
such that $L(\lambda)$'s and $M^{\lambda}$'s are respectively pairwise non-isomorphic irreducible $G$-modules and $\mathscr{R}^G$-modules. 

Now taking $V$ to be the space $\bigwedge(\W^{h,-})$, $G$ to be the group corresponding to $\lgo$, and $\mathscr{R}$ to be the Clifford algebra $\Cl=\Cl(\W^h)$, we have a pairwise non-isomorphic irreducible $(G,\Cl^G)$-bimodule decomposition of $\bigwedge(\W^{h,-})$. Except for two special cases, the Fock space $\F = \bigwedge(\W^{h,-})$ and the above decomposition can be directly applied to $\F$.
\begin{itemize}
    \item For type $A^{(1)}$, it is routine to replace $\W$ by $\bar{\W}$ in $\Cl(\W^h)$ and $\bigwedge(\W^{h,-})$ to recover the Fock space.
    \item For type $D^{(2)}$, $\F = \bigwedge(\W^{h,-})_{\text{even}}$ is not a $\Cl(\W^h)$-module. Since the generating relations of $\Cl(\W^h)$ are of even degree, we can consider its subalgebra $\Cl(\W^h)_{\text{even}}$, which is the quotient of $T(\W^h)_{\text{even}}$ given by the same relations. 
    It is easy to check that $\F = \bigwedge(\W^{h,-})_{\text{even}}$ is an irreducible representation of $\Cl(\W^h)_{\text{even}}$, and since the action of $\lghat$ on $\F$ is generated by degree-$2$ operators, we have $\lgo \subset \Cl(\W^h)_{\text{even}}$. 
    Hence, we can give a conjugation action of $G$ on $\mathscr{R} = \Cl(\W^h)_{\text{even}}$.
\end{itemize}
Therefore, in all cases we have a pairwise non-isomorphic irreducible $(G,\mathscr{R}^G)$-bimodule decomposition of $\F$.
\[\F=\bigoplus_{\lambda} V_G(\lambda) \otimes M^{\lambda}.\]

We briefly sketch the theory of determining $\mathscr{R}^G$. 
The following lemma summaries results from \cite[\S 3 and \S 4]{Howe} needed in our proofs.

\begin{Lem}\label{howe-dual}
    Let $U,V$ be $\mathbb{C}$-vector spaces with $\dim U$ finite and $\dim V$ countable. 
    The non-degenerate bilinear form on each space determines a (quasi-)isotropic decomposition, respectively. 
    Denote by $\lgl(V)_f$ the matrix algebra with finitely many non-zero entries, and $\lsp(V)_f$, $\lso(V)_f$ its subalgebras. Consider the conjugation action of the classical group of $U$ on $\Cl(U\otimes V)$. 
    \begin{enumerate}
        \item   The algebra of invariants of $GL(U)$ in $\End\!\left(\bigwedge(U\otimes V)\right)$ is generated by $\lgl(V)_f$.
        \item   If $\dim U$ is even and $V$ admits an isotropic decomposition $V=V^+\oplus V^-$, then the algebra of invariants of $Sp(U)$ in $\Cl(U\otimes V^-)$ is generated by $\lsp(V)_f$.
        \item   The algebra of invariants of $O(U)$ in $\Cl(U\otimes V)$ is generated by $\lso(V)_f$.
        In particular, if the decomposition of $V$ is quasi-isotropic, then 
            \begin{itemize}
                \item when $\dim U$ is even, $\lso(V)_f$ generates all commutants of $Pin(U)$;
                \item when $\dim U$ is odd, $\lso(V)_f$ generates all commutants of $Spin(U)$.
            \end{itemize}
    \end{enumerate}
\end{Lem}

We use the preceding lemma to identify in each type the group $G$ and its commutant $\Cl^G$, generated by an infinite matrix algebra $\liegd{}^{\infty}_f$.
The commutant pair $(G, \liegd{}^{\infty}_f)$ is within the framework of the uncompleted Clifford algebra $\Cl(U \otimes V)$.

Recall the construction of $\W^h$.
For type $A^{(1)}$, write $L_{\bar{0}}=2l$. Then
\[
\bar{\W}^h = \W^+_{2l}\otimes \left(\W^+_{2r}\otimes t^h\mathbb{C}[t^{\pm1}]\right)
    \oplus \W^-_{2l}\otimes \left(\W^-_{2r}\otimes t^h\mathbb{C}[t^{\pm1}]\right).
\]
Taking $U_{\bar{0}} = \W^+_{2l}$ and $V_{\bar{0}} = \W^+_{2r}\otimes t^h\mathbb{C}[t^{\pm1}]$, we have
${\rm gr}\,\Cl(\bar{\W}^h)\cong \bigwedge\!\left((U_{\bar{0}} \otimes V_{\bar{0}})\oplus (U_{\bar{0}} \otimes V_{\bar{0}})^*\right)$,
which is isomorphic to $\End\!\left(\bigwedge(U_{\bar{0}} \otimes V_{\bar{0}})\right)$.
Hence in this case we may take $G = GL(l)$, and its commutant is generated by the Lie algebra $\lgl(V_{\bar{0}})_f$.

For the other types,
\begin{align*}
\W^h &= \W_{L_{\bar{0}}}\otimes \left(\W_{R_{\bar{0}}}\otimes t^h\mathbb{C}[t^{\pm1}]
    \oplus \W_{R_{\bar{1}}}\otimes t^{h'}\mathbb{C}[t^{\pm1}]\right) \\
    &\quad\oplus \W_{L_{\bar{1}}}\otimes \left(\W_{R_{\bar{1}}}\otimes t^h\mathbb{C}[t^{\pm1}]
    \oplus \W_{R_{\bar{0}}}\otimes t^{h'}\mathbb{C}[t^{\pm1}]\right).
\end{align*}
We treat the $\W_{L_{\bar{0}}}$ and $\W_{L_{\bar{1}}}$ components separately. 
Let $\{ \bar{a}, \bar{b} \} = \mathbb{Z}_2$, and take $U_{\bar{a}} = \W_{L_{\bar{a}}}$ and
$V_{\bar{a}} = \W_{R_{\bar{a}}}\otimes t^h\mathbb{C}[t^{\pm1}] \oplus \W_{R_{\bar{b}}}\otimes t^{h'}\mathbb{C}[t^{\pm1}]$.
For type $C^{(1)}$, we take $G = Sp(L_{\bar{0}})$, and its commutant is generated by the affine Lie algebra $\lsp(V_{\bar{0}})_f$.
For types $O^{(\mathrm r)}$ or $A^{(2)}$, the commutant is generated by the Lie algebra $\lso(V_{\bar{a}})_f$. 
To determine $G_{\bar{a}}$, consider the (quasi-)isomorphic decomposition of $V_{\bar{a}}$. 
Note that the degree-$0$ polynomial part of $V$ corresponds to $\W_{R_{\bar{a} + \overline{2h}}}$. Thus:

\begin{itemize}
    \item If $R_{\bar{a} + \overline{2h}}$ is even, then $(U_{\bar{a}} \otimes V_{\bar{a}})^-$ decomposes as a direct sum of natural $\lso(L_{\bar{0}})$-modules, so the action lifts to $G_{\bar{a}} = O(L_{\bar{0}})$. 
    \item If $R_{\bar{a} + \overline{2h}}$ is odd, then $(U_{\bar{a}} \otimes V_{\bar{a}})^-$ contains spinor modules as an $\lso(L_{\bar{0}})$-module, and the decomposition of $V_{\bar{a}}$ is quasi-isotropic.
    Discuss the parity of $L_{\bar{a}}$. If $L_{\bar{a}}$ is even the action lifts to $G_{\bar{a}} = Pin(L_{\bar{a}})$, and if $L_{\bar{a}}$ is odd it lifts to $G_{\bar{a}} = Spin(L_{\bar{a}})$.
\end{itemize}
It is remarkable that since there are sign twists in the constructions for types $A^{(2)}$ comparing with type $O^{(\mathrm r)}$, the resulting commutant Lie algebras differ by subtle conventions. 
For simplicity here we denote both the algebras by $\lso(V_{\bar{a}})_f$.
For the precise differences arising from choice of sign conventions in infinite matrix algebras, see \cite[\S 1.3]{Wang1999}.

Piecing two components together, we set
$G=G_{\bar{0}}\times G_{\bar{1}}$ acting respectively on $\W_{L_{\bar{0}}}$ and $\W_{L_{\bar{1}}}$.
The Lie algebra generating the commutant is then the direct sum
$\liegd{}^{\infty}_f = (\liegd{}^{\infty}_f)_{\bar{0}} \oplus (\liegd{}^{\infty}_f)_{\bar{1}}$ with $(\liegd{}^{\infty}_f)_{\bar{a}} = \lso(V_{\bar{a}})_f$.

Summarizing the above discussion, the finite-infinite commutant pairs $(G,\liegd{}^{\infty}_f)$
inside $\Cl$ for each type are chosen as shown in the following table.

\begin{table}[!htbp]
\centering
\scalebox{0.95}{
\begin{tabular}{c||c|c}
\hline
Type         & $G$          & $\liegd{}^{\infty}_f$ \\
\hline
$A^{(1)}$   &  $GL(l)$      & $\lgl(V_{\bar{0}})_f$, $V_{\bar{0}} = \W^+_{2r}\otimes t^h\mathbb{C}[t^{\pm1}]$  \\
\hline
$C^{(1)}$   &  $Sp(2l)$     & $\lsp(V_{\bar{0}})_f$, $V_{\bar{0}} = \W_{R_{\bar{0}}}\otimes t^h\mathbb{C}[t^{\pm1}]$  \\
\hline
$O^{(\mathrm r)}$ or $A^{(2)}$  & 
    \makecell
    {
        $G = G_{\bar{0}} \times G_{\bar{1}}$, \\
        for $\bar{a}=\bar{0},\bar{1}$,\\
        $G_{\bar{a}} = 
        \begin{cases}
            O(L_{\bar{a}}),& R_{\overline{2h}} \text{ even};\\
            Pin(L_{\bar{a}}),& R_{\overline{2h}} \text{ odd},\, L_{\bar{a}}\text{ even};\\
            Spin(L_{\bar{a}}),& R_{\overline{2h}} \text{ odd},\, L_{\bar{a}}\text{ odd}.
        \end{cases}$
    }& 
    \makecell
    {
        $\lso(V_{\bar{0}})_f \oplus \lso(V_{\bar{1}})_f$, \\
        for $\{\bar{a}, \bar{b}\}=\mathbb{Z}_2$,\\
        $V_{\bar{a}} = \W_{R_{\bar{a}}}\otimes t^h\mathbb{C}[t^{\pm1}] \oplus \W_{R_{\bar{b}}}\otimes t^{h'}\mathbb{C}[t^{\pm1}]$.
    }\\
\hline
\end{tabular}
}
\caption{Finite-infinite commutant pairs $(G,\liegd{}^{\infty}_f)$ in $\Cl$}
\label{table-fin-G}
\end{table}
In particular, since we always assume $L_{\bar{1}}\le 1$, the low-dimensional orthogonal groups are interpreted as $O(0)=Pin(0)=1$ trivial and $O(1)=Spin(1)=\pm1$.
When $L_{\bar{1}} = 0$, we use the convention that $V_{\bar{1}} = 0$ since the component is zero.

To give a realization of the affine algebra $\lgdhat$ inside the infinite matrix algebra, we need to introduce the central extension $\lgdhat{}^{\infty}_f$ as well as its completed version $\lgdhat{}^{\infty}$.
The algebra $\lgl(V_{\bar{a}})_f$ admits a central extension 
$\lglhat(V_{\bar{a}})_f = \lgl(V_{\bar{a}})_f \oplus \mathbb{C} K_{\bar{a}}$ defined by the 2-cocycle
\[
\tau(x, y) = \Tr([J_{\infty}, x]y), \mbox{ where } J_{\infty} = \sum_{\gamma < 0, \mbox{ or } \gamma = 0, p>0} E_{(p, \gamma)(p, \gamma)}.
\]
Denote by $\lgl(V_{\bar{a}})$ the completed Lie algebra of matrices with finitely many non-zero diagonals. Using the same $\tau$ one can also construct the (unique) non-trivial central extension $\lglhat(V_{\bar{a}}) = \lgl(V_{\bar{a}}) \oplus \mathbb{C} K_{\bar{a}}$.

In this way, we extend in each case the algebra $\liegd{}^{\infty}_f$ to $\lgdhat{}^{\infty}_f$ and $\lgdhat{}^{\infty}$ as direct sums of subalgebras of $\lglhat(V_{\bar{a}})_f$ and $\lglhat(V_{\bar{a}})$, respectively.
In the presence of a normal-ordering renormalization, the commutant pair $(G,\liegd{}^{\infty}_f)$ in $\Cl$ is naturally lifted to the dual pair $(G, \lgdhat{}^{\infty}_f)$ on  the Fock space.
The action can be written as $\dot{\pi}{}^{\infty}: \lgdhat{}^{\infty}_f \to \Cl$, 
with the central element $K_{\bar{0}}$ mapping to the scalar $l$ for type $A^{(1)}$, or $K_{\bar{a}}$ mapping to $L_{\bar{a}}$ for other types.

Note that the action of $\lgdhat{}^{\infty}_f$ on $\F$ is locally finite, that is, given $v \in \F$, for each diagonal all but finitely many entries send $v$ to $0$.
Abusing notations, the extended morphism $\dot{\pi}{}^{\infty}: \lgdhat{}^{\infty} \lra \widetilde{\Cl}$ gives a well-defined representation of $\lgdhat{}^{\infty}$ on $\F$.
Since each locally finite $\lgdhat{}^{\infty}_f$-module can be identically viewed as a $\lgdhat{}^{\infty}$-module,  $M^{\lambda}$'s are still pairwise non-isomorphic irreducible as $\lgdhat{}^{\infty}$-module.
Via the process
$$
\liegd{}^{\infty}_f 
\xrightarrow{\text{Normal Ordering}} 
\lgdhat{}^{\infty}_f 
\xrightarrow{\text{Completion}} 
\lgdhat{}^{\infty},
$$
eventually we have a pairwise non-isomorphic irreducible $(G, \lgdhat{}^{\infty})$-bimodule decomposition:
\[
\F=\bigoplus_{\lambda} V_G(\lambda) \otimes M^{\lambda}.
\]

To complete the duality it suffices to determine the highest weight vectors and the weights in each component.
For this we briefly introduce some notations for the weights of $\lgdhat{}^{\infty}_f$.
Let ${\bf I}(V_{\bar{a}})$ be the index set of the space $V_{\bar{a}}$ for $a = 0, 1$ with subsets ${\bf I}(V_{\bar{a}})^{\pm}$ corresponding to the (quasi)-isotropic decomposition.
Write $X_{(p,\gamma)(q,\eta)}$ the basis element of $\liegd{}^{\infty}$, with $X = E, C$ or $D$ depending on types.
The Cartan subalgebra $\lhdhat{}^{\infty}_f$ has basis $X_{(p,\gamma)(p,\gamma)}, K_{\bar{0}}, K_{\bar{1}}$, where $(p,\gamma)$ ranges over  ${\bf I}(V_{\bar{0}})^- \sqcup {\bf I}(V_{\bar{1}})^-$.
Denote by $\dot{\epsilon}_{(p,\gamma)}, \Lambda_{\bar{0}}, \Lambda_{\bar{1}}$ the dual elements in the dual space.
Since we only consider locally finite modules, we can restrict our discuss to the dual subspace $(\lhdhat{}^{\infty})^*_f = \bigoplus\mathbb{C}\dot{\epsilon}_{(p,\gamma)} \oplus \mathbb{C}\Lambda_{\bar{0}} \oplus \mathbb{C}\Lambda_{\bar{1}}$.
Obviously each monomial basis vector in $\F$ is a weight vector of $\lgdhat{}^{\infty}_f$. 
Similarly as in \S \ref{Maya-diagrams}, we can read the weight $\dot{\lambda}{}^{\infty}_{\M}$ of a Maya diagram by:
\[
\dot{b}_{(p,\gamma)}=\#\{ i \mid \mbox{The cell } (i,p,\gamma)\mbox{ in } \M \mbox{ is black}\},
\]
\[
\dot{\lambda}{}^{\infty}_{\M} = \sum\limits_{(p,\gamma)} \left(\dot{b}_{(p,\gamma)} - \dot{b}_{(-p,-\gamma)} + \left(\frac{L_{\overline{2h}}}{2} - \lfloor \frac{L_{\overline{2h}}}{2} \rfloor\right)\delta_{\gamma, 0} \right) \dot{\epsilon}_{(p,\gamma)}.
\] 
In particular, for each $\gamma<0$, the coefficient of $\dot{\epsilon}_{(p,\gamma)}$ is precisely $\dot{b}_{(p,\gamma)}$, which is a non-negative integer.

Regarded through representation in $\widetilde{\Cl}$, the affine algebra $\lgdhat$ can be viewed as a subalgebra of $\lgdhat{}^{\infty}$ with the embedding written as $E_{pq}(n) \mapsto \sum_{\gamma}E_{(p,\gamma),(q,n+\gamma)}$.
Note that through this embedding, the diagram automorphisms of $\lgdhat$ can also act on $\lgdhat{}^{\infty}$.
Reading the weights of Maya diagrams gives a projection on weight spaces: 
\begin{align*}
    {\sf pr}: (\lhdhat{}^{\infty})^*_f &\to \tilde{\dot{\lieh}}{}^*, \\
    \dot{\epsilon}_{(\pm|p|,\gamma)} &\mapsto \pm\dot{\epsilon}_{|p|} + \gamma\dot{\delta};\quad \dot{\epsilon}_{(0,\gamma)}, \dot{\epsilon}_{(\gh{0},\gamma)} \mapsto 0\\
    \Lambda_{\bar{0}}, \Lambda_{\bar{1}} &\mapsto \dot{\Lambda}_0.
\end{align*}
Note that under the projection ${\sf pr}$, the signed root lattice is not always preserved especially when the Dynkin diagram of $\lgdhat$ has branching point. 
However, if we take $\lgdhat_{\ex} = \lsohat(2r+1,1)$, 
one can check that we always have ${\sf pr}((\dot{Q}{}^{\infty})^{\pm}) = \tilde{\dot{Q}}{}^{\pm}_{\ex}$. This observation will be used in later proofs.

For types $O^{(\r)}$ and $A^{(2)}$, the infinite matrix algebra $\lgdhat{}^{\infty}$ also admits non-trivial diagram automorphisms in the cases $R_{\overline{2h}}$ even.
To simplify the discussion we assume $L_{\bar{1}} = 0$, while the $L_{\bar{1}} = 1$ case can be given analogously.
Let $(p_1,\gamma_1) = (1, -\frac{1}{2})$ for $h = \frac{1}{2}$ or $(-l, 0)$ for $h = 0$ the label of the ``highest" row in ${\bf I}(V_{\bar{0}})^-$.
Then in case $R_{\overline{2h}}$ is even, the rows $\pm(p_1,\gamma_1)$ are adjacent in ${\bf F}_{\rm ver}$.
Denote by $\dot{\sigma}_{\pm(p_1,\gamma_1)}$ the diagram automorphism switching the branching end points of the Dynkin diagram of $\lgdhat{}^{\infty}$, and let $\dot{\Sigma}{}^{\infty} = \langle \dot{\sigma}_{\pm(p_1,\gamma_1)} \rangle$.
It acts on $\dot{\B}(\dot{\M})$ by permuting the rows $\pm(p_1,\gamma_1)$ (see also \S \ref{diag-auto-on-abacus}).
In this way we can also make $\F$ a $(\sgmo \ltimes \lgo, \dot{\Sigma}{}^{\infty} \ltimes \lgdhat{}^{\infty})$-module.

As remarked in Example \ref{example-diag-auto-1}, we shall compare the decompositions of $\F$ as a $(\sgmo \ltimes \lgo, \dot{\Sigma}{}^{\infty} \ltimes \lgdhat{}^{\infty})$-module and as a $(G, \lgdhat{}^{\infty})$-module. In case that $\sgmo$ is non-trial, $\sgmo=\langle\sigma\rangle$ and $\sigma$ acts on Maya diagrams by interchanging the columns indexed by $\rkCar$ and $-\rkCar$. With the help of diagram automorphisms, we are ready to describe all joint highest weight vectors in $\F$.
For $Y = (y_i)_{i=1}^l \in \P^l_{\infty}$, let $Y_+ = (\max\{y_i-r, 0\})_{i=1}^l$ and $Y_- = Y - Y_+ \in \P^l_r$.
Define Maya diagrams $\M^h_Y$ for $h = 0, \frac{1}{2}$  as follows.
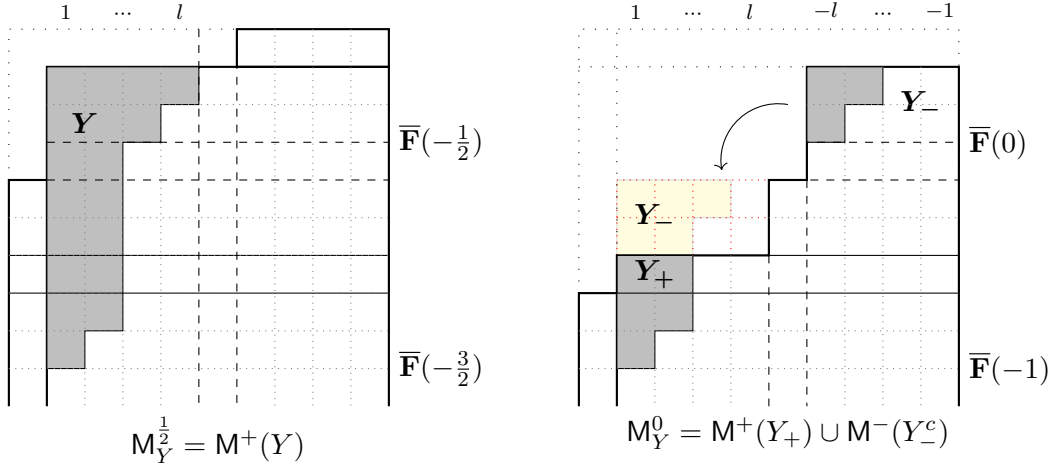
\begin{center}
\begin{tikzpicture}[scale=0.5]
    \begin{scope}
        \draw[thick] (0,-9) -- (0,0) -- (9,0) -- (9,-9);
        \draw[thick] (5,0) rectangle (9,1);
        \draw[thick] (0,-3) -- (-1,-3) -- (-1,-9);
        \draw[loosely dotted] (-1,-3) -- (-1,1) -- (5,1);
        \draw[fill=lightgray] (0, 0) 
        -- (0, -8) -- (1, -8)
        -- (1, -7) -- (2, -7)
        -- (2, -2) -- (3, -2)
        -- (3, -1) -- (4, -1)
        -- (4, 0) -- (0, 0) -- cycle;
        
        \draw[dashed] (0, -2) -- (9, -2);
        \draw[dashed] (0, -3) -- (9, -3);
        \draw (-1,-5) rectangle (9,-6);
        \draw[dashed] (4, 1) -- (4, -9);
        \draw[dashed] (5, 1) -- (5, -9);

        \node[above] at (0.5,1) {$\scriptstyle {1}$};
        \node[above] at (2,1) {$\scriptstyle {\cdots}$};
        \node[above] at (3.5,1) {$\scriptstyle {l}$};
        \node[below] at (4.5, -9) {$\M^{\frac{1}{2}}_Y = \M^+(Y)$};
        \foreach \x in {1,2,...,3}
            \draw[gray, dotted] (\x,0) -- (\x,-9);
        \foreach \x in {6,7,...,8}
            \draw[gray, dotted] (\x,1) -- (\x,-9);
        \draw[gray, dotted] (0,-1) -- (9,-1);
        \foreach \y in {4,5,...,8}
            \draw[gray, dotted] (-1,-\y) -- (9,-\y);
        \node[font=\bfseries\boldmath] at (1, -1.5) {$Y$};
        \node[right] at (9, -2) {$\overline{\bf F}(-\frac{1}{2})$};
        \node[right] at (9, -8) {$\overline{\bf F}(-\frac{3}{2})$};
    \end{scope}
    \begin{scope}[shift={(15,0)}]
        \begin{colormixin}{15!white}
        \draw[fill=yellow, thin] (0, -3) 
        -- (0, -5) -- (2, -5)
        -- (2, -4) -- (3, -4)
        -- (3, -3) -- (0, -3) -- cycle;
        \end{colormixin}
        \draw[red, dotted] (0,-3) grid[step=1] (4,-5);

        \draw[thick] (0,-9) -- (0,-5) -- (4,-5) -- (4,-3) -- (5,-3) -- (5,0) -- (9,0) -- (9,-9);
        \draw[thick] (0,-6) -- (-1,-6) -- (-1,-9);
        \draw[loosely dotted] (-1,-6) -- (-1,1) -- (9,1) -- (9,0);
        \draw[loosely dotted] (0,-5) -- (0,1);
        \draw[loosely dotted] (-1,0) -- (5,0);
        \draw[fill=lightgray] (0, -5) 
        -- (0, -8) -- (1, -8)
        -- (1, -7) -- (2, -7)
        -- (2, -5) -- (0, -5) -- cycle;
        \node[font=\bfseries\boldmath] at (1, -5.5) {$Y_+$};

        \draw[fill=lightgray] (5, 0) 
        -- (5, -2) -- (6, -2)
        -- (6, -1) -- (7, -1)
        -- (7, 0) -- (5, 0) -- cycle;
        \node[font=\bfseries\boldmath] at (8, -1) {$Y_-$};
        \node[font=\bfseries\boldmath] at (1, -4) {$Y_-$};
        \draw[->] (4.5,-1) arc (80:190:1.5);

        \draw (0,-5) rectangle (9,-6);
        \draw[dashed] (4, -5) -- (4, -9);
        \draw[dashed] (5, -3) -- (5, -9);
        \draw[dashed] (5, -3) rectangle (9, -2);
        \node[above] at (0.5,1) {$\scriptstyle {1}$};
        \node[above] at (2,1) {$\scriptstyle {\cdots}$};
        \node[above] at (3.5,1) {$\scriptstyle {l}$};
        \node[above] at (5.5,1) {$\scriptstyle {-l}$};
        \node[above] at (7,1) {$\scriptstyle {\cdots}$};
        \node[above] at (8.5,1) {$\scriptstyle {-1}$};
        \node[below] at (4.5, -9) {$\M^{0}_Y = \M^+(Y_+) \cup \M^-(Y_-^c)$};

        \foreach \x in {1,2,...,3}
            \draw[gray, dotted] (\x,-5) -- (\x,-9);
        \foreach \x in {6,7,...,8}
            \draw[gray, dotted] (\x,0) -- (\x,-9);
        \draw[gray, dotted] (5,-1) -- (9,-1);
        \draw[gray, dotted] (4,-4) -- (9,-4);
        \foreach \y in {7,8}
            \draw[gray, dotted] (-1,-\y) -- (9,-\y);
        \node[right] at (9, -2) {$\overline{\bf F}(0)$};
        \node[right] at (9, -8) {$\overline{\bf F}(-1)$};
    \end{scope}    
\end{tikzpicture}
\end{center}
One can routinely check that when $Y$ runs over $\P^l_{\infty}$, the above Maya diagrams give joint highest weight vectors with weights $\lambda_{\M^h(Y)}$ exhausting all possible weights in $\sigdom$. 

Taking into account the correspondence between Maya diagrams and column abaci, one can straightforwardly check that for \( v = v(\M_Y^h) \), the conditions \(\sigma(v) \neq v\) and \(\dot{\sigma}(v) \neq v\) cannot both hold. 
Moreover, it is easy to see that \( v = v(\M_Y^h) \), \(\sigma(v)\), and \(\dot{\sigma}(v)\) are all highest weight vectors over \((\lgo, \lgdhat{}^{\infty})\). 
On the one hand, if $\sigma(v) \ne v$, then the two $\lgo$-modules generated by $v$ and $\sigma(v)$ sum up to a $G$-module. Since $\dot{\sigma}$ acts trivially, this makes a $(G, \lgdhat{}^{\infty})$-module into a $(\sgmo \ltimes \lgo, \dot{\Sigma}{}^{\infty} \ltimes \lgdhat{}^{\infty})$-module.
On the other hand, if $\dot{\sigma}(v) \ne v$, then according to the choice of $G$ one can check that these two vectors generate a pair of non-isomorphic $G$-modules. In this way, the two non-isomorphic $(G, \lgdhat{}^{\infty})$-modules generated by $v$ and $\sigma(v)$ sum up to a $(\sgmo \ltimes \lgo, \dot{\Sigma}{}^{\infty} \ltimes \lgdhat{}^{\infty})$-module.

For simplicity we consider the following case as an example, while other cases can be done in a similar way.

\begin{Exam}\label{example-diag-auto-2}
Take $(\lghat,\lgdhat) = (\lsohat(2l), \lsohat(2r))$ and $h = \frac{1}{2}$.
Write $\sigma = \sigma_{l, l-1}\in \sgmo$ and $\dot{\sigma} = \dot{\sigma}_{\pm(1, \frac{1}{2})} \in \dot{\Sigma}{}^{\infty}$.
The irreducible modules of $G = O(2l)$ and $\lgo = \lso(2l)$ fail to coincide in the following situations: (notations for $G$-modules are as in \cite{Wang1999})
\begin{itemize}
    \item Let $v(\M), v(\M')$ be two $(\lgo, \lgdhat{}^{\infty})$-highest weight vectors given by the following Maya diagrams respectively with $y_l>0$:
        
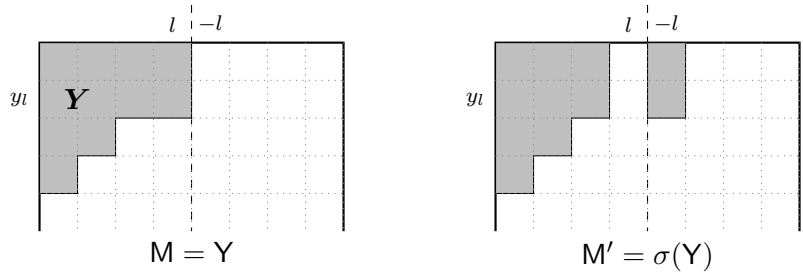
\begin{center}
        \begin{tikzpicture}[scale=0.5]
            \begin{scope}
                \draw[thick] (0,-5) -- (0,0) -- (8,0) -- (8,-5);
                \draw[dashed] (4, 1) -- (4, -5);
                \draw[fill=lightgray] (0, 0)
                -- (0, -4) -- (1, -4)
                -- (1, -3) -- (2, -3)
                -- (2, -2) -- (4, -2)
                -- (4, 0) -- (0, 0) -- cycle;
                
                \node[above] at (3.5,0) {$\scriptstyle {l}$};
                \node[above] at (4.5,0) {$\scriptstyle {-l}$};
                \node[left] at (0,-1.5) {$\scriptstyle {y_l}$};
                \node[font=\bfseries\boldmath] at (1, -1.5) {$Y$};
                \node[below] at (4, -5) {$\M = \Y$};
                \foreach \x in {1,2,...,8}
                    \draw[gray, dotted] (\x,0) -- (\x,-5);
                \foreach \y in {1,2,...,4}
                    \draw[gray, dotted] (0,-\y) -- (8,-\y);
            \end{scope}
            \begin{scope}[shift={(12,0)}]
                \draw[thick] (0,-5) -- (0,0) -- (8,0) -- (8,-5);
                \draw[dashed] (4, 1) -- (4, -5);
                \draw[fill=lightgray] (0, 0)
                -- (0, -4) -- (1, -4)
                -- (1, -3) -- (2, -3)
                -- (2, -2) -- (3, -2)
                -- (3, 0) -- (0, 0) -- cycle;
                \draw[fill=lightgray] (4,0) rectangle (5,-2);

                \node[above] at (3.5,0) {$\scriptstyle {l}$};
                \node[above] at (4.5,0) {$\scriptstyle {-l}$};
                \node[left] at (0,-1.5) {$\scriptstyle {y_l}$};
                \node[below] at (4, -5) {$\M' = \sigma(\Y)$};
                \foreach \x in {1,2,...,8}
                    \draw[gray, dotted] (\x,0) -- (\x,-5);
                \foreach \y in {1,2,...,4}
                    \draw[gray, dotted] (0,-\y) -- (8,-\y);
            \end{scope}    
        \end{tikzpicture}
        \end{center}
        The weights satisfy $\lambda_{\M'} = \sigma(\lambda_{\M})$ and $\dot{\lambda}{}^{\infty}_{\M'} = \dot{\lambda}{}^{\infty}_{\M}$.
        The irreducible $G$-module $V_G(|\lambda_{\M}|)$ splits into two non-isomorphic $\lgo$-module generated by $v(\M), v(\M')$.
        Then there is an identification between the components given by 
        \[
        V_G(|\lambda_{\M}|) \otimes M^{\lambda_{\M}} =  \left(V_{\lgo}(\lambda_{\M}) \oplus V_{\lgo}(\sigma(\lambda_{\M}))\right) \otimes V_{\lgdhat{}^{\infty}}(\dot{\lambda}{}^{\infty}_{\M}) = V_{\sgmo \ltimes \lgo}(\lambda_{\M}) \otimes V_{\lgdhat{}^{\infty}}(\dot{\lambda}{}^{\infty}_{\M}).
        \]
    \item Let $v(\M), v(\M')$ be two $(\lgo, \lgdhat{}^{\infty})$-highest weight vectors given by the following Maya diagrams respectively with $y^t_1<l$:
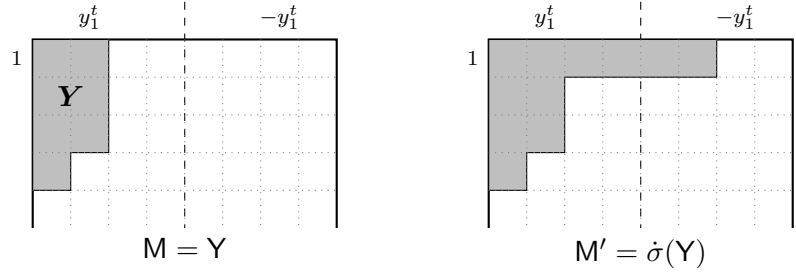
        \begin{center}
        \begin{tikzpicture}[scale=0.5]
            \begin{scope}
                \draw[thick] (0,-5) -- (0,0) -- (8,0) -- (8,-5);
                \draw[dashed] (4, 1) -- (4, -5);
                \draw[fill=lightgray] (0, 0) 
                -- (0, -4) -- (1, -4)
                -- (1, -3) -- (2, -3)
                -- (2, 0) -- (0, 0) -- cycle;
                
                \node[above] at (1.5,0) {$\scriptstyle {y^t_1}$};
                \node[above] at (6.5,0) {$\scriptstyle {-y^t_1}$};
                \node[left] at (0,-0.5) {$\scriptstyle {1}$};
                \node[font=\bfseries\boldmath] at (1, -1.5) {$Y$};
                \node[below] at (4, -5) {$\M = \Y$};
                \foreach \x in {1,2,...,7}
                    \draw[gray, dotted] (\x,0) -- (\x,-5);
                \foreach \y in {1,2,...,4}
                    \draw[gray, dotted] (0,-\y) -- (8,-\y);
            \end{scope}
            \begin{scope}[shift={(12,0)}]
                \draw[thick] (0,-5) -- (0,0) -- (8,0) -- (8,-5);
                \draw[dashed] (4, 1) -- (4, -5);
                \draw[fill=lightgray] (0, 0) 
                -- (0, -4) -- (1, -4)
                -- (1, -3) -- (2, -3)
                -- (2, -1) -- (6, -1)
                -- (6, 0) -- (0, 0) -- cycle;
                
                \node[above] at (1.5,0) {$\scriptstyle {y^t_1}$};
                \node[above] at (6.5,0) {$\scriptstyle {-y^t_1}$};
                \node[left] at (0,-0.5) {$\scriptstyle {1}$};
                \node[below] at (4, -5) {$\M' = \dot{\sigma}(\Y)$};
                \foreach \x in {1,2,...,7}
                    \draw[gray, dotted] (\x,0) -- (\x,-5);
                \foreach \y in {1,2,...,4}
                    \draw[gray, dotted] (0,-\y) -- (8,-\y);
            \end{scope}    
        \end{tikzpicture}
        \end{center}
        Then the weights $\lambda_{\M'} = \lambda_{\M}$ and $\dot{\lambda}{}^{\infty}_{\M'} = \dot{\sigma} (\dot{\lambda}{}^{\infty}_{\M})$.
        The $G$-modules generated by $v(\M), v(\M')$ are $V_G(\lambda_{\M}), V_G(\lambda_{\M}\otimes \det)$ respectively.
        These isomorphic irreducible $\lgo$-modules fail to retain isomorphic as $G$-module.
        Then there is an identification between the components given by 
        \begin{align*}
           &\left(V_G(\lambda_{\M}) \otimes M^{\lambda_{\M}}\right) \oplus \left(V_G(\lambda_{\M}\otimes \det) \otimes M^{\lambda_{\M}\otimes \det}\right) \\
        =  &\left(V_{\lgo}(\lambda_{\M}) \otimes V_{\lgdhat{}^{\infty}}(\dot{\lambda}{}^{\infty}_{\M})\right) \oplus \left(V_{\lgo}(\lambda_{\M}) \otimes V_{\lgdhat{}^{\infty}}(\dot{\sigma}(\dot{\lambda}{}^{\infty}_{\M}))\right) = V_{\lgo}(\lambda_{\M}) \otimes V_{\dot{\Sigma}{}^{\infty} \ltimes \lgdhat{}^{\infty}}(\dot{\lambda}{}^{\infty}_{\M}).
        \end{align*}
\end{itemize}
\end{Exam}

We state the duality theorem between finite dimensional algebras and infinite matrix algebras:
\begin{Thm}[cf.\cite{Wang1999}] \label{fin-inf}
    As a $(\sgmo \ltimes \lgo, \dot{\Sigma}{}^{\infty} \ltimes \lgdhat{}^{\infty})$-module, $\F$ decomposes as:
    \[
    \F = \bigoplus_{Y \in \P^l_{\infty}} V_{\sgmo_{\lleft} \ltimes \lgo_{\lleft}}(\lambda_{\M^h(Y)}) \otimes V_{\dot{\Sigma}{}^{\infty} \ltimes \lgdhat{}^{\infty}}(\dot{\lambda}{}^{\infty}_{\M^h(Y)}; \levelL)
    \]
    with joint highest weight vectors $v^h_Y$ given by Maya diagrams $\M^h(Y)$ associate to $Y$ defined above. 
\end{Thm} 

\begin{proof}
This theorem is nothing but a higher-level analogue of \cite{Wang1999}. We only give a sketch. The key point is that up to diagram automorphism the choice of $\lambda_{\M^h(Y)}$ exhausts all possible highest weight which can show up in certain direct summands of $\F$.
By general duality theory, the decomposition of $\F$ has to be strongly multiplicity free, and therefore the above decomposition exhausts all irreducible summands and the proof is finished.
\end{proof}

\begin{Remark}
    We cannot directly carry the general duality theory to affine-affine case, since in the former case a variant form of Jacobson density theorem is needed and therefore the finite dimensional condition is crucial.
    Alternatively, we can combine finite-dimensional duality theory and the identities on anomalies to deduce the affine duality.
\end{Remark}

\subsection{From (fin, inf) to (aff, aff)}\label{pf-main-dual}
We are now in the position to prove the main duality theorem.

\begin{proof}[Proof of Theorem \ref{mainDualThmIntro}]
By definition of $\lgo_{\lleft}$, $\lgdhat{}^{\infty}$, we have the following diagram. 
\[
\begin{matrix}
    \lghat_{\lleft} & \oplus & \lgdhat & \curvearrowright & \F \\
    \begin{turn}{90} $\subset$ \end{turn} & & \begin{turn}{90} $\supset$ \end{turn} & & \begin{turn}{90} $=$ \end{turn} \\
    \lgo_{\lleft} & \oplus & \lgdhat{}^{\infty} & \curvearrowright & \F
\end{matrix}
\]
The second row is a dual pair we obtained in the previous subsection. We need to find all joint highest weight vectors in $\F$ with respect to $\lghat_{\lleft}$ and $\lgdhat$.

By Theorem \ref{fin-inf}, for each $Y=(y_1,y_2,\dots,y_l)\in \P_{\infty}^l$ a partition with at most $l$ parts, there is  a joint highest weight vector 
\[
v^h_Y\in V_{\sgmo_{\lleft} \ltimes \lgo_{\lleft}}(\lambda_{\M^h(Y)}) \otimes V_{\dot{\Sigma}{}^{\infty} \ltimes \lgdhat{}^{\infty}}(\dot{\lambda}{}^{\infty}_{\M^h(Y)}; \levelL).
\] 
Consider the weight $(\lambda_{\M^h(Y)}, \dot{\lambda}_{\M^h(Y)}; {\sf d}_{\M^h(Y)})$. 
Up to $\Sigma_{\lleft}$, the weight $\lambda_{\M^h(Y)}$ lies in $\sigdom(\levelR)$ if and only if $(\lambda_{\M^h(Y)} | \theta_{\lleft}) \le \levelR$, or equivalently, $y_1 \le r$.
Therefore the weights $\lambda_{\M^h(Y)}$, $Y \in \P^{l}_{r}$ exhaust all possible classical parts of affine weights appearing in $\F$. 
In these cases we can abbreviate $\Y$ for $\M^h(Y)$.
It is routine to check that for $Y \in \P^{l}_{r}$, each $v^h_{Y}$ is clearly a joint highest weight vector with respect to $(\lghat_{\lleft}, \lgdhat)$ of classical weight $(\lambda_{\Y}, \dot{\lambda}_{\Y})$.
In particular, $\dot{\lambda}_{\Y} \in \dot{P}^+_{\dot{\Sigma}}(\levelL)$.
By Proposition \ref{VirEv} we have an identity for anomalies:
\begin{equation}
    {\sf d}^{\lghat}_{\lambda_{\Y};\levelR} + {\sf d}^{\lgdhat}_{\dot{\lambda}_{\Y};\levelL} =
    \begin{cases}
        \frac{1}{2}{\sf d}_{Y} + \frac{LR}{32}, &\mbox{ for type } (A^{(2)},A^{(2)});\\
        {\sf d}_{\Y}, &\mbox{ for other types.}
    \end{cases}\tag{5.2.1}\label{H1}
\end{equation}

Let $v \in \F$ be a $(\lghat_{\lleft}, \lgdhat)$-highest weight vector of classical weight $(\lambda, \dot{\mu})$. We shall prove that $v$ is equal to some $v_Y^h$ for some $Y\in\P^l_r$ up to scalar and the actions of $(\Sigma,\dot{\Sigma})$. 
Up to diagram automorphisms we may assume $\lambda \in \sigdom(\levelR)_{\lleft}$ and $\dot{\mu} \in \dot{P}^+_{\dot{\Sigma}}(\levelL)$.
Restricting to $\lgo_{\lleft}$, $v$ is clearly a $\lgo_{\lleft}$-highest weight vector of weight $\lambda$. We may assume $\lambda = \lambda_{\Y}$ for some $Y \in \P^{l}_{r}$.
Then according to Theorem \ref{fin-inf}, $v$ must lie in the space $V_{\lgo_{\lleft}}(\lambda_{\Y}) \otimes V_{\dot{\Sigma}{}^{\infty} \ltimes \lgdhat{}^{\infty}}(\dot{\lambda}{}^{\infty}_{\Y})$, 
or equivalently, $v = g. \dot{\sigma}{}^{\infty}(v^h_Y)$ for some $g \in \lgdhat{}^{\infty}, \dot{\sigma}{}^{\infty} \in \dot{\Sigma}{}^{\infty}$.
Since $Y \in \P^{l}_{r}$, an easy observation on abacus shows that $\dot{\sigma}{}^{\infty}(v^h_Y) = \dot{\sigma}(v^h_Y)$ for some $\dot{\sigma} \in \dot{\Sigma}$.
Then $\dot{\sigma}(v) = \dot{\sigma} g \dot{\sigma} (v^h_Y) = \dot{\sigma}(g).(v^h_Y)$ is again a $(\lghat_{\lleft}, \lgdhat)$-highest weight vector.
Up to diagram automorphisms we can reduce our discussion to the case $v = g. v^h_Y$ for some $g \in (\lgdhat{}^{\infty})^-$.
Again by Proposition \ref{VirEv} we have
\begin{equation}
    {\sf d}^{\lghat_{\lleft}}_{\lambda_{\Y};\levelR} + {\sf d}^{\lgdhat}_{\dot{\mu};\levelL} = 
    \begin{cases}
        \frac{1}{2}{\sf d}_{v} + \frac{LR}{32}, &\mbox{ for type } (A^{(2)},A^{(2)});\\
        {\sf d}_{v}, &\mbox{ for other types.}
    \end{cases}
       \tag{5.2.2}\label{H2}
\end{equation}

We claim that $\dot{\mu}=\dot{\lambda}_{\Y}$. Canceling  out ${\sf d}^{\lghat_{\lleft}}_{\lambda_{\lleft}(Y);\levelR}$ from (\ref{H2}) and (\ref{H1}) and taking formula (\ref{Formula-vacuum}) into account, in both cases we get
\begin{equation}
    \frac{(\dot{\mu} | \dot{\mu} + 2\dot{\rho}) - (\dot{\lambda}_{\Y} | \dot{\lambda}_{\Y} + 2\dot{\rho})}{2 (\levelL + \ddCox)} = {\sf d}_v - {\sf d}_{\Y}. \tag{5.2.3}\label{H}
\end{equation}

We discuss how the action of $(\lgdhat{}^{\infty})^-$ on $v_Y^h$ affects the eigenvalues of $D^{\Cl}(0)$. 
Let $v = \sum v_i$ be a decomposition such that each $v_i$ is non-zero of weight $\dot{\mu}{}^{\infty}_i$ with respect to $\lgdhat{}^\infty$. Then each $v_i$ is of the form $g_i v^h_Y$ with $g_i \in U(\lgdhat{}^{\infty})^-$ of weight $\dot{\zeta}_i = \dot{\mu}{}^{\infty}_i - \dot{\lambda}{}^{\infty}_{\Y} \in (\dot{Q}{}^{\infty})^-$.
Since with respect to $\tilde{\liegd}$, the weight vectors $v_i$ sum up to a weight vector $v$, we deduce that through ${\sf pr}$ each $\dot{\zeta}_i$ maps to the same element in affine root lattice, say,
\[
{\sf pr}(\dot{\zeta}_i) = (\dot{\mu} - \dot{\lambda}_{\Y}) + ({\sf d}_{\Y} - {\sf d}_v)\dot{\delta} \in \tilde{\dot{Q}}{}^{-}_{\ex}.
\]
In other words, writing $\dot{b} = {\sf d}_{\Y} - {\sf d}_v$ and $\bar{\dot{\beta}} = \dot{\mu} - \dot{\lambda}_{\Y}$, together with the decomposition $\dot{\delta}=\dot{\theta}_{
\ex}+\dot{a}_0\dot{\alpha}_{0,\ex}$, we have $\dot{b}\le 0$ and $\bar{\dot{\beta}} + \dot{b}\dot{\theta}_{\ex} \le 0$ with respect to $\dot{Q}{}^{-}_{\ex}$. By the inequality (\ref{Qexplus}),  since $\dot{\mu}, \dot{\lambda}_{\Y}, \dot{\rho} \in \dot{P}_{\dot{\Sigma}}^+$, we have 
\[
(\bar{\dot{\beta}} + \dot{b}\dot{\theta}_{\ex} \mid \dot{\mu}) + (\bar{\dot{\beta}} + \dot{b}\dot{\theta}_{\ex} \mid \dot{\lambda}_{\Y}) + (\bar{\dot{\beta}} + \dot{b}\dot{\theta}_{\ex} \mid 2\dot{\rho}) \le 0.
\]
Assume the equality holds, then all summands are zero since they are all non-positive. If $r=1$, then either $\bar{\dot{\beta}} + \dot{b}\dot{\theta}_{\ex} = 0$ or $\dot{\mu}=0=\dot{\lambda}_{\Y}$.  Now assume that $r>1$. For non-branching types, we can directly deduce that $\bar{\dot{\beta}} + \dot{b}\dot{\theta}_{\ex} = 0$ according to the fact that $(\dot{Q}^+ | \dot{\rho}) > 0$. 
For branching types recall that $\dot{\theta}_{\ex} = 2\dot{\epsilon}_1$, and the assumption implies that $\bar{\dot{\beta}} + \dot{b}\dot{\theta}_{\ex} = k \dot{\epsilon}_r$ for some $k \le 0$.
Write $\dot{\mu}-\dot{\lambda}_{\Y} = k\dot{\epsilon}_r-\dot{b}\dot{\theta}_{\ex}$, then $0 = (k\dot{\epsilon}_r | \dot{\mu}) - (k\dot{\epsilon}_r | \dot{\lambda}_{\Y}) = (k\dot{\epsilon}_r | k\dot{\epsilon}_r-\dot{b}\dot{\theta}_{\ex}) = k^2$, whence $\bar{\dot{\beta}} + \dot{b}\dot{\theta}_{\ex} = 0$.
As a conclusion, in all cases the above equality holds implies that either  $\bar{\dot{\beta}} + \dot{b}\dot{\theta}_{\ex} = 0$, or $\dot{\mu}=\dot{\lambda}_{\Y}=0$ and there is nothing to prove.

Moreover, from $\dot{\mu}, \dot{\lambda}_{\Y} \in \dot{P}^+_{\dot{\Sigma}}(\levelL)$ we deduce that $(\dot{\mu} | \dot{\theta}_{\ex}), (\dot{\lambda}_{\Y} | \dot{\theta}_{\ex}) \le \levelL$, and if both equalities hold, then $(\bar{\dot{\beta}} | \dot{\theta}_{\ex}) = 0$. Revisiting the equality (\ref{H}) we have:
\begin{align*}
    {\rm LHS} =& \frac{1}{2 \r (\levelL + \ddCox)} (\dot{\mu} - \dot{\lambda}_{\Y} | \dot{\mu} + \dot{\lambda}_{\Y} + 2 \dot{\rho}) \\
                =& \frac{1}{2 \r (\levelL + \ddCox)} \left((\bar{\dot{\beta}} + \dot{b}\dot{\theta}_{\ex} | \dot{\mu} + \dot{\lambda}_{\Y} + 2 \dot{\rho}) - (\dot{b}\dot{\theta}_{\ex} | \dot{\mu} + \dot{\lambda}_{\Y} + 2 \dot{\rho})\right) \\
                \le& \frac{-\dot{b}}{2 (\levelL + \ddCox)} (\dot{\theta}_{\ex} | \dot{\mu} + \dot{\lambda}_{\Y} + 2 \dot{\rho})
                \le \frac{-\dot{b}}{2 (\levelL + \ddCox)} (\levelL + \levelL + 2\ddCox) = -\dot{b} = {\rm RHS}.
\end{align*}
According to above discussion, the equality (\ref{H}) implies that $\bar{\dot{\beta}} + \dot{b}\dot{\theta}_{\ex} = 0$, and either $\dot{b} = 0$ or $(\bar{\dot{\beta}} | \dot{\theta}_{\ex}) = 0$, whence $(\bar{\dot{\beta}} | \dot{\theta}_{\ex}) = 0$ can again imply $\dot{b} = 0$ together with the former condition. 
Consequently we have $\bar{\dot{\beta}} = 0$ and $ \dot{b}= 0$. That is, $\dot{\mu} = \dot{\lambda}_{\Y}$ and ${\sf d}_v = {\sf d}_{\Y}$.
In other words, up to diagram automorphisms every $(\lghat_{\lleft}, \lgdhat)$-highest weight vector must have weight $(\lambda_{\Y}, \dot{\lambda}_{\Y}; {\sf d}_{\Y})$ for some $Y\in \P^{l}_{r}$.

Finally we consider the multiplicity. 
Based on the above discussion that ${\sf pr}(\dot{\zeta}_i) = 0$, we claim that $\dot{\zeta}_i = 0$ for each $i$.
Write $\dot{\zeta}_i = \dot{\mu}{}^{\infty}_i - \dot{\lambda}{}^{\infty}_{\Y} = \sum\limits_{(p, \gamma)} c_{(p, \gamma)} \dot{\epsilon}_{(p, \gamma)}$.
From ${\sf pr}(\dot{\zeta}_i) = 0$ we have the equations:
        \[
            \begin{cases}
                -c_{(-p,0)} + \sum\limits_{\gamma<0} \left( c_{(p,\gamma)} - c_{(-p,\gamma)} \right) = 0, \mbox{ for all } p>0;\\
                \sum\limits_{\gamma<0} \gamma \left(\sum\limits_{p>0} (c_{(p,\gamma)} + c_{(-p,\gamma)}) + c_{(0,\gamma)}\right) = 0.
            \end{cases}
        \]
 If $h = 0$, the choice of $Y\in \P^l_r$ implies that the weight $\dot{\lambda}{}^{\infty}_{\Y}$ concentrates in $(p,0)$ with $p<0$. That is, $c_{(p, \gamma)} \ge 0$ for $\gamma < 0$.
        Then the above equations force $\dot{\zeta}_i = 0$.
If $h = \frac{1}{2}$, analogously we have that $\dot{\lambda}{}^{\infty}_{\Y}$ concentrates in $(p, -\frac{1}{2})$ with $p>0$. That is, $c_{(p,\gamma)} \ge 0$ for $\gamma \ne -\frac{1}{2}$ or $\gamma = -\frac{1}{2}, p\le 0$.
        Again from the above equations we have $\dot{\zeta}_i = 0$.

As a conclusion $\dot{\mu}{}^{\infty}_i = \dot{\lambda}{}^{\infty}_{\Y}$ for each $i$. 
That is, $v$ and $v^h_Y$ lie in the same $(\lgo_{\lleft}, \lgdhat{}^{\infty})$-weight space of the module generated by $v^h_Y$, forcing $v = v^h_Y$ up to a scalar. This completes the proof of our main duality theorem.
\end{proof}

\section{Moving vectors and defects of weights}

Recall the classical theory of representation of type $\lslhat(l)$, which categorifies the Ariki-Koike algebras.
For $\Lambda \in \tilde{P}^+(\levelR)$, $\beta \in Q^+$ such that $\Lambda - \beta\in P(\Lambda)$, the defect of the block $\mathcal{H}^{\Lambda}_{\beta}$, 
$$
\defect(\Lambda, \beta) = \frac{1}{2}\big((\Lambda|\Lambda)-(\Lambda-\beta\mid \Lambda-\beta)\big)= (\Lambda|\beta) - \frac{1}{2}(\beta | \beta),
$$
has a nice combinatorial interpretation, see \cite{JL,LQ,LT,LQT} for details.
Consider an abacus corresponding to $(\Lambda, \beta)$. 
Drawing it vertically, one can operate a move on the abacus by pushing a bead upwards one slot. 
The total amount of such available moves equals precisely the defect of the block.
Inspired by this, we shall next define the vertical moves on abacus of arbitrary types.

In this section we shall fix a dual pair $(\lghat_{\lleft}, \lgdhat)$ and $h \in \{ 0, \frac{1}{2} \}$. Thus $L_{\bar{a}},R_{\bar{a}},a=0,1$ are accordingly fixed. For type $A^{(1)}$, the main result  Theorem \ref{thm-def-weight} below in this section was proved \cite{JL} (see also \cite[Lemma 4.1.15]{LQ}) combinatorially. For the sake of brevity, we only consider non-$A^{(1)}$ types. However, the method in this section can also be applied to type $A^{(1)}$.

\subsection{Move and moving vector}\label{moving-vector}
\begin{Def}
For $v(\M) \in \F$ a monomial basis element given by a Maya diagram $\M$, we say that a (vertical) move on $v(\M)$ is one of the following operations on $\dot{\B}(\dot{\M})$:
\begin{enumerate} 
    \item pushing up a bead upwards one slot to an empty position;
    \item on a half column, removing a bead at the most top position.
\end{enumerate}
    Record each move in the following ways:
\begin{enumerate}[label = {\rm (\roman*)}]
    \item if a move of type $(1)$ sends a bead into a row indexed by $p>0$, or pushes a bead out of a row indexed by $-p<0$, we record the move by the symbol ${\vec{e}}_p$;
    \item if a move of type $(1)$ sends a bead into a row indexed by $-1$, or pushes a bead out of a row indexed by $1$, we record the move by the symbol ${\vec{e}}_0$;
    \item if a move of type $(2)$ removes a bead at a row indexed by $-r$(resp. $1$), and in the column abacus, 
    \begin{itemize}
        \item there is a row indexed by $0$ (resp. $\gh{0}$), we record the move by the symbol ${\vec{e}}_r$ (resp. ${\vec{e}}_0$);
        \item there is no row indexed by $0$ (resp. $\gh{0}$), we record the move by the symbol $\frac{1}{2}{\vec{e}}_r$ (resp. $\frac{1}{2}{\vec{e}}_0$).
    \end{itemize}
\end{enumerate}

\medskip 
For a series of moves $w_1{\vec{e}}_{p_1}, \ldots, w_n{\vec{e}}_{p_n}$, where each $w_k \in \{1,\frac{1}{2}\}$, we record the total move as the formal sum $\sum_{k = 1}^n w_k{\vec{e}}_{p_k} \in \Span_{\frac{1}{2}\mathbb{Z}}\{{\vec{e}}_p \mid p = 0, 1, \ldots, r\}$.
\end{Def}

\begin{Remark}\label{remark-half-move}
    $(1)$ Flipping back to row abacus via the Uglov map \S\ref{Uglov-map}, the symbol ${\vec{e}}_p$ records the row on which the move happens.
    That is, as an operation on the row abacus, each move ${\vec{e}}_p, p = 1, 2, \ldots, r$ brings beads into the $p$-th row, while the move ${\vec{e}}_0$ carries beads out of the $1$-th row.

    $(2)$ In the sense of Remark \ref{double-half-row}, the symbols $\frac{1}{2}{\vec{e}}_p, p = 0, r$ are reasonable when we duplicate the half columns and consider on $\mathsf{D\dot{B}}(\M)$. In these cases, the given operation corresponding to only one move on a full column. Since the duplicated half column is $\frac{1}{2}$-weighted, a single move on this column is recorded as $\frac{1}{2}{\vec{e}}_p$.
    Note that the move $\frac{1}{2}{\vec{e}}_r$ exists only if $R_{\bar{0}}$ even, and $\frac{1}{2}{\vec{e}}_0$ exists only if $R_{\bar{1}} = 0$.
\end{Remark}

\begin{Exam}
    Consider the dual pair $(\lsohat(7,1), \lsohat(5))$ with $h = \frac{1}{2}$. On the following column abacus, we can perform a move on each column as indicated by the arrows:
    
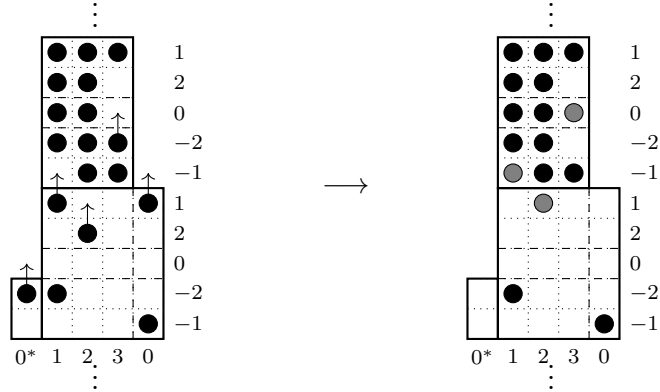
\begin{center}
    \begin{tikzpicture}[scale=0.4]
    \begin{scope}[shift={(-10,0)}]
	\draw[dotted] (0,0) grid[step=1] (4,-5);
	\draw[thick] (0,0) rectangle (4,-5);
	\draw[dashed] (3,0) --++(0,-5);
	\draw[dashed] (4,0) --++(0,-5);
	\draw[dashed] (0,-2) --++(4,0);
	\draw[dashed] (0,-3) --++(4,0);
	\draw[dotted] (-1,-3) grid[step=1] (0,-5);
	\draw[thick] (-1,-3) rectangle (0,-5);
	\node[right] at (4,-0.5) {${\scriptstyle 1}$};
	\node[right] at (4,-1.5) {${\scriptstyle 2}$};
	\node[right] at (4,-2.5) {${\scriptstyle 0}$};
	\node[right] at (4,-3.5) {${\scriptstyle -2}$};
	\node[right] at (4,-4.5) {${\scriptstyle -1}$};
	\draw[fill=black] (0.5,-0.5) circle (0.3);
    \draw[->] (0.5,-0.5) -- (0.5,0.5);
	\draw[fill=black] (3.5,-0.5) circle (0.3);
    \draw[->] (3.5,-0.5) -- (3.5,0.5);
	\draw[fill=black] (1.5,-1.5) circle (0.3);
    \draw[->] (1.5,-1.5) -- (1.5,-0.5);
	\draw[fill=black] (-0.5,-3.5) circle (0.3);
    \draw[->] (-0.5,-3.5) -- (-0.5,-2.5);
	\draw[fill=black] (0.5,-3.5) circle (0.3);
	\draw[fill=black] (3.5,-4.5) circle (0.3);
    \node[below] at (-0.5,-5) {${\scriptstyle \gh{0}}$};
    \node[below] at (0.5,-5) {${\scriptstyle 1}$};
    \node[below] at (1.5,-5) {${\scriptstyle 2}$};
    \node[below] at (2.5,-5) {${\scriptstyle 3}$};
    \node[below] at (3.5,-5) {${\scriptstyle 0}$};
    \node[below] at (1.75,-5){$\vdots$};
	\end{scope}
	\begin{scope}[shift={(-10,5)}]
	\draw[dotted] (0,0) grid[step=1] (3,-5);
	\draw[thick] (0,0) rectangle (3,-5);
	\draw[dashed] (0,-2) --++(3,0);
	\draw[dashed] (0,-3) --++(3,0);
	\draw[fill=black] (0.5,-0.5) circle (0.3);
	\draw[fill=black] (1.5,-0.5) circle (0.3);
	\draw[fill=black] (2.5,-0.5) circle (0.3);
	\draw[fill=black] (0.5,-1.5) circle (0.3);
	\draw[fill=black] (1.5,-1.5) circle (0.3);
	\draw[fill=black] (0.5,-2.5) circle (0.3);
	\draw[fill=black] (1.5,-2.5) circle (0.3);
	\draw[fill=black] (0.5,-3.5) circle (0.3);
	\draw[fill=black] (1.5,-3.5) circle (0.3);
	\draw[fill=black] (2.5,-3.5) circle (0.3);
    \draw[->] (2.5,-3.5) -- (2.5,-2.5);
	\draw[fill=black] (1.5,-4.5) circle (0.3);
	\draw[fill=black] (2.5,-4.5) circle (0.3);
	\node[right] at (4,-0.5) {${\scriptstyle 1}$};
	\node[right] at (4,-1.5) {${\scriptstyle 2}$};
	\node[right] at (4,-2.5) {${\scriptstyle 0}$};
	\node[right] at (4,-3.5) {${\scriptstyle -2}$};
	\node[right] at (4,-4.5) {${\scriptstyle -1}$};
    
    \node[above] at (1.75,0){$\vdots$};
	\end{scope}

    \node at (0,0) {$\longrightarrow $};

    \begin{scope}[shift={(5,0)}]
	\draw[dotted] (0,0) grid[step=1] (4,-5);
	\draw[thick] (0,0) rectangle (4,-5);
	\draw[dashed] (3,0) --++(0,-5);
	\draw[dashed] (4,0) --++(0,-5);
	\draw[dashed] (0,-2) --++(4,0);
	\draw[dashed] (0,-3) --++(4,0);
	\draw[dotted] (-1,-3) grid[step=1] (0,-5);
	\draw[thick] (-1,-3) rectangle (0,-5);
	\node[right] at (4,-0.5) {${\scriptstyle 1}$};
	\node[right] at (4,-1.5) {${\scriptstyle 2}$};
	\node[right] at (4,-2.5) {${\scriptstyle 0}$};
	\node[right] at (4,-3.5) {${\scriptstyle -2}$};
	\node[right] at (4,-4.5) {${\scriptstyle -1}$};
	\draw[fill=gray] (0.5,0.5) circle (0.3);
	\draw[fill=gray] (1.5,-0.5) circle (0.3);
	\draw[fill=black] (0.5,-3.5) circle (0.3);
	\draw[fill=black] (3.5,-4.5) circle (0.3);
    \node[below] at (-0.5,-5) {${\scriptstyle \gh{0}}$};
    \node[below] at (0.5,-5) {${\scriptstyle 1}$};
    \node[below] at (1.5,-5) {${\scriptstyle 2}$};
    \node[below] at (2.5,-5) {${\scriptstyle 3}$};
    \node[below] at (3.5,-5) {${\scriptstyle 0}$};
    \node[below] at (1.75,-5){$\vdots$};
	\end{scope}
	\begin{scope}[shift={(5,5)}]
	\draw[dotted] (0,0) grid[step=1] (3,-5);
	\draw[thick] (0,0) rectangle (3,-5);
	\draw[dashed] (0,-2) --++(3,0);
	\draw[dashed] (0,-3) --++(3,0);
	\draw[fill=black] (0.5,-0.5) circle (0.3);
	\draw[fill=black] (1.5,-0.5) circle (0.3);
	\draw[fill=black] (2.5,-0.5) circle (0.3);
	\draw[fill=black] (0.5,-1.5) circle (0.3);
	\draw[fill=black] (1.5,-1.5) circle (0.3);
	\draw[fill=black] (0.5,-2.5) circle (0.3);
	\draw[fill=black] (1.5,-2.5) circle (0.3);
	\draw[fill=black] (0.5,-3.5) circle (0.3);
	\draw[fill=black] (1.5,-3.5) circle (0.3);
	\draw[fill=gray] (2.5,-2.5) circle (0.3);
	\draw[fill=black] (1.5,-4.5) circle (0.3);
	\draw[fill=black] (2.5,-4.5) circle (0.3);
	\node[right] at (4,-0.5) {${\scriptstyle 1}$};
	\node[right] at (4,-1.5) {${\scriptstyle 2}$};
	\node[right] at (4,-2.5) {${\scriptstyle 0}$};
	\node[right] at (4,-3.5) {${\scriptstyle -2}$};
	\node[right] at (4,-4.5) {${\scriptstyle -1}$};
    \node[above] at (1.75,0){$\vdots$};
	\end{scope}
    \end{tikzpicture}
    \end{center}
    Among the five moves, the first three moves on the columns $1, 2, 3$ are of type $(1)$ and recorded as ${\vec{e}}_0$, ${\vec{e}}_1$, ${\vec{e}}_2$ respectively.
    The other two moves on the half columns $0, \gh{0}$ are of type $(2)$. 
    For the move on the half column $0$, since there is a row indexed by $0$ in the column abacus, we record the move as ${\vec{e}}_2$; while for the move on the half column $\gh{0}$, since there is no row indexed by $\gh{0}$ in the column abacus, we record the move as $\frac{1}{2}{\vec{e}}_0$.
    Conclusively, the total move on the given abacus is recorded as $\frac{3}{2}{\vec{e}}_0 + {\vec{e}}_1 + 2{\vec{e}}_2$.
\end{Exam}

Note that vertical moves on the abacus leave the horizontal $\lgo$-weight unchanged, but do change the $\lgdtilde$-weight.
Therefore, we can also study the moves by considering the change of weights they produce.
Define a linear map $\Psi: \Span_{\frac{1}{2}\mathbb{Z}}\{{\vec{e}}_p \mid p = 0, 1, \ldots, r\} \to \lhdtilde{}^*$,
\[\begin{array}{cc}
\begin{array}{rl}
    {\vec{e}}_p \mapsto \dot{\beta}_p = & \dot{\epsilon}_{p} - \dot{\epsilon}_{p+1},  1 \le p < r;\\
    {\vec{e}}_r \mapsto \dot{\beta}_r = & 
    \begin{cases}
    2\dot{\epsilon}_{r},& R_{\bar{0}}= 2r;\\
    \dot{\epsilon}_{r},& R_{\bar{0}} = 2r+1;
    \end{cases}
\end{array}
& \;
\begin{array}{rl}
    {\vec{e}}_0 \mapsto \dot{\beta}_0 = &
    \begin{cases}
    -\dot{\epsilon}_{1} + \dot{\epsilon}_{r} + \dot{\delta},& \lgdhat = \lslhat(r);\\
    -2\dot{\epsilon}_{1} + \dot{\delta},& R_{\bar{1}} = 0;\\
    -\dot{\epsilon}_{1} + \dot{\delta},& R_{\bar{1}} = 1.
    \end{cases}
\end{array}
\end{array}\]
Since $\Psi(\frac{1}{2}{\vec{e}}_p) = \frac{1}{2}\dot{\beta}_p$ for $p = 0, r$, the map $\Psi$ is well-defined.
It is easy to verify that $\Psi$ is injective in each case.
Therefore, through this embedding $\Psi$, we can regard a move as an element in $\lhdtilde{}^*$.

Observe that in the cases that $\lgdtilde$ is not of branching type, these $\dot{\beta}_p$'s are precisely the simple roots of $\lgdtilde$, and therefore a move can be regarded as an operation via the Chevalley generator $\dot{e}_p$.
Otherwise, if $p = 0$ or $r$ is a branching point on the Dynkin diagram, then $\dot{\beta}_p \notin \tilde{\dot{Q}}^+$ and the move cannot be directly interpreted as Lie action.
For these cases, we extend $\tilde{\dot{Q}}^+$ to a larger lattice of other types, such that all moves ${\vec{e}}_p$ as well as $\frac{1}{2}{\vec{e}}_p$ are simple roots of $\lgdtilde_{\ex}$. 

As mentioned in \S\ref{Parity}, the process of taking $\lgdtilde_{\ex}$ is given by replacing each pair of branching points at the head or tail of the Dynkin diagram of $\lgdtilde$ by a double arrow. The direction depends on whether there is a half row in the abacus as well as whether there is a move of type $\frac{1}{2}{\vec{e}}_p$, as shown in Table \ref{table-double-arrow}.
Note that although the first row and the last row of the table correspond to the same $\lgdhat_{\ex}$, only in the last row the moves $\frac{1}{2}{\vec{e}}_p$ can happen. 
In these cases, we may also think that the roots $\dot{\epsilon}_{r}$ and $-\dot{\epsilon}_{1} + \frac{1}{2}\dot{\delta}$ are $\frac{1}{2}$-weighted. 

\begin{table}[!htbp]
\centering
\begin{tabular}{|l|c||l|c|}
\hline
\multicolumn{1}{|c|}{}                  & head    &      & tail   \\ 
\hline
$R_{\bar{1}}$ odd                         &
\makecell{outward\\
\begin{tikzpicture}[scale=1.2]
    \draw[line width=.5pt, double distance=2pt,->-] (1,0) node[dot]{} node[below]{${\scriptstyle 1}$} -- (0,0) node[dot]{}node[below]{${\scriptstyle 0}$};
\end{tikzpicture} 
}
&
$R_{\bar{0}}$ odd                         &
\makecell{outward\\
\begin{tikzpicture}[scale=1.2]
    \draw[line width=.5pt, double distance=2pt,->-] (4,0) node[dot]{} node[below]{${\scriptstyle l-1}$} -- (5,0) node[dot]{}node[below]{${\scriptstyle l}$};
\end{tikzpicture} 
}
\\ 
\hline
$R_{\bar{1}}$ even, $L_{\overline{2h'}}$ even    &
\makecell{inward\\
\begin{tikzpicture}[scale=1.2]
    \draw[line width=.5pt, double distance=2pt,->-] (0,0) node[dot]{} node[below]{${\scriptstyle 0}$} -- (1,0) node[dot]{}node[below]{${\scriptstyle 1}$};
\end{tikzpicture} 
}
&
$R_{\bar{0}}$ even, $L_{\overline{2h}}$ even     &
\makecell{inward\\
\begin{tikzpicture}[scale=1.2]
    \draw[line width=.5pt, double distance=2pt,->-] (5,0) node[dot]{} node[below]{${\scriptstyle l}$} -- (4,0) node[dot]{}node[below]{${\scriptstyle l-1}$};
\end{tikzpicture} 
}
\\ 
\hline
$R_{\bar{1}}$ even, $L_{\overline{2h'}}$ odd    & 
\makecell{outward\\
\begin{tikzpicture}[scale=1.2]
    \draw[line width=.5pt, double distance=2pt,->-] (1,0) node[dot]{} node[below]{${\scriptstyle 1}$} -- (0,0) node[dot]{}node[below]{${\scriptstyle 0}$};
\end{tikzpicture} 
}
& 
$R_{\bar{0}}$ even, $L_{\overline{2h}}$ odd     & 
\makecell{outward\\
\begin{tikzpicture}[scale=1.2]
    \draw[line width=.5pt, double distance=2pt,->-] (4,0) node[dot]{} node[below]{${\scriptstyle l-1}$} -- (5,0) node[dot]{}node[below]{${\scriptstyle l}$};
\end{tikzpicture} 
}
\\ 
\hline
\end{tabular}
\caption{The replacement rule of the branching point on the Dynkin diagram of $\lgdhat$}\label{table-double-arrow}
\end{table}

Detailed choices of $\lgdhat_{\ex}$ for each dual pair are given in Table \ref{table-g-ex}.
In each case the assignment $\dot{\epsilon}_p \mapsto \dot{\epsilon}_p$ induces an embedding $\tilde{\dot{Q}}^+ \hookrightarrow \tilde{\dot{Q}}^+_{\ex}$.
In particular, one can routinely check that $\tilde{\dot{P}}^+_{\ex}(\levelR) \subset \tilde{\dot{P}}^+_{\dot{\Sigma}}(\levelR)$.
Taking $\tilde{\dot{Q}}^+_{\ex}:=\tilde{\dot{Q}}^+$ in non-braching cases, we can uniformly conclude that the moves corresponding to the simple roots in $\tilde{\dot{Q}}^+_{\ex}$.  

\begin{table}[!htbp]
\centering
\renewcommand{\arraystretch}{1.4}
\begin{tabular}{|c|c||c|c|}
\hline
Dual pair $(\lghat, \lgdhat)$   &   $\lgdhat_{\ex}$  &   Dual pair $(\lghat, \lgdhat)$   &   $\lgdhat_{\ex}$  \\ 
\hline
$(\lsohat(2l), \lsohat(2r))$    &   $\lsphat(2r)$      &   $(\lsohat(2l), \lsohat(2r+1))$  &   $\lslhat{}^{(2)}(2r+1)$  \\
\hline
$(\lsohat(2l+1), \lsohat(2r)), h = 0$  &    $\lslhat{}^{(2)}(2r+1)$ &   $(\lsohat(2l+1), \lsohat(2r)), h = \frac{1}{2}$ &   ${}^{t}\!\!\lslhat{}^{(2)}(2r+1)$ \\
\hline
$(\lsohat(2l+1), \lsohat(2r+1)), h = 0$  &    $\lslhat{}^{(2)}(2r+1)$ &   $(\lsohat(2l+1), \lsohat(2r+1)), h = \frac{1}{2}$ &   $\lsohat(2r+1, 1)$ \\
\hline
$(\lsohat(2l+1,1), \lsohat(R_{\bar{0}}))$&    $\lsohat(2r+1,1)$     &   $(\lglhat{}^{(2)}(2l), \lslhat{}^{(2)}(2r))$    &   $\lsphat(2r)$ \\
\hline
    &            &   $(\lglhat{}^{(2)}(2l+1), \lslhat{}^{(2)}(2r))$    &   $\lslhat{}^{(2)}(2r+1)$ \\
\hline
\end{tabular}
\caption{The choice of $\lgdhat_{\ex}$ for each dual pair}\label{table-g-ex}
\end{table}

Next we give the definition of moving vector, which records the total move from one monomial basis element to another.

\begin{Def}
    Suppose $v, v'$ are two monomial basis elements that $v'$ can be obtained by $v$ via successive moves $w_1{\vec{e}}_{p_1}, \ldots, w_n{\vec{e}}_{p_n}$ with $w_k \in \{1,\frac{1}{2}\}, k = 1, 2, \ldots, n$. 
    We define the moving vector ${\bf mv}(v,v')$ from $v$ to $v'$ to be the coefficient vector of the formal sum $\sum_{k = 1}^n w_k{\vec{e}}_{p_k}$. 
    That is, ${\bf mv}(v,v') = (m_0,\ldots, m_r)\in (\frac{1}{2}\mathbb{N})^{r+1}$ with $\sum_{k = 1}^n w_k{\vec{e}}_{p_k} = (m_0,\ldots, m_r)\cdot({\vec{e}}_{0}, \ldots, {\vec{e}}_{r})^T$.

    In particular, if there is no more available move on $v'$, we write $v_{\rm max}:=v'$ the highest state that $v$ can move to. In this case we abbreviate ${\bf mv}(v):= {\bf mv}(v,v_{\rm max})$ and call it the moving vector of $v$. This array of components records the process pushing all beads upwards as high as possible.
\end{Def}

Under this definition ${\bf mv}(v,v')$ is determined only by the vertical weights viewed on column abaci of $v$ and $v'$, independent of the order or route of the moves. 
Note that the sum ${\bf mv}(v,v')\cdot \mathbf{1} = \sum_{k = 1}^n w_k$ records the number of moves we need to reach $v'$ from $v$.
This is also a weighted height of $\sum_{k = 1}^n w_k \dot{\beta}_{p_k}$ in $\tilde{\dot{Q}}^+_{\ex}$, denoted by ${\rm ht}_{\ex}(\sum_{k = 1}^n w_k \dot{\beta}_{p_k}) = \sum_{k = 1}^n w_k$.

By the definition of column abacus we know that all possible moves on any abacus will terminate in finite steps.
We study the properties of such terminal abacus via the combinatorial model. 
As mentioned above, it is easy to verify that during each move, the horizontal weight of the abacus, as well as the charge sequence of each column remains unchanged.
Suppose $v$ is a monomial basis element in $\F$ of weight $(\lambda, \dot{\mu}; -k)$, it is easy to see that its terminal abacus $v_{\max}$ is precisely the vacuum abacus with the same charge sequence $\sum_{i=1}^l (b_i - b_{-i})\epsilon_i = \lambda - \frac{R_{\overline{2h}}}{2}\mathbf{1}$ as $v$ (see Remark \ref{def-abacus}).
Note that $v_{\max}$ is independent of $\dot{\mu}, -k$. That is, for any abacus of horizontal weight $\lambda$, its terminal abacus is the same vacuum abacus defined by the charge sequence $\lambda - \frac{R_{\overline{2h}}}{2}\mathbf{1}$.

The above combinatorial discussion provides us a way to characterize $v_{\max}$.
Suppose $v \in \F_{(\lambda, \dot{\mu}; -k)}$ is a monomial basis element, then $v_{\max}$ has the unique maximal vertical $\lgdtilde_{\ex}$-weight among the vectors with the same horizontal $\lgo$-weight as $v$.
In fact, Lemma \ref{Pmax} below also proves this through the Lie theoretic perspective.
To state this characterization precisely, for $\lambda' \in P, \dot{\mu}' \in \dot{P}$, we define the sections
\begin{align*}
    P_{(\lambda', -)}:=& \{ \dot{\nu} - k \dot{\delta} \in \tilde{\dot{P}} \mid \F_{(\lambda', \dot{\nu}; -k)} \ne 0 \}, \\
    P_{(-, \dot{\mu}')}:=& \{ \nu - k \delta \in \tilde{P} \mid \F_{(\nu, \dot{\mu}'; -k)} \ne 0 \}.
\end{align*}
Then $v_{\max}$ is of weight $(\lambda, \dot{\mu}_{\max}; -k_{\max})$ with $\dot{\mu}_{\max} -k_{\max}\dot{\delta} = \max_{\dot{Q}_{\ex}} P_{(\lambda, -)}$, 
and therefore the moving vector corresponds to $\Psi({\bf mv}(v)) = (\dot{\mu}_{\max} -k_{\max}\dot{\delta}) - (\dot{\mu} - k\dot{\delta})$.
This observation shows that all monomial basis elements lying in the same weight space $\F_{(\lambda, \dot{\mu}; -k)}$ have the same moving vector.

Take $\Lambda \in \tilde{P}^+(\levelR)$, $\beta = \bar{\beta} + b\delta \in \tilde{Q}^+$ such that $\Lambda - \beta \in P(\Lambda)$. 
Suppose the dual weight pair corresponding to $\Lambda$ is $(\bar{\Lambda}, \dot{\bar{\Lambda}})$ with $\dot{\bar{\Lambda}}\in P^+_{\dot{\Sigma}}(\levelL)$, so that the joint highest weight vector $v_{\Lambda}$ is in $\F_{(\bar{\Lambda}, \dot{\bar{\Lambda}}; -{\sf d}_{\Lambda})}$.
Starting from $v_{\Lambda}$, after applying a series of $\lghat_{\lleft}$ operators we can get $v_0 \in \F_{(\bar{\Lambda} - \bar{\beta}, \dot{\bar{\Lambda}}; -{\sf d}_{\Lambda} - b)}$.
Let $v$ be any monomial summand of $v_0$. Then we call the abacus corresponding to $v$ a reduced $\lghat_{\lleft}$-abacus of weight $\Lambda - \beta$.
In other words, a reduced abacus of weight $\Lambda - \beta$ is any abacus of weight $(\bar{\Lambda} - \bar{\beta}, \dot{\bar{\Lambda}}; -{\sf d}_{\Lambda} - b)$.
It is remarkable here that for type $A^{(1)}$, the condition added on vertical weight to be $\dot{\bar{\Lambda}}$ is equivalent to requiring the charge sequence lying in $\mathcal{A}^r_e$ (see \cite[pp.16]{LQ}, for example). The moving vector ${\bf mv}(v)$ remains unchanged for a different choice of $v$, and depends only on $\Lambda, \beta$. This means the moving vector is a block invariant.
  
\begin{Def}
    We define the moving vector ${\bf mv}(\Lambda, \beta)$ associate to the weight $\Lambda - \beta$ to be the moving vector of  any reduced $\lghat$-abacus of weight $\Lambda-\beta$.
\end{Def}

    Moreover, Since $\F_{(\lambda, \mu, -k)} \cong \F_{w(\lambda, \mu, -k)}$ for $w$ in the affine Weyl group,  we see that all weights lying in a  Weyl group orbit have the same moving vector.

Concluding the process, the relations between $\Lambda, \beta$ and ${\bf mv}(\Lambda, \beta)$ can be roughly interpreted by the following diagram:
\begin{center}
$
\xymatrix{
        & & (\bar{\Lambda} - \bar{\beta}, \dot{\mu}_{\max}; -k_{\max}) \\
    (\bar{\Lambda}, \dot{\bar{\Lambda}}; -{\sf d}_{\Lambda})\ar[rr]^(.45){-\beta = -\bar{\beta} - b\delta}_(.45){\lghat_{\lleft} \text{-action}}   & &  (\bar{\Lambda} - \bar{\beta}, \dot{\bar{\Lambda}}; -{\sf d}_{\Lambda} - b) \ar@{-->}[u]^{{\bf mv}(\Lambda, \beta)}_{\text{move}}
}
$
\end{center}

\subsection{Orders of partitions and weights}

According to the discussion above, the weight of the terminal abacus is the unique maximal weight in the section $P_{(\lambda, -)}$ with respect to the order ${\dot{Q}_{\ex}}^+$.
In order to study the moving vector through the weights, first we need some lemmas analysing the order of weights.
Recall from Theorem \ref{mainDualThmIntro} that, up to graph automorphisms, the dominant weights we concern are given by the partitions and their Young diagrams, thus the order of weights is closely related to the (weak) dominance order of partitions.

We make the following convention. When the root lattice is clear and the order is with respect to ${Q}^+$, ${\dot{Q}}^+$, $\tilde{Q}^+$ or $\tilde{\dot{Q}}^+$, we simply write $>$ and omit the specification; for example, if $\lambda, \mu \in \dlho$, then $\lambda > \mu \Leftrightarrow \lambda - \mu \in {Q}^+$.
On the other hand, when the order is with respect to ${\dot{Q}}^+_{\ex}$ or $\tilde{\dot{Q}}^+_{\ex}$, we will explicitly write $>_{\ex}$. 

Recall the definition of the  dominance order of partitions: for two partitions $Y = (y_1, y_2, \ldots)$ and $Z = (z_1, z_2, \ldots)$, we say $Y$ dominates $Z$, denoted by $Y \succeq Z$, if for all $k \ge 1$, the $k$-th prefix sum $\sum_{i=1}^k y_i \ge \sum_{i=1}^k z_i$; and we say $Y$ strictly dominates $Z$, denoted by $Y \succ Z$, if $Y \succeq Z$ and $Y \ne Z$.

Recall from the formulae (\ref{weightY}) and (\ref{weightY-A}) the relations between the partitions and the weights of Maya diagrams, we have the following results. 

\begin{Lem}\label{order-to-order}
    Let $Y,Z\in \P^l_r$ be two partitions, whose corresponding Maya diagrams of the Young diagrams are of weights $(\lambda_{\Y}, \dot{\lambda}_{\Y})$ and $(\lambda_{\Z}, \dot{\lambda}_{\Z})$, respectively.
    \begin{enumerate}
        \item For $h = 0$, $\lambda_{\Y} > \lambda_{\Z}$ implies $Y^c \succ Z^c$; 
        for $h = \frac{1}{2}$, $\lambda_{\Y} > \lambda_{\Z}$ implies $Y \succ Z$.
        \item Suppose $\dot{\lambda}_{\Y} - \dot{\lambda}_{\Z} \in \dot{Q}$. Then $\dot{\lambda}_{\Y} >_{\ex} \dot{\lambda}_{\Z}$ if and only if $Y^t \succ Z^t$.
    \end{enumerate}
\end{Lem}
\begin{proof}
We first consider the prefix sum $\sum\limits_{i=1}^k c_i$ for  a simple root $\alpha_j = \sum\limits_{i=1}^l c_i \epsilon_i$ in a root system of finite classical type of rank $l$.
It is straightforward to check that the prefix sums are all non-negative. Moreover, if $j$ is not a branching point on the Dynkin diagram, then $\sum\limits_{i=1}^k c_i$ is non-zero only for $k=j$. 

(1) For $h=0$, by formula (\ref{weightY}), $\lambda_{\Y}-\lambda_{\Z}=Y^c-Z^c$ belongs to $Q^+$. Therefore, all prefix sums of $Y^c-Z^c$ is non-negative. Hence $Y^c\succ Z^c$. The case $h=\frac{1}{2}$ is similar.

    (2) The necessity  is similar to (1). For the sufficiency, suppose $Y^t \succ Z^t$. It follows from formula (\ref{weightY}) that  $\dot{\beta} = \dot{\lambda}_{\Y} - \dot{\lambda}_{\Z} = Y^t - Z^t \in \dot{Q}$.
    Writing it as linear combinations of simple roots in $\dot{Q}_{\ex}$ or the standard basis $\dot{\epsilon}_p$ respectively, we have $\dot{\beta} = \sum\limits_{p = 1}^r n_p \dot{\beta}_p = \sum\limits_{p = 1}^r b_p \dot{\epsilon}_p$.
    Since the Dynkin diagram of $\dot{Q}_{\ex}$ contains no branching points, each $\dot{\beta}_p$ has only its $p$th prefix sum non-zero. Thus, for all $p=1,2,\ldots,r$, the $p$th prefix sum of $\dot{\beta}$ is $n_p$ times the $p$th prefix sum  of $\dot{\beta}_p$. Since $\dot{\beta}\succ 0$, we deduce that $n_p$'s are   all non-negative and not all $0$, forcing $\dot{\beta} \in \dot{Q}_{\ex}^+$.
\end{proof}

In the following discussion, we shall consider the pair of dual weights $(\lambda, \dot{\lambda}) \in P^+(\levelR) \times \dot{P}^+(\levelL)$ such that $V_{\lghat}(\lambda;\levelR) \otimes V_{\lgdhat}(\dot{\lambda};\levelL)$ is a summand of $\F$.
That is, there exists a partition $Y_{\lambda} \in \P^l_r$ such that $\lambda \in \Sigma(\lambda_{\Y_{\lambda}})$ and $\dot{\lambda} \in \dot{\Sigma}(\dot{\lambda}_{\Y_{\lambda}})$.
In particular, if we make the further assumption that $(\lambda, \dot{\lambda}) \in \sigdom(\levelR) \times \dot{P}^{+}_{\dot{\Sigma}}(\levelL)$, then we have $\lambda = \lambda_{\Y_{\lambda}}, \dot{\lambda} = \dot{\lambda}_{\Y_{\lambda}}$, and the joint highest weight vector of the summand is $v(\Y_{\lambda})$.
In this case, the pair $(\lambda, \dot{\lambda})$ is in $1-1$ correspondence, with the corresponding Young diagrams in conjugation or complement with each other.

The following lemma reveals how the dominance order of partitions behaves under taking conjugates or conjugate complements (see \S\ref{Parity}). 

\begin{Lem}\label{reverse-dom-order}
    Let $Y,Z\in \P^l_r$ be two partitions  and let $n \in \mathbb{N}$.
    \begin{enumerate}
        \item If $n \ge \size(Y)$, and $Y + (n - \size(Y)) \epsilon_1 \succeq Z$, then $n \ge \size(Z)$, and $Z^t + (n - \size(Z)) \dot{\epsilon}_1 \succeq Y^t$.
        In particular, if $\size(Y)=\size(Z)$, then $Y \succeq Z$ if and only if $Z^t \succeq Y^t$.
        \item If $Y + n\epsilon_1 \succeq Z$, then $Z^{\dagger} + n\dot{\epsilon}_1 \succeq Y^{\dagger}$. In particular, $Y \succeq Z$ if and only if $Z^{\dagger} \succeq Y^{\dagger}$.
    \end{enumerate}
\end{Lem}
\begin{proof}
    We only prove (1) since (2) can be proved similarly.

    From $Y + (n - \size(Y))\epsilon_1 \succeq Z$, taking sizes on both sides gives $n \ge \size(Z)$. Consider the $s$th prefix sum of $Y^t$ and $Z^t + (n -\size(Z)) \dot{\epsilon}_1$ for $s = 1, \ldots, r$,  
    \[
        \sum_{p=1}^s y^t_p = \size(Y) - \sum_{i=1}^l \max\{0, y_i - s\},\quad 
        \sum_{p=1}^s (z^t_p + \delta_{p1}(n - \size(Z)) = n - \sum_{i=1}^l \max\{0, z_i - s\}.
    \]
    Define for each integer $s$ the non-decreasing convex function $\phi_s(t) = \max\{0,t-s\}$. By Karamata's inequality \cite[Theorem 1.5.2]{HNShi}, if partition $(a_1,a_2,\ldots,a_l)$ dominates partition $(b_1,b_2,\ldots,b_l)$, then $\sum_{i=1}^l\phi_s(a_i)\ge \sum_{i=1}^l\phi_s(b_i)$. Applying this to $Y + (n -\size(Y)) \epsilon_1 \succeq Z$ gives 
    \[
        \sum_{i=1}^l \max\{0, y_i + \delta_{i1}(n - \size(Y))- s\} \ge \sum_{i=1}^l \max\{0, z_i - s\}.
    \]
    Since $n-\size(Y)\geq 0$, we have  $\max\{0, y_1-s\} + n - \size(Y) \ge \max\{0, y_1 + n - \size(Y)- s\}    $.
    Rearranging the terms gives
    \[
        \size(Y)- \sum_{i=1}^l \max\{0, y_i - s\} \le n - \sum_{i=1}^l \max\{0, z_i - s\},
    \]
    that is, $Z^t + (n - \size(Z)) \dot{\epsilon}_1 \succeq Y^t$.
\end{proof}

With the above order-reversing conclusion, we obtain some basic properties of dual weight pairs. 
\begin{Lem}\label{dualYoung}
    Let $(\lambda, \dot{\lambda}), (\mu, \dot{\mu}) \in \sigdom(\levelR) \times \dot{P}^{+}_{\dot{\Sigma}}(\levelL)$ be dual weight pairs. 
    \begin{enumerate}
        \item When $h = 0$, $\lambda > \mu$ implies $\dot{\lambda} \ngtr \dot{\mu}$. 
        \item When $h = \frac{1}{2}$, if $\size(Y_\lambda)  = \size(Y_\mu)$, then $\lambda > \mu$ implies  $\dot{\lambda} \ngtr \dot{\mu}$. 
        \item If $\dot{\lambda} - \dot{\mu} \in \dot{Q}_{\ex}$, then there exists $\sigma \in \Sigma$ such that $\sigma(\lambda) - \mu \in Q$. 
    \end{enumerate}
\end{Lem}
\begin{proof}
    Combining Lemmas \ref{order-to-order} and \ref{reverse-dom-order} directly gives (1) and (2).

    Now we prove (3). The choice of the weights together with Remark \ref{parity-in-Fock} implies that $\levelR \Lambda_0 + \lambda$ and $\levelR \Lambda_0 + \mu$ are lying in the same parity set. 
    Recall the root sieving class decomposition in each case from Lemma \ref{root-sieving-eq}.
    For $\lghat = \lsohat(2l+1)$, $\lsohat(2l+1,1)$ or $\lglhat{}^{(2)}(2l+1)$, the whole parity set forms a single root sieving equivalent class, so $\lambda-\mu\in Q$ by the definition of sieving equivalence.     Otherwise $\lghat = \lsohat(2l)$, $\lglhat{}^{(2)}(2l)$ or $\lsphat(2l)$, and $Q$ is generated by $\pm\epsilon_i\pm\epsilon_j$ and contains all integer vectors for which the sum of all components is even.     According to Table \ref{table-g-ex-all}, there are three possibilities for $\lgdhat_{\ex}$, that is, $\lgdhat_{\ex}=\lsphat(2r)$, $\lslhat{}^{(2)}(2r+1)$ or $\lsohat(2r+1,1)$.

 If $\lgdhat_{\ex} = \lsphat(2r)$, then $\dot{\lambda} - \dot{\mu} \in \dot{Q}_{\ex}$ implies $\size(Y_{\lambda}) - \size(Y_{\mu}) \equiv 0 \pmod 2$. This induces that $\lambda - \mu \in Q$. 
    
If $\lgdhat_{\ex} = \lslhat{}^{(2)}(2r+1)$, the corresponding dual pair is $(\lsohat(2l), \lsohat(2r+1))$ or $(\lglhat{}^{(2)}(2l), \lslhat{}^{(2)}(2r+1))$.
The parity set is $\tilde{P}^+(\levelR)_{0,1}$ for $h = 0$ and $\tilde{P}^+(\levelR)_{1,0}$ for $h = \frac{1}{2}$.
            By Lemma \ref{root-sieving-eq} we observe that the two equivalent classes are permuted by $\sigma_{0,1}$ for $h = 0$ and by $\sigma_{l-1, l}$ for $h = \frac{1}{2}$.            That is, we always have a $\Sigma$-conjugation of $\lambda$ which lies in the same $Q$-lattice with $\mu$.
    
If $\lgdhat_{\ex} = \lsohat(2r+1,1)$, the corresponding dual pair is $(\lsohat(2l), \lsohat(2r+1,1))$.
            The parity set is $\tilde{P}^+(\levelR)_{1,1}$, whence again by Lemma \ref{root-sieving-eq} the two equivalent classes are permuted by $\Sigma$, that is, $\sigma(\lambda)-\mu\in Q$ for some $\sigma\in\Sigma$.  
\end{proof}

\begin{Prop}\label{Pmax}
    Let $(\mu, \dot{\mu}) \in P^+(\levelR) \times \dot{P}^+(\levelL)$ be a pair of dual weights in $\F$.
    \begin{enumerate}
        \item The weight $\mu - {\sf d}_{\mu} \delta$ is  maximal in $P_{(-, \dot{\mu})}$ with respect to $\tilde{Q}^+$. 
        \item If $\dot{\mu} \in \dot{P}^{+}_{\dot{\Sigma}}(\levelL)$, then $\dot{\mu} - {\sf d}_{\mu} \dot{\delta}$ is the unique maximal element in $P_{(\mu, -)}$ with respect to $\tilde{\dot{Q}}^+_{\ex}$. 
    \end{enumerate}
\end{Prop}

\begin{proof}
    (1) Recall from Section \ref{DiagAuto}  that each $\dot{\sigma} \in \dot{\Sigma}$ induces an isomorphism of  weight spaces. This implies that $P_{(-, \dot{\mu})} = P_{(-, \dot{\sigma}(\dot{\mu}))}$. 
    Meanwhile, $\Sigma_{\lleft}$ preserves the order with respect to $\tilde{Q}^+$. 
    Therefore we may assume $(\mu, \dot{\mu}) \in \sigdom(\levelR) \times \dot{P}^{+}_{\dot{\Sigma}}(\levelL)$. 

    It is clear from $v(\Y_{\mu})\in \F_{(\mu, \dot{\mu}; -{\sf d}_{\mu})}$ that $\mu - {\sf d}_{\mu}\delta \in P_{(-,\dot{\mu})}$. Suppose, for contradiction, that there exists $\lambda - m\delta > \mu - {\sf d}_{\mu}\delta$ such that $\F_{(\lambda, \dot{\mu}; -m)} \ne 0$.
    Take non-zero $v \in \F_{(\lambda, \dot{\mu}; -m)}\cap (V_{\lghat_{\lleft}}(\nu; \levelR) \otimes V_{\lgdhat}(\dot{\nu}; \levelL))$ for some pair $(\nu,\dot{\nu})$ of dual weights. Again we may assume $(\nu, \dot{\nu}) \in \sigdom(\levelR) \times \dot{P}^{+}_{\dot{\Sigma}}(\levelL)$.
    Then we can go down from the joint highest weight vector $v(\Y_{\nu})$ to $v$ via $g \in U(\lghat_{\lleft})^-$ and $g' \in U(\lgdhat)^-$ successively, that is, $v = gg' v(\Y_{\nu})$, thanks to the commutativity of the actions of $(\lghat, \lgdhat)$. 
    We have $g'v(\Y_{\nu}) \in \F_{(\nu, \dot{\mu}; -k)}$. It follows that $\dot{\nu} - {\sf d}_{\nu}\dot{\delta} \ge \dot{\mu} - k\dot{\delta}$, and $\nu - k \delta \ge \lambda - m\delta > \mu - {\sf d}_{\mu}\delta$, as shown in  the following diagram of weights.
    \begin{center}
    $
    \xymatrix{
        (\nu, \dot{\nu}; -{\sf d}_{\nu})\ar[d]   &   &  \\
        (\nu, \dot{\mu}; -k)\ar[r]    & (\lambda, \dot{\mu}; -m)\ar@{-->}[r]   &    (\mu, \dot{\mu}; -{\sf d}_{\mu}).
    }
    $
    \end{center}
    Now write $\beta = \nu - \mu + (-k + {\sf d}_{\mu})\delta$ and $\dot{\beta }= \dot{\nu} - \dot{\mu} + (-{\sf d}_{\nu} + k )\dot{\delta}$. Since $\beta>0, \,\dot{\beta} \ge 0$, we have ${\sf d}_{\mu} \ge k \ge {\sf d}_{\nu}$.
    The $1-1$ correspondence given by $\sigdom(\levelR) \times \dot{P}^{+}_{\dot{\Sigma}}(\levelL)$ 
    implies that $\dot{\nu}\ne\dot{\mu}$.
    
    Next we compare $\mu$ and $\nu$ using Lemma \ref{dualYoung} and formula (\ref{degreeV}). If $h = 0$, then ${\sf d}_{\mu} = \frac{1}{16}\dim \W^h(0) = {\sf d}_{\nu}$ and ${\sf d}_{\mu} = k = {\sf d}_{\nu}$, thus $\nu > \mu$ and $\dot{\nu} > \dot{\mu}$, contradicting to Lemma \ref{dualYoung} (2).

    If $h = \frac{1}{2}$, then, by formula (\ref{degreeV}), we have 
    ${\sf d}_{\nu} -{\sf d}_{\mu}=\frac{1}{2}(\size(Y_{\nu})-\size(Y_{\mu}))$.  Decomposing $\delta = a_0 \alpha_0 + \theta$ as well as $\dot{\delta} = a_0 \dot{\alpha}_0 + \dot{\theta}$ in $\beta$ and $\dot{\beta}$, 
        removing the $\alpha_0, \dot{\alpha}_0$ parts and multiplying with $\mathbf{1}$, together with the fact that $\theta \cdot \mathbf{1} = 2 = \dot{\theta} \cdot \mathbf{1}$, we obtain
       \[
            \size(Y_{\nu}) - \size(Y_{\mu})+ 2(-k + {\sf d}_{\mu}) \ge 0\quad\text{and}\quad  
            \size(Y_{\nu}) - \size(Y_{\mu})+ 2(-{\sf d}_{\nu} + k) \ge 0.
\]
It follows that  ${\sf d}_{\nu} \ge k \ge {\sf d}_{\mu}$, forcing ${\sf d}_{\mu} = {\sf d}_{\nu}$ since we already have ${\sf d}_{\mu} \ge k \ge {\sf d}_{\nu}$. Then similarly we have $\nu > \mu$ as well as $\dot{\nu} > \dot{\mu}$, which again cause a contradiction. 
    
In summary, it is impossible for $\lambda - m\delta > \mu - {\sf d}_{\mu}\delta$ to appear in $P_{(-, \dot{\mu})}$, 
    which means $\mu - {\sf d}_{\mu} \delta$ is maximal in $P_{(-, \dot{\mu})}$.

    (2) Let $\dot{\lambda}$, $m$ be such that $\F_{(\mu, \dot{\lambda}; -m)} \ne 0$. 
    Now we prove $\dot{\mu} - {\sf d}_{\mu}\dot{\delta} \ge_{\ex} \dot{\lambda} - m\dot{\delta}$.
    Without loss of generality, assume $\dot{\lambda} \in \dot{P}^{+}_{\dot{\Sigma}}(\levelL)$.
    From (1) we know that $\lambda - {\sf d}_{\lambda}\delta \ge \mu - m\delta$ in $P_{(-, \dot{\lambda})}$. Since 
    $\delta = a_0 \alpha_0 + \theta$, we have 
    \[\lambda-\mu+(-{\sf d}_{\lambda}+m)(\theta+a_0\alpha_0)\geq 0.\]
    This implies $a_0(-{\sf d}_{\lambda} + m) \in \mathbb{N}$.  Scaling with $2\epsilon_1 \ge \theta$ we get $\lambda - \mu + 2(-{\sf d}_{\lambda} + m)\epsilon_1 \ge 0$. 
 Using Lemma \ref{order-to-order} to compare the orders of weights and partitions, we get:
    \begin{itemize}
        \item If $h = 0$, then ${\sf d}_{\lambda} = {\sf d}_{\mu} \le m$, and $Y^c_{\lambda} + 2(-{\sf d}_{\lambda} + m)\epsilon_1 \succeq Y^c_{\mu}$.
        From Lemma \ref{reverse-dom-order}(2), $Y^t_{\mu} + 2(-{\sf d}_{\mu} + m)\epsilon_1 \succeq Y^t_{\lambda}$, 
        hence $\dot{\mu} + (-{\sf d}_{\mu} + m)\dot{\theta}_{\ex} \ge_{\ex} \dot{\lambda}$.
        \item If $h = \frac{1}{2}$, then $Y_{\lambda} + 2(-{\sf d}_{\lambda} + m)\epsilon_1 \succeq Y_{\mu}$. 
        From Lemma \ref{reverse-dom-order}(1), $Y^t_{\mu} + 2(-{\sf d}_{\mu} + m)\epsilon_1 \succeq Y^t_{\lambda}$,
        and $m \ge {\sf d}_{\mu}$, hence similarly $\dot{\mu} + (-{\sf d}_{\mu} + m)\dot{\theta}_{\ex} \ge_{\ex} \dot{\lambda}$.
    \end{itemize}
For both cases we have $\dot{\mu} - {\sf d}_{\mu}\dot{\delta} \ge_{\ex} \dot{\lambda} - m\dot{\delta}$.
\end{proof}

According to Lemma \ref{Pmax}, the terminal weights $(\bar{\Lambda} - \bar{\beta}, \mu_{\max}; -k_{\max})$ of the move we want is precisely of the form $w(\lambda, \dot{\lambda}; -{\sf d}_{\lambda})$, 
with $(\lambda, \dot{\lambda}) \in P^{+}(\levelR) \times \dot{P}^{+}_{\dot{\Sigma}}(\levelL)$ a pair of dual weights and $w$ is a Weyl group element such that $\overline{w(\lambda)} = \bar{\Lambda} - \bar{\beta}$.
In other words, up to an action of the affine Weyl group, the terminal $v_{\max}$ of the move  is another joint highest weight vector.

\medskip 
The following corollary of Lemma \ref{Pmax} ensures that the weight spaces we are going to discuss in our reduction process is always non-zero.

\begin{Cor}\label{wtSpaceNz}
    Let $(\mu, \dot{\mu}) \in P^+(\levelR) \times \dot{P}^+(\levelL)$ be a pair of dual weights in $\F$. 
    \begin{enumerate}
    \item Let $\lambda,k$ satisfy $\mu - \lambda \in Q$ and $\F_{(\lambda, \dot{\mu}; -k)} \ne 0$, then there exists non-zero $v\in \F_{(\lambda, \dot{\mu}; -k)}$ such that $v \in \lghat_{\lleft}^- . v(\Y_{\mu})$.
    \item If $\dot{\mu} \in \dot{P}^{+}_{\dot{\Sigma}}(\levelL)$, then the section $P_{(\mu, -)}$ coincides with the weight set $\dot{P}_{\ex}(\dot{\mu} + \levelL\dot{\Lambda}_0 - {\sf d}_{\mu} \dot{\delta})$.
    \end{enumerate}
\end{Cor}
\begin{proof}
    (1) By Lemma \ref{Pmax}(1), $\mu - {\sf d}_{\mu}\delta > \lambda - k\delta$. If $\lambda \in P^+(\levelR)$, then $V_{\lgtilde}(\mu + \levelR \Lambda_0 - {\sf d}_{\mu}\delta)_{\lambda + \levelR \Lambda_0 - k\delta} \ne 0$ according to the convex-hull argument in \cite[Proposition 12.5]{Kac}.
    
    On the other hand, for a general $\lambda$, consider the corresponding dominant weights in the same orbit under affine Weyl group. 
    That is, take a Weyl group element $w$ such that $w(\lambda + \levelR \Lambda_0 - k \delta)\in \tilde{P}^+(\levelR)$. 
    Then again by Lemma \ref{Pmax} we can also get $w(\lambda + \levelR \Lambda_0 - k \delta) < \mu + \levelR \Lambda_0 - {\sf d}_{\mu}\delta$, 
    whence $V_{\lgtilde}(\mu + \levelR \Lambda_0 - {\sf d}_{\mu}\delta)_{\lambda + \levelR \Lambda_0 - k\delta} \cong V_{\lgtilde}(\mu + \levelR \Lambda_0 - {\sf d}_{\mu}\delta)_{w(\lambda + \levelR \Lambda_0 - k \delta)} \ne 0$.

    (2) According to Lemma \ref{Pmax}(2) and \cite[Proposition 12.5]{Kac}, it suffices to show that the section $P_{(\mu, -)}$ is convex with respect to $\tilde{\dot{Q}}_{\ex}$ and closed under the Weyl group $\dot{\mathcal{W}}_{\ex}$. For $\lgdhat$ being of non-branching types, there is nothing to prove since $\lgdhat = \lgdhat_{\ex}$. Let $\lgdhat$ be of branching type.

    Note that the action of $\dot{\sigma}_{r-1,r}$ on weight lattice coincides with the action of the simple reflection corresponding to $\dot{\beta}_r$, and similarly for $\dot{\sigma}_{0,1}$ if exists.
    With $\dot{\Sigma}$ take the place of these simple reflections, the section $P_{(\mu, -)}$ is closed under the action of $\dot{\mathcal{W}}_{\ex}$. 
    
    It remains to show that each $\dot{\beta}_p$-string is unbroken in $P_{(\mu, -)}$ for $p = 0,r$. 
    Let $\dot{\lambda} + m \dot{\delta} \in P_{(\mu, -)}$ corresponding to a Maya diagram $\M$, with $\dot{\lambda} = (\dot{\lambda}_1, \dot{\lambda}_2, \ldots, \dot{\lambda}_r)$. 
    Assume $\dot{\lambda}_r > 0$. Then either that among uncut columns of $\M$ there are more black beads in the $r$th row than in the $-r$th row, or that the $-r$th position at the top of the cut column on $\M$ is white. 
    In either case we can construct a Maya diagram of weight $(\mu, \dot{\lambda} - \dot{\beta}_r; m)$ by pulling a bead off the $r$th row or to the $-r$th row.
    This means the $\dot{\beta}_r$-string through $\dot{\lambda} + m \dot{\delta} \in P_{(\mu, -)}$ is unbroken. A similar discussion holds for $\dot{\beta}_0$, which finishes the convexity statement.
\end{proof}

\subsection{Defects of weights}\label{def-and-move}

This subsection is devoted to the following theorem, which is the main result of this section. 
\begin{Thm}\label{thm-def-weight}
    For $\Lambda \in \tilde{P}^+(\levelR)$, $\beta \in Q^+$ such that $\Lambda - \beta \in P(\Lambda)$, the defect of the algebra $\mathscr{R}^{\Lambda}_{\beta}$ equals to the sum of the components of its moving vector. In other words, 
    $\defect(\Lambda, \beta) = {\bf mv}(\Lambda, \beta) \cdot \mathbf{1}$. 
\end{Thm}

First we list some basic properties of the defect $\defect(\Lambda, \beta)$.

\begin{Lem}
    Let $\Lambda \in \tilde{P}^+(\levelR)$, $\beta \in \tilde{Q}^+$ such that $\Lambda - \beta  \in P(\Lambda)$. 
    \begin{enumerate}
        \item For any $k \in \mathbb{C}$, $\defect(\Lambda + k\delta, \beta) = \defect(\Lambda, \beta)$.
        \item Suppose $w$ is a Weyl group element and $w(\Lambda-\beta)=\Lambda-\beta'$ for some  $\beta' \in \tilde{Q}^+$. Then $\defect(\Lambda, \beta) = \defect(\Lambda, \beta')$. In other words, defect of weight is invariant under affine Weyl group. 
        \item If $\defect(\Lambda,\beta)=0$, then there is a Weyl group element $w$ such that $w(\Lambda-\beta)=\Lambda$.
        \item For any $\sigma \in \Sigma$, $\defect(\Lambda, \beta) = \defect\left(\sigma(\Lambda), \sigma(\beta)\right)$.
    \end{enumerate}
\end{Lem}

\begin{proof}
  The statement  (1) can be given through straightforward calculations according to the definition of defect, while (2) and (4) follow respectively from the facts that both $w \in \mathcal{W}$ and $\sigma \in \Sigma$ keep the bilinear form $(-|-)$ invariant. The statement (3) follows from \cite[Proposition 11.4a]{Kac}
\end{proof}

Since $P_{(\lambda, -)} = P_{(\sigma(\lambda), -)}$ for $\sigma \in \Sigma_{\lleft}$, we also have ${\bf mv}(\Lambda, \beta) = {\bf mv}(\sigma(\Lambda), \sigma(\beta))$, and therefore in the following discussion we may assume $\Lambda \in \tilde{P}^{+}_{\Sigma_{\lleft}}(\levelR)$.
The moving vector ${\bf mv}(\Lambda, \beta)$ under $\Psi$ corresponds to the difference between two dominant weights $\dot{\bar{\Lambda}} - ({\sf d}_\Lambda + b)\dot{\delta}$ and $\dot{\lambda} - k_{\max}\dot{\delta}$, which lies in $\tilde{\dot{Q}}^+_{\ex}$. 
Thanks to a result of Stembridge \cite[Corollary 2.7]{Stem}, we can give the following reduction:

\begin{Lem}\label{Stembridge}
    Let $\lambda, \mu \in \tilde{P}^+$ such that $\lambda > \mu$ with respect to $\tilde{Q}^+$.
    \begin{enumerate}
        \item  If each of $\nu \in \tilde{P}$ with $\lambda > \nu > \mu$ satisfies $\nu \notin \tilde{P}^+$, then $\lambda - \mu \in \Delta$ is a root.
        \item There exists a sequence of weights $\lambda = \lambda_n > \lambda_{n-1} > \cdots > \lambda_1 > \lambda_0 = \mu$, such that $\lambda_i \in \tilde{P}^+$ and $\lambda_{i} - \lambda_{i-1} \in \Delta$ for each $i = 1, \ldots, n$. 
    \end{enumerate}
\end{Lem}

Given $\lambda, \mu \in \tilde{P}^+$, we say these two weights are \emph{adjacent} in $\tilde{P}^+$ if the condition in Lemma \ref{Stembridge} (1) holds, since there are no other weights in $\tilde{P}^+$ between them.
This adjacent relation of weights will reflect deep connection between the corresponding blocks, as given in \S \ref{prop-rep-type}.

\begin{proof}[Proof of Theorem \ref{thm-def-weight}]
Through moves we increases the weight from  $\dot{\bar{\Lambda}} - ({\sf d}_\Lambda + b)\dot{\delta}$ to $\dot{\lambda} - k_{\max}\dot{\delta}$. With the help of Lemma \ref{Stembridge}, we can divide this process into several steps.
Set the following increasing sequence of weights with respect to $\tilde{\dot{Q}}^+_{\ex}$: 
\[
\dot{\bar{\Lambda}} - ({\sf d}_\Lambda + b)\dot{\delta} = \dot{\mu}_0 - k_0\dot{\delta} < \dot{\mu}_1 - k_1\dot{\delta} < \cdots < \dot{\mu}_n - k_n\dot{\delta} = \dot{\lambda} - k_{\max}\dot{\delta},
\]
such that each $\dot{\mu}_i$ lies in  $\dot{P}^+_{\ex}$, and the difference of each step $\dot{\zeta}_i:= \dot{\mu}_i - \dot{\mu}_{i-1} + (k_{i-1} - k_i)\dot{\delta} \in \dot{\Delta}_{\ex}^+$ is a root.
By Lemma \ref{dualYoung}(3), taking the pair of dual weights $(\mu_i, \dot{\mu}_i) \in \sigdom(\levelR) \times \dot{P}^{+}_{\dot{\Sigma}}(\levelL)$, there exists $\sigma_i \in \Sigma$ such that $\bar{\Lambda} - \sigma_i(\mu_i) \in Q$.
According to Corollary \ref{wtSpaceNz}, each weight space $\F_{(\bar{\Lambda} - \bar{\beta}, \dot{\mu}_{i}; -k_{i})}$ is non-zero, and we can get a non-zero $v_i \in \F_{(\bar{\Lambda} - \bar{\beta}, \dot{\mu}_{i}; -k_{i})}$ from the joint highest weight vector $\sigma_i(v(\Y_{\mu_i}))$ via $\lghat_{\lleft}$-action.
In other words, for each $i$ we locate a weight space $\sigma_i(\mu_i, \dot{\mu}_i; -{\sf d}_{\mu_i})$ from which we are able to go down into the weight space $(\bar{\Lambda} - \bar{\beta}, \dot{\mu}_i; -k_i)$. This reduction process can be shown graphically as in Figure \ref{step-by-step-move}.
\begin{figure}[htbp]
\centering
\begin{tikzpicture}[>=stealth]
        
        \def\RightX{4.5}
        
        \node[anchor=west] (R1) at (\RightX, 4.5) {$(\bar{\Lambda} - \bar{\beta}, \dot{\mu}_n; -k_{n})$};
        \node[anchor=west] (R2) at (\RightX, 3.0) {$(\bar{\Lambda} - \bar{\beta}, \dot{\mu}_i; -k_i)$};
        \node[anchor=west] (R3) at (\RightX, 1.5) {$(\bar{\Lambda} - \bar{\beta}, \dot{\mu}_{i-1}; -k_{i-1})$};
        \node[anchor=west] (R4) at (\RightX, 0.0) {$(\bar{\Lambda} - \bar{\beta}, \dot{\bar{\Lambda}}; -{\sf d}_\Lambda - b)$};

        \node[anchor=west] (L1) at (0, 4.5) {$\sigma_n(\mu_n, \dot{\mu}_n; -{\sf d}_{\mu_n})$};
        \node[anchor=west] (L2) at (-1.5, 3.0) {$\sigma_i(\mu_i, \dot{\mu}_i; -{\sf d}_{\mu_i})$};
        \node[anchor=west] (L3) at (-3.0, 1.5) {$\sigma_{i-1}(\mu_{i-1}, \dot{\mu}_{i-1}; -{\sf d}_{\mu_{i-1}})$};
        \node[anchor=west] (L4) at (-4.5, 0.0) {$(\bar{\Lambda}, \dot{\bar{\Lambda}}; -{\sf d}_\Lambda)$};

        \draw[->] (L1.east) -- node[above] {$w$} (R1.west);
        \draw[->] (L2.east) -- node[above] {$-\beta_i$} (R2.west);
        \draw[->] (L3.east) -- node[above] {$-\beta_{i-1}$} (R3.west);
        \draw[->] (L4.east) -- node[above] {$-\beta=-\bar{\beta}-b\delta$} (R4.west);

        \draw[->, densely dotted] ([xshift=2.2cm, yshift=0.1cm]R2.north west) 
            -- node[right] {\small $\dot{\zeta}_{i+1}+\dots+\dot{\zeta}_n$} 
            ([xshift=2.2cm, yshift=-0.1cm]R1.south west);
            
        \draw[->] ([xshift=2.2cm, yshift=0.1cm]R3.north west) 
            -- node[right] {\small $\dot{\zeta}_i$} 
            ([xshift=2.2cm, yshift=-0.1cm]R2.south west);
            
        \draw[->, densely dotted] ([xshift=2.2cm, yshift=0.1cm]R4.north west) 
            -- node[right] {\small $\dot{\zeta}_1+\dots+\dot{\zeta}_{i-1}$} 
            ([xshift=2.2cm, yshift=-0.1cm]R3.south west);
            
    \end{tikzpicture}
\caption{Step-by-step reduction of the moving process}\label{step-by-step-move}
\end{figure}

With this reduction, we discuss how the defect changes in each step.
Set $\Lambda_{(i)} \coloneq  \sigma_i(\mu_i + \levelR\Lambda_0 - {\sf d}_{\mu_i}\delta)$ and $\beta_i \coloneq  \Lambda_{(i)} - (\bar{\Lambda} - \bar{\beta} + \levelR\Lambda_0 - k_i \delta) \in \tilde{Q}$, the defect of the weight space $V_{\lghat}(\Lambda_{(i)})_{\Lambda_{(i)} - \beta_i}$ is 
\begin{align*}
    \defect\left(\Lambda_{(i)} , \beta_i\right) &= \frac{1}{2}\left( |\sigma_i(\mu_i + \levelR\Lambda_0 - {\sf d}_{\mu_i}\delta)|^2 - \left|\bar{\Lambda} - \bar{\beta} + \levelR\Lambda_0 - k_i \delta\right|^2 \right) \\
    &= \frac{1}{2}\left( |\mu_i|^2 - \left|\bar{\Lambda} - \bar{\beta}\right|^2 \right) - \levelR({\sf d}_{\mu_i} - k_i). 
\end{align*}
The difference of defects in each step is
\begin{equation}
    \defect(\Lambda_{(i-1)} , \beta_{i-1}) - \defect(\Lambda_{(i)} , \beta_i) = \frac{1}{2}(|\mu_{i-1}|^2 - |\mu_{i}|^2) + \levelR(-{\sf d}_{\mu_{i-1}} + {\sf d}_{\mu_i} + (k_{i-1} - k_i)). \tag{6.3.1}\label{D}
\end{equation}
Note that $\defect(\Lambda_{(0)} , \beta_{0}) = \defect(\Lambda, \beta)$, and $\defect(\Lambda_{(n)} , \beta_{n}) = 0$ since the weights are in the same Weyl group orbit.
Thus we have $\defect(\Lambda, \beta) = \sum_{i = 1}^n \left(\defect(\Lambda_{(i-1)} , \beta_{i-1}) - \defect(\Lambda_{(i)} , \beta_i)\right)$.
On the other hand, the moving vector maps under $\Psi$ to the sum of the differences of weights $\dot{\zeta}_i$ in each step, 
that is, $\Psi({\bf mv}(\Lambda, \beta)) = \sum_{i = 1}^n \dot{\zeta}_i$.
It is sufficient for us to prove the following Claim. 

\medskip 
\noindent 
{\bf Claim.} {\em   $\defect(\Lambda_{(i-1)} , \beta_{i-1}) - \defect(\Lambda_{(i)} , \beta_i) = {\rm ht}_{\ex} \dot{\zeta}_i$ for each $i$.}

\medskip 
    According to the construction, $\dot{\zeta}_i = \dot{\mu}_i - \dot{\mu}_{i-1} + (k_{i-1} - k_i)\dot{\delta} \in \dot{\Delta}_{\ex}^+$ is a root.
    Following a case-by-case discussion we consider all possible positive roots $\dot{\zeta}_i$.

    By examining the set of positive roots for each non-branching type from Appendix \ref{pos-roots}, 
    we can conclude a unified characterization of $\dot{\Delta}_{\ex}^+$.
    Consider the set
    \[
        S = \{ \dot{\epsilon}_p - \dot{\epsilon}_q \mid p < q\} \cup \{ \dot{\epsilon}_p + \dot{\epsilon}_q \} \cup \{ \dot{\epsilon}_p, -\dot{\epsilon}_p + \frac{1}{2}\dot{\delta} \}.
    \]
    Then any positive root can always be expressed as $\dot{\zeta}_i = \dot{\alpha} + n\dot{\delta}$, where $\dot{\alpha} \in \pm S \cup \{0\}$ and $n\dot{\delta} \in (\dot{\Delta}_{\ex})^+_{im}$.
    We first prove that Claim holds for $\dot{\zeta}_i \in (\dot{\Delta}_{\ex})^+_{im}$ and $\dot{\zeta}_i \in S$ through a case-by-case discussion.
    After that, we consider the case $\dot{\zeta}_i = \dot{\alpha} + n\dot{\delta}$, thereby providing a complete discussion for all positive roots.

    We briefly state the idea.
    Since in each case $\dot{\mu}_j = Y_{\mu_j}^t + \frac{L_{\overline{2h}} \mod 2}{2}\mathbf{1}$ for $j = i-1, i$, we can interpret the difference $\dot{\mu}_i - \dot{\mu}_{i-1}$ using Young diagrams and give the relations between ${\mu}_i,{\mu}_{i-1}$.

    {\em Case 1}. $\dot{\zeta}_i = (k_{i-1} - k_i)\dot{\delta}$.  Then $\dot{\mu}_i = \dot{\mu}_{i-1}$ and $|\mu_{i-1}|^2 = |\mu_{i}|^2$. 
    Since ${\rm ht}_{\ex} \dot{\delta} = \levelR$ for all types (recall again the special convention for $D^{(2)}$ in \S\ref{Parity}),
    \[{\rm ht}_{\ex} \dot{\zeta}_i = \levelR (k_{i-1} - k_i) = \defect(\Lambda_{(i-1)} , \beta_{i-1}) - \defect(\Lambda_{(i)} , \beta_i).\]
    
    {\em Case 2}. $\dot{\zeta}_i = \dot{\epsilon}_p - \dot{\epsilon}_q$, $1 \le p < q \le r$. 
    Then ${\rm ht}_{\ex} \dot{\zeta}_i = q - p$ and $\dot{\mu}_{i} = \dot{\mu}_{i-1} + \dot{\epsilon}_p - \dot{\epsilon}_q$, ${\sf d}_{\mu_{i-1}} = {\sf d}_{\mu_i}$, $k_{i-1} = k_i$. 
    The operation getting Young diagram $Y_{\mu_i}$ from $Y_{\mu_{i-1}}$ is given in Figure \ref{Young+p-q}.

    If $h = \frac{1}{2}$, then $\mu_j = Y_{\mu_j} + \frac{R_{\bar{1}}}{2}\mathbf{1}$. 
    That is, for some $1 \le m < n \le l$, 
    \[
    {\mu_{i-1}} = (\ldots, \overset{m}{q}, \ldots, \overset{n}{p-1}, \ldots) + \frac{R_{\bar{1}}}{2}\mathbf{1}; \; {\mu_i} = (\ldots, \overset{m}{q-1}, \ldots, \overset{n}{p}, \ldots) + \frac{R_{\bar{1}}}{2}\mathbf{1}.
    \] 
    Then the equality becomes 
    ${\rm LHS} = \frac{1}{2}(|\mu_{i-1}|^2 - |\mu_{i}|^2) = q - p = {\rm ht}_{\ex} \dot{\zeta}_i$.

    If $h = 0$, then $\mu_j = Y^c_{\mu_j} + \frac{R_{\bar{0}} \pmod{2} }{2}\mathbf{1}$. 
    That is, for some $1 \le n < m \le l$, 
    \[
    {\mu_{i-1}} = (\ldots, \overset{n}{-p+1}, \ldots, \overset{m}{-q}, \ldots) + \frac{R_{\bar{0}}}{2}\mathbf{1}; \; {\mu_i} = (\ldots, \overset{n}{-p}, \ldots, \overset{m}{-q+1}, \ldots) + \frac{R_{\bar{0}}}{2}\mathbf{1}.
    \] 
    Then the equality becomes 
    ${\rm LHS} = \frac{1}{2}(|\mu_{i-1}|^2 - |\mu_{i}|^2) = q - p = {\rm ht}_{\ex} \dot{\zeta}_i$. 

\medskip     
The remaining cases are similar, and we only give a sketch.

{\em Case 3}. $\dot{\zeta}_i = \dot{\epsilon}_p + \dot{\epsilon}_q$, $1 \le p , q \le r$. Then ${\rm ht}_{\ex} \dot{\zeta}_i = R_{\bar{0}} + 1 - q - p$ and $\dot{\mu}_{i} = \dot{\mu}_{i-1} + \dot{\epsilon}_p + \dot{\epsilon}_q$, ${\sf d}_{\mu_i} - {\sf d}_{\mu_{i-1}} = 2h$. 
    The operation getting Young diagram $Y_{\mu_i}$ from $Y_{\mu_{i-1}}$ is given in Figure \ref{Young+p+q}.

 {\em Case 4}. $\dot{\zeta}_i = \dot{\epsilon}_p$, $1 \le p \le r$. Then ${\rm ht}_{\ex} \dot{\zeta}_i = \frac{R_{\bar{0}}+1}{2} - p$ and $\dot{\mu}_{i} = \dot{\mu}_{i-1} + \dot{\epsilon}_p$, ${\sf d}_{\mu_i} - {\sf d}_{\mu_{i-1}} = h$. 
    The operation getting Young diagram $Y_{\mu_i}$ from $Y_{\mu_{i-1}}$ is given in Figure \ref{Young+p}.

{\em Case 5}. $\dot{\zeta}_i = -\dot{\epsilon}_p + \frac{1}{2}\dot{\delta}$, $1 \le p \le r$. 
    Then ${\rm ht}_{\ex} \dot{\zeta}_i = p + \frac{R_{\bar{1}} - 1}{2} $ and $\dot{\mu}_{i} = \dot{\mu}_{i-1} - \dot{\epsilon}_p$, ${\sf d}_{\mu_i} - {\sf d}_{\mu_{i-1}} = -h$, $k_{i-1} - k_i = \frac{1}{2}$. 
    The operation getting Young diagram $Y_{\mu_i}$ from $Y_{\mu_{i-1}}$ is given in Figure \ref{Young-p}.

\medskip     
\begin{figure}[htbp]
    \centering
        \begin{subfigure}{0.45\textwidth}
            \centering
            \renewcommand{\thesubfigure}{(Case 2)}
            \begin{tikzpicture}[scale=0.5, every node/.style={scale=0.8}, font=\sffamily]
                \useasboundingbox (0,0) rectangle (8,-6);
                \draw[very thick] (0,0) rectangle (8,-6);

                \fill[gray!15] (0,0) -- (7,0) -- (7,-1) -- (5,-1) -- (5,-2) -- (4,-2) -- (4,-4) -- (2,-4) -- (2,-5) -- (0,-5) -- cycle;

                \draw[thick, darkgray] (0,0) -- (7,0) -- (7,-1) -- (5,-1) -- (5,-2) -- (4,-2) -- (4,-4) -- (2,-4) -- (2,-5) -- (0,-5) -- cycle;

                \fill[red!10] (3,-3) rectangle (4,-4);
                \draw[thick, red, dashed] (3,-3) rectangle (4,-4);
                \node[red] at (3.5, -3.5) {\Large $-$};

                \fill[blue!10] (5,-1) rectangle (6,-2);
                \draw[thick, blue, dashed] (5,-1) rectangle (6,-2);
                \node[blue] at (5.5, -1.5) {\Large $+$};

                \draw[-{Stealth[scale=1.2]}, thick, gray, dashed] (3.7, -3.2) to[out=45, in=225] (5.3, -1.8);

                \node[left] at (0, -1.5) {Row $p$};
                \draw[thick, dashed, darkgray] (0, -1.5) -- (5, -1.5);

                \node[left] at (0, -3.5) {Row $q$};
                \draw[thick, dashed, darkgray] (0, -3.5) -- (3, -3.5);

                \node[above] at (3.5, 0) {Col $m$};
                \draw[thick, dashed, darkgray] (3.5, 0) -- (3.5, -3);

                \node[above] at (5.5, 0) {Col $n$};
                \draw[thick, dashed, darkgray] (5.5, 0) -- (5.5, -1);

                \draw[dotted, lightgray] (0,0) grid[step=1] (8,-6);

            \end{tikzpicture}
            \caption{$Y^t_{{\mu}_{i}} = Y^t_{{\mu}_{i-1}} + \dot{\epsilon}_p - \dot{\epsilon}_q$}
            \label{Young+p-q}
        \end{subfigure}
        \hspace{1em}
        \begin{subfigure}{0.45\textwidth}
            \centering
            \renewcommand{\thesubfigure}{(Case 3)}
            \begin{tikzpicture}[scale=0.5, every node/.style={scale=0.8}, font=\sffamily]
                \useasboundingbox (0,0) rectangle (8,-6);
                \draw[very thick] (0,0) rectangle (8,-6);

                \fill[gray!15] (0,0) -- (7,0) -- (7,-1) -- (5,-1) -- (5,-2) -- (4,-2) -- (4,-3) -- (3,-3) -- (3,-4) -- (2,-4) -- (2,-5) -- (0,-5) -- cycle;

                \draw[thick, darkgray] (0,0) -- (7,0) -- (7,-1) -- (5,-1) -- (5,-2) -- (4,-2) -- (4,-3) -- (3,-3) -- (3,-4) -- (2,-4) -- (2,-5) -- (0,-5) -- cycle;

                \fill[blue!10] (3,-3) rectangle (4,-4);
                \draw[thick, blue, dashed] (3,-3) rectangle (4,-4);
                \node[blue] at (3.5, -3.5) {\Large $+$};

                \fill[blue!10] (5,-1) rectangle (6,-2);
                \draw[thick, blue, dashed] (5,-1) rectangle (6,-2);
                \node[blue] at (5.5, -1.5) {\Large $+$};


                \node[left] at (0, -1.5) {Row $p$};
                \draw[thick, dashed, darkgray] (0, -1.5) -- (5, -1.5);

                \node[left] at (0, -3.5) {Row $q$};
                \draw[thick, dashed, darkgray] (0, -3.5) -- (3, -3.5);

                \node[above] at (3.5, 0) {Col $m$};
                \draw[thick, dashed, darkgray] (3.5, 0) -- (3.5, -3);

                \node[above] at (5.5, 0) {Col $n$};
                \draw[thick, dashed, darkgray] (5.5, 0) -- (5.5, -1);

                \draw[dotted, lightgray] (0,0) grid[step=1] (8,-6);

            \end{tikzpicture}
            \caption{$Y^t_{{\mu}_{i}} = Y^t_{{\mu}_{i-1}} + \dot{\epsilon}_p + \dot{\epsilon}_q$}
            \label{Young+p+q}        
        \end{subfigure}
    \end{figure}
\begin{figure}[htbp]
    \centering
        \begin{subfigure}{0.45\textwidth}
            \centering
            \renewcommand{\thesubfigure}{(Case 4)}
            \begin{tikzpicture}[scale=0.5, every node/.style={scale=0.8}, font=\sffamily]
                \useasboundingbox (0,0) rectangle (8,-6);
                \draw[very thick] (0,0) rectangle (8,-6);

                \fill[gray!15] (0,0) -- (7,0) -- (7,-1) -- (5,-1) -- (5,-2) -- (4,-2) -- (4,-4) -- (2,-4) -- (2,-5) -- (0,-5) -- cycle;

                \draw[thick, darkgray] (0,0) -- (7,0) -- (7,-1) -- (5,-1) -- (5,-2) -- (4,-2) -- (4,-4) -- (2,-4) -- (2,-5) -- (0,-5) -- cycle;

                \fill[blue!10] (5,-1) rectangle (6,-2);
                \draw[thick, blue, dashed] (5,-1) rectangle (6,-2);
                \node[blue] at (5.5, -1.5) {\Large $+$};


                \node[left] at (0, -1.5) {Row $p$};
                \draw[thick, dashed, darkgray] (0, -1.5) -- (5, -1.5);

                \node[above] at (5.5, 0) {Col $n$};
                \draw[thick, dashed, darkgray] (5.5, 0) -- (5.5, -1);

                \draw[dotted, lightgray] (0,0) grid[step=1] (8,-6);

            \end{tikzpicture}
            \caption{$Y^t_{{\mu}_{i}} = Y^t_{{\mu}_{i-1}} + \dot{\epsilon}_p$}
            \label{Young+p}
        \end{subfigure}
        \hspace{1em}
        \begin{subfigure}{0.45\textwidth}
            \centering
            \renewcommand{\thesubfigure}{(Case 5)}
            
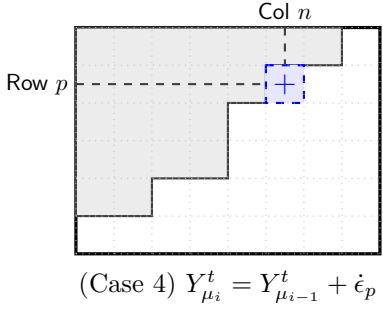
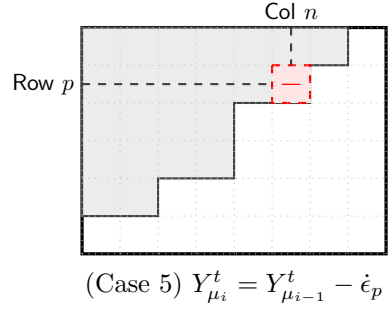
\begin{tikzpicture}[scale=0.5, every node/.style={scale=0.8}, font=\sffamily]
                \useasboundingbox (0,0) rectangle (8,-6);
                \draw[very thick] (0,0) rectangle (8,-6);

                \fill[gray!15] (0,0) -- (7,0) -- (7,-1) -- (6,-1) -- (6,-2) -- (4,-2) -- (4,-4) -- (2,-4) -- (2,-5) -- (0,-5) -- cycle;

                \draw[thick, darkgray] (0,0) -- (7,0) -- (7,-1) -- (6,-1) -- (6,-2) -- (4,-2) -- (4,-4) -- (2,-4) -- (2,-5) -- (0,-5) -- cycle;

                \fill[red!10] (5,-1) rectangle (6,-2);
                \draw[thick, red, dashed] (5,-1) rectangle (6,-2);
                \node[red] at (5.5, -1.5) {\Large $-$};


                \node[left] at (0, -1.5) {Row $p$};
                \draw[thick, dashed, darkgray] (0, -1.5) -- (5, -1.5);

                \node[above] at (5.5, 0) {Col $n$};
                \draw[thick, dashed, darkgray] (5.5, 0) -- (5.5, -1);

                \draw[dotted, lightgray] (0,0) grid[step=1] (8,-6);

            \end{tikzpicture}
            \caption{$Y^t_{{\mu}_{i}} = Y^t_{{\mu}_{i-1}} - \dot{\epsilon}_p$}
            \label{Young-p}        
        \end{subfigure}
    \end{figure}
 
    So far we have finished the discussions for $\dot{\zeta}_i \in (\dot{\Delta}_{\ex})^+_{im}$ or $\dot{\zeta}_i \in S$.
    Finally we suppose $\dot{\zeta}_i = {\dot{\alpha}} + n \dot{\delta}$, $\dot{\alpha} \in \pm S \cup \{0\}$, $n \dot{\delta} \in (\dot{\Delta}_{\ex})^+_{im}$. 
    The idea is to check the desired equality using the additivity on both sides.
    Take the weight space $(\bar{\Lambda} - \bar{\beta}, \dot{\mu}_i; -k_i - n)$, 
    which is non-zero since $n \dot{\delta} \in (\dot{\Delta}_{\ex})^+_{im}$ and therefore $V_{\lghat_{\lleft}}(\Lambda_{(i)})_{\Lambda_{(i)} - \beta_i - n\dot{\delta}} \ne 0$. 
    Consider the following weight diagram.
    \begin{center}
    $
    \xymatrix{
            &   &   (\mu_i, \dot{\mu}_i; -{\sf d}_{\mu_i})\ar[r]^{-\beta_{i}} & (\bar{\Lambda} - \bar{\beta}, \dot{\mu}_i; -k_i) \\
            &   (\mu_{i-1}, \dot{\mu}_{i-1}; -{\sf d}_{\mu_{i-1}})\ar[rr]^{-\beta_{i-1}} & & (\bar{\Lambda} - \bar{\beta}, \dot{\mu}_{i-1}; -k_{i-1})\ar@{-->}[u]_{\dot{\zeta}_{i}} \\
        (\mu_i, \dot{\mu}_i; -{\sf d}_{\mu_i})\ar[rrr]^{-\beta_{i} - n\delta} & & & (\bar{\Lambda} - \bar{\beta}, \dot{\mu}_i; -k_i - n) \ar@{-->}[u]_{-\dot{\alpha}}
    }
    $
    \end{center}
The case $-\dot{\alpha} \in S$  has been discussed in the previous cases. 
    If on the contrary $-\dot{\alpha} \in -S$, then  in the above diagram, we goes down from the middle row to the bottom row via $\dot{\alpha}$, which can be discussed analogously.
    Checking the changes of weights as well as defects, we have:
    \begin{align*}
            \left(\defect(\Lambda_{(i)} , \beta_i + n\delta) - \defect(\Lambda_{(i-1)} , \beta_{i-1})\right) + \left(\defect(\Lambda_{(i-1)} , \beta_{i-1}) - \defect(\Lambda_{(i)} , \beta_i)\right) \\
        \hfill    = \defect(\Lambda_{(i)} , \beta_i + n\delta) - \defect(\Lambda_{(i)} , \beta_i),
    \end{align*}
    \[
        {\rm ht}_{\ex} (-\dot{\alpha}) + {\rm ht}_{\ex}\dot{\zeta}_i = {\rm ht}_{\ex}n \dot{\delta}. 
    \]
    Then by Cases 1-5 above, the first term as well as the third term in two equations coincide respectively, forcing $\defect(\Lambda_{(i-1)} , \beta_{i-1}) - \defect(\Lambda_{(i)} , \beta_i) = {\rm ht}_{\ex} \dot{\zeta}_i$. 

We have verified the equality for all possible positive roots, which finishes the proof.
\end{proof}

\section{Concluding comments}

Finally, we briefly present some applications of the duality theorem, abacus models, and moving vectors. 
We mainly focus on the following two points. 
On the one hand, based on the new realization of duality, we refine the duality results regarding the counting of $\max^+(\Lambda)$ in \cite{KOO}.
On the other hand, using moving vectors, we provide a new and more concise description of the conditions on the representation types of cyclotomic KLR algebras of types $A^{(1)}$ and $C^{(1)}$ in \cite{AHSW, ASW, LQ}.

\subsection{Counting of $\max^+$}
Recall that in \cite{KOO} the weight set $\max^+(\Lambda)$ was characterized by defining root-sieving equivalence classes of dominant weights, and formulae counting the sizes of these classes were provided.
By observing a certain symmetry between level and rank in the formula, several pairs of sets of equal size under the interchange of level and rank are constructed, providing a level-rank duality at the counting level.
Based on the concrete realization of level-rank duality in this paper, we revisit the relations of the $\max^+$ sets corresponding to these dual weights.

Since the duality we consider is between $\Sigma \ltimes \lghat$ and $\dot{\Sigma} \ltimes \lgdhat$, the correspondence between the dominant weights of the affine algebras may not be one‑to‑one, as considered in Example \ref{example-diag-auto-1}.
In other words, the level-rank duality we constructed gives a one-to-one correspondence between the orbits of the dominant weights of the affine Lie algebras on both sides under the action of the diagram automorphism groups.
Combining the root-sieving equivalence class decompositions given in Lemma \ref{root-sieving-eq} and Remark \ref{parity-in-Fock}, we can realize the counting equality of $\max^+(\Lambda)$ sets to concrete correspondence of weights.

We use two examples to illustrate this correspondence between orbits and give the counting of their $\max^+$ sets.

\begin{Exam}{\rm 
    $(1)$ Let $(\lghat, \lgdhat) = (\lsohat(2l), \lsohat(2r+1, 1))$ and $h = \frac{1}{2}$.
    By Lemma \ref{root-sieving-eq}, the dominant weights of $\lsohat(2l)$ in this case are divided into two root-sieving equivalence classes: 
    \[
    \tilde{P}^+(2r+2)_{\bar{1}, \bar{1}} / {\sim} = \{(2r+2)\Lambda_0 + \varpi_{l-1}\} \sqcup \{(2r+2)\Lambda_0 + \varpi_{l}\}.
    \]
    Note that both the diagram automorphisms $\sigma_{01}$ and $\sigma_{l-1,l}$ permute the two equivalence classes.
    For $Y\in \P^l_r$, $\lambda_{\Y}$ and $\sigma_{01}\sigma_{l-1,l}(\lambda_{\Y})$ lie in  the latter equivalence class, 
    while $\sigma_{01}(\lambda_{\Y})$ and $\sigma_{l-1,l}(\lambda_{\Y})$ lie in the former. 
    Under duality, all these four weights correspond to the same $\dot{\lambda}_{\Y}$. In other words,
    \[
        \left\{ \lambda_{\Y}, \ \sigma_{01}\sigma_{l-1,l}(\lambda_{\Y}) \right\} \overset{2:1}{\longleftrightarrow} \{ \dot{\lambda}_{\Y} \};
        \quad 
        \left\{ \sigma_{01}(\lambda_{\Y}), \ \sigma_{l-1,l}(\lambda_{\Y}) \right\} \overset{2:1}{\longleftrightarrow} \{ \dot{\lambda}_{\Y} \}.
    \]
    On the other hand, calculating the size of the equivalence classes according to the formulae in \cite{KOO} yields
    \[
        \left|
        \max^+_{D^{(1)}_l} \left( (2r+1)\Lambda_0 + \Lambda_i \right) \right| = 2 \binom{l+r}{r}, \; i = l, l-1;
        \quad
        \left| \max^+_{D^{(2)}_{r+1}} \left( 2l\dot{\Lambda}_0 \right) \right| =  \binom{l+r}{r}.
    \]
    That is, the correspondence of the $\Sigma$-orbits of the root-sieving equivalence classes provides the correspondence for the counting of the $\max^+$ sets.

    $(2$) Let $(\lghat, \lgdhat) = (\lsohat(2l+1), \lsohat(2r+1))$ and $h = 0$.
    In this case, all dominant weights with respect to $\lghat$ or $\lgdhat$ in $\F$ lie respectively in the same root-sieving equivalence class 
    $\max^+_{B^{(1)}_l} \left( 2r\Lambda_0 + \Lambda_l \right)$ or $\max^+_{B^{(1)}_r} \left( 2l\dot{\Lambda}_0 + \dot{\Lambda}_r \right)$.
    Similar to the case in Example \ref{example-diag-auto-2}, the Young diagrams $Y\in\P^l_r$ corresponding to the highest weight vectors are divided into the following two types based on their behavior under the diagram automorphism groups $\Sigma$ and $\dot{\Sigma}$.

If $y_l = 0$, then $y^c_1 = r$ and $y^t_1 < l$. Thus $\sigma_{01}(\lambda_{\Y}) = \lambda_{\Y}$ and $\dot{\lambda}_{\Y} \ne \dot{\sigma}_{01}(\dot{\lambda}_{\Y})$.
        There are $\binom{l+r-1}{l-1}$ such $Y$. In this case, we have the orbit set correspondence:
        \[
            \left\{ \lambda_{\Y} \mid y_l = 0 \right\} \overset{1:2}{\longleftrightarrow} \{ \dot{\lambda}_{\Y}, \ \dot{\sigma}_{01}(\dot{\lambda}_{\Y}) \mid y_l = 0 \}.
        \]
  
 If $y_l \ne 0$, then $y^c_1 < r$ and $y^t_1 = l$. Thus $\sigma_{01}(\lambda_{\Y}) \ne \lambda_{\Y}$ and $\dot{\lambda}_{\Y} = \dot{\sigma}_{01}(\dot{\lambda}_{\Y})$.
        There are $\binom{l+r-1}{r-1}$ such $Y$. The orbit set correspondence is as follows.
        \[
            \left\{ \lambda_{\Y}, \ \sigma_{01}(\lambda_{\Y}) \mid y_l \ne 0 \right\} \overset{2:1}{\longleftrightarrow} \{ \dot{\lambda}_{\Y} \mid y_l \ne 0 \}.
        \]

 Calculating the size of the equivalence classes according to the formulae in \cite{KOO} yields
    \begin{align*}
        \left|
        \max^+_{B^{(1)}_l} \left( 2r\Lambda_0 + \Lambda_l \right) \right| & = \binom{l+r-1}{l-1} + 2\binom{l+r-1}{l} ;
        \\
        \left| \max^+_{B^{(1)}_r} \left( 2l\dot{\Lambda}_0 + \dot{\Lambda}_r \right) \right|
        & =  2\binom{l+r-1}{r} + \binom{l+r-1}{r-1}.
    \end{align*}
    That is, although the two $\max^+$ sets are of equal size in this case, this equality at the counting level is actually not given by a one-to-one correspondence of roots, but happens to be pieced together by two sets of orbit correspondences.
}\end{Exam}

\subsection{Conditions for finite representation type}

In this subsection, we reformulate the conditions on the representation types of cyclotomic KLR algebras of types $A^{(1)}$ and $C^{(1)}$ in \cite{AHSW, ASW, LQ}.

According to \cite{LQ}, for type $A^{(1)}$, the moving vectors on the abacus corresponding to a block can be used to give criterion of finite representation type.
In short, the sufficiency requires a detailed discussion of the charge sequence of the abacus. However, by summarizing various cases, we can deduce a more concise necessary condition. Namely, a block is of finite representation type only if the image of its moving vector under $\Psi$ is a root not exceeding $\dot{\delta}$. 
The idea of ``splitting the move into steps, each corresponding to a positive root'' in \S \ref{def-and-move} is closely related to the aforementioned condition.

On the other hand, also for type $A^{(1)}$, a quiver was constructed in \cite{ASW}, whose vertices are elements in the root-sieving equivalence class of $\Lambda$. 
This quiver helps to classify various cases of $\Lambda$ and $\beta$ in detail, and giving criteria for the block to be of finite, tame, or wild type, respectively. 
The criteria for type $C^{(1)}$ was later given in \cite{AHSW}. 

By mapping each dominant weight on the quiver constructed in \cite{ASW} to a Young diagram according to formula (\ref{weightY}) or (\ref{weightY-A}) and considering its dual, we can routinely verify that 
two dominant weights $\Lambda$ and $\Lambda'$ are directly linked on the quiver if and only if their dual weights $\dot{\Lambda}$ and $\dot{\Lambda}'$ are adjacent in the sense of Lemma \ref{Stembridge}(1), i.e., there are no other dominant weights between them.
 
Consider dual pairs of type $A^{(1)}$ or $C^{(1)}$.
    For $\Lambda \in \tilde{P}^+(\levelR)$ and $\beta \in Q^+$ such that $\Lambda - \beta \in P(\Lambda)$, and 
    assume that there is a Weyl group element $w$,  $\lambda \in P^+(\levelR)$, and $k\in \mathbb{Q}$ such that 
    $\Lambda - \beta - {\sf d}_{\Lambda}\delta = w(\lambda + \levelR\Lambda_0 - k\delta)$.
    Let $(\lambda, \dot{\lambda})$ and $(\bar{\Lambda}, \dot{\bar{\Lambda}})$ be dual pairs of weights. Then, by the methods in \cite{ASW,AHSW}, the algebra $\mathscr{R}^{\Lambda}_{\beta}$ is of finite representation type if and only if 
    the two weights $\dot{\lambda} - {\sf d}_{\lambda}\dot{\delta} > \dot{\bar{\Lambda}} - k\dot{\delta}$ are adjacent in the sense of Lemma \ref{Stembridge}(1). Particularly, we have the following proposition. 

\begin{Prop}\label{prop-rep-type}
   For type $A^{(1)}$ or $C^{(1)}$, $\Lambda \in \tilde{P}^+(\levelR)$ and $\beta \in Q^+$ such that $\Lambda - \beta \in P(\Lambda)$, if the  algebra $\mathscr{R}^{\Lambda}_{\beta}$ has finite representation type, then the moving vector $\Psi({\bf mv}(\Lambda, \beta)) \in \dot{\Delta}^+$ is a positive root not exceeding $\dot{\delta}$, 
    and the defect of the weight space satisfies $\defect(\Lambda, \beta) < \levelR$.
\end{Prop}

\section*{Appendix}
\appendix
\section{Datum of each type}
\label{appendix-datum-types}

For the affine Lie algebra $\lghat$ of each classical type, we list the following data:
\begin{itemize}
    \item the Dynkin diagram and the Dynkin indices $a_i$, $a_i^{\vee}$;
    \item the finite dimensional Lie algebra $\lieg$ used to construct $\lghat$, and its dimension;
    \item the root data of its finite dimensional subalgebra $\lgo$, writing as the coordinates under the standard basis $\dlho = \Span\{\epsilon_i\}$, including simple roots $\alpha_i$, the highest root $\theta$ and the half sum of positive roots $\rho$;
    \item the weight data of its finite dimensional subalgebra $\lgo$, writing as the coordinates under the standard basis $\dlho = \Span\{\epsilon_i\}$, including fundamental dominant weights $\varpi_i$;
    \item the group of diagram automorphisms $\Sigma$ needed for duality and its action;
    \item For $R\in \mathbb{Z}_{>0}$, the set of orbit representatives of level-$R$ dominant weights under $\Sigma$.
\end{itemize}

\medskip
\noindent
{\bf A.1\quad Type $D^{(1)}_{\rkCar}$, $\lghat = \lsohat({2 \rkCar})$}
    \begin{center}
    \begin{tikzpicture}[scale=1.2]
        \draw (1,0)--(0.3, 0.7)node[dot]{}node[right]{${\scriptstyle 0}$};
        \draw (1,0)--(0.3,-0.7)node[dot]{}node[right]{${\scriptstyle 1}$};
        \draw (1,0)node[dot]{}node[below]{${\scriptstyle 2}$}--(2,0)node[dot]{}node[below]{${\scriptstyle 3}$};
        \node at (2.5,0) {$\cdots\cdots$};
        \draw (3,0)node[dot]{}node[below]{${\scriptstyle l-3}$}--(4,0)node[dot]{}node[below]{${\scriptstyle l-2}$}--(4.7,0.7)node[dot]{}node[right]{${\scriptstyle l-1}$};
        \draw (4,0)--(4.7,-0.7)node[dot]{}node[right]{${\scriptstyle l}$};

        \draw[<->,dashed] (0,0.5) arc (150:210:1);
        \node[left] at (-0.1,0) {${\scriptstyle \sigma_{01}}$};

        \draw[<->,dashed] (5,0.5) arc (30:-30:1);
        \node[right] at (5.1,0) {${\scriptstyle \sigma_{l-1,l}}$};
    \end{tikzpicture} 
    \end{center}
    \begin{itemize}
        \item $(a_0, \dots, a_{\rkCar}) = (1, 1, 2, \dots, 2, 1, 1)$; $(a_0^{\vee}, \dots, a_{\rkCar}^{\vee}) = (1, 1, 2, \dots, 2, 1, 1)$; $\dCox = {2 \rkCar}-2$.
        \item $\lieg = \lso(2l)$, ${\rm dim} \lieg = {\rkCar}({2 \rkCar}-1)$.
        \item $\alpha_i = \epsilon_i - \epsilon_{i+1}$ for $1 \le i \le {\rkCar}-1$; $\alpha_{\rkCar} = \epsilon_{{\rkCar}-1} + \epsilon_{{\rkCar}}$; $\theta = \epsilon_1 + \epsilon_2$; $\rho = \sum\limits_{i=1}^{\rkCar} ({\rkCar}-i) \epsilon_i$.
        \item $(\epsilon_i | \epsilon_j) = \delta_{i, j}$; $D={\rm diag}(1, 1,\dots, 1, 1)$.
        \item $\varpi_i = \epsilon_1 + \cdots + \epsilon_i$ for $1 \le i \le {\rkCar}-2$; 
        
        $\varpi_{{\rkCar}-1} = \frac{1}{2}(\epsilon_1+ \cdots+ \epsilon_{{\rkCar}-1} - \epsilon_{\rkCar})$; $\varpi_{\rkCar} = \frac{1}{2}(\epsilon_1+ \cdots+ \epsilon_{{\rkCar}-1} + \epsilon_{\rkCar})$.
        \item Diagram automorphism: $\Sigma = \langle \sigma_{01}, \sigma_{{\rkCar}-1, {\rkCar}} \rangle$,
        
        $\sigma_{01}$ switches the vertices $0$ and $1$; $\sigma_{{\rkCar}-1, {\rkCar}}$ switches the vertices ${\rkCar}-1$ and ${\rkCar}$.
        \item The group $\Sigma$ acts on $\tilde{\lieh}{}^*$ as follows. 
        
        $$\sigma_{01}: s\delta + {\lLevel}\Lambda_0 + \sum_{i=1}^{\rkCar} y_i \epsilon_i  \mapsto  (s-\frac{{\lLevel}}{2} + y_1)\delta + {\lLevel}\Lambda_0 + ({\lLevel} - y_1) \epsilon_1 + \sum_{i=2}^{\rkCar} y_i \epsilon_i.$$
        $$\sigma_{{\rkCar}-1, {\rkCar}}: s\delta + {\lLevel}\Lambda_0 + \sum_{i=1}^{\rkCar} y_i \epsilon_i  \mapsto s\delta + {\lLevel}\Lambda_0 + \sum_{i=1}^{{\rkCar}-1} y_i \epsilon_i - y_{\rkCar} \epsilon_{\rkCar}.$$
        \item $\sigdomdel({\lLevel}) = \left\{ {\sum\limits_{i=0}^{\rkCar}} c_i \Lambda_i \mid  \sum\limits_{i=0}^{\rkCar} c_i a_i^\vee = {\lLevel}, c_0 \ge c_1, c_{\rkCar} \ge c_{{\rkCar}-1} \right\}$. 
    \end{itemize}

\medskip 
\noindent
{\bf A.2\quad Type $B^{(1)}_{\rkCar}$, $\lghat = \lsohat({2 \rkCar}+1)$}
    \begin{center}
    \begin{tikzpicture}[scale=1.2]
        \draw (1,0)--(0.3, 0.7)node[dot]{}node[right]{${\scriptstyle 0}$};
        \draw (1,0)--(0.3,-0.7)node[dot]{}node[right]{${\scriptstyle 1}$};
        \draw (1,0)node[dot]{}node[below]{${\scriptstyle 2}$}--(2,0)node[dot]{}node[below]{${\scriptstyle 3}$};
        \node at (2.5,0) {$\cdots\cdots$};
        \draw (3,0)node[dot]{}node[below]{${\scriptstyle l-2}$}--(4,0);
        \draw[line width=.5pt, double distance=2pt,->-] (4,0) node[dot]{} node[below]{${\scriptstyle l-1}$} -- (5,0) node[dot]{}node[below]{${\scriptstyle l}$};

        \draw[<->,dashed] (0,0.5) arc (150:210:1);
        \node[left] at (-0.1,0) {${\scriptstyle \sigma_{01}}$};
    \end{tikzpicture} 
    \end{center}
    \begin{itemize}
        \item $(a_0, \dots, a_{\rkCar}) = (1, 1, 2, \dots, 2, 2)$; $(a_0^{\vee}, \dots, a_{\rkCar}^{\vee}) = (1, 1, 2, \dots, 2, 1)$; $\dCox = {2 \rkCar}-1$.
        \item $\lieg = \lso(2l+1)$, ${\rm dim} \lieg = {\rkCar}({2 \rkCar}+1)$.
        \item $\alpha_i = \epsilon_i - \epsilon_{i+1}$ for $1 \le i \le {\rkCar}-1$; $\alpha_{\rkCar} = \epsilon_{{\rkCar}}$; $\theta = \epsilon_1 + \epsilon_2$; $\rho = \sum\limits_{i=1}^{\rkCar} ({\rkCar} + \frac{1}{2} - i) \epsilon_i$.
        \item $(\epsilon_i | \epsilon_j) = \delta_{i, j}$; $D={\rm diag}(1, 1,\dots, 1, \frac{1}{2})$.
        \item $\varpi_i = \epsilon_1 + \cdots + \epsilon_i$ for $1 \le i \le {\rkCar}-1$; $\varpi_{{\rkCar}} = \frac{1}{2}(\epsilon_1+ \cdots \epsilon_{{\rkCar}-1} + \epsilon_{\rkCar})$.
        \item Diagram automorphism: $\Sigma = \langle \sigma_{01} \rangle$,
        
            $\sigma_{01}$ switches the vertices $0$ and $1$.
        \item The group $\Sigma$ acts on $\tilde{\lieh}{}^*$ as follows. 
        $$\sigma_{01}: s\delta + {\lLevel}\Lambda_0 + \sum_{i=1}^{\rkCar} y_i \epsilon_i  \mapsto  (s-\frac{{\lLevel}}{2} + y_1)\delta + {\lLevel}\Lambda_0 + ({\lLevel} - y_1) \epsilon_1 + \sum_{i=2}^{\rkCar} y_i \epsilon_i.$$
        \item $\sigdomdel({\lLevel}) = \left\{ \sum\limits_{i=0}^{\rkCar} c_i \Lambda_i \mid  \sum\limits_{i=0}^{\rkCar} c_i a_i^\vee = {\lLevel}, c_0 \ge c_1 \right\}$. 
    \end{itemize}

\medskip 
\noindent
{\bf A.3\quad Type $D^{(2)}_{{\rkCar}+1}$, $\lghat = \lsohat({2 \rkCar}+1,1)$}
    \begin{center}
    \begin{tikzpicture}[scale=1.2]
        \draw[line width=.5pt, double distance=2pt,->-] (1,0) node[dot]{} node[below]{${\scriptstyle 1}$} -- (0,0) node[dot]{}node[below]{${\scriptstyle 0}$};
        \draw (1,0)--(2,0)node[dot]{}node[below]{${\scriptstyle 2}$};
        \node at (2.5,0) {$\cdots\cdots$};
        \draw (3,0)node[dot]{}node[below]{${\scriptstyle l-2}$}--(4,0);
        \draw[line width=.5pt, double distance=2pt,->-] (4,0) node[dot]{} node[below]{${\scriptstyle l-1}$} -- (5,0) node[dot]{}node[below]{${\scriptstyle l}$};
    \end{tikzpicture} 
    \end{center}    
    \begin{itemize}
        \item $(a_0, \dots, a_{\rkCar}) = (2, \dots, 2)$; $(a_0^{\vee}, \dots, a_{\rkCar}^{\vee}) = (1, 2, \dots, 2, 1)$; $\dCox = {2 \rkCar}$.
        \item $\lieg = \lso(2l+1,1)$, ${\rm dim} \lieg = ({\rkCar}+1)({2 \rkCar}+1)$.
        \item $\alpha_i = \epsilon_i - \epsilon_{i+1}$ for $1 \le i \le {\rkCar}-1$; $\alpha_{\rkCar} = \epsilon_{{\rkCar}}$; $\theta = \epsilon_1$; $\rho = \sum\limits_{i=1}^{\rkCar} ({\rkCar} + \frac{1}{2} - i) \epsilon_i$.
        \item $(\epsilon_i | \epsilon_j) = \delta_{i, j}$; $D={\rm diag}(\frac{1}{2}, 1,\dots, 1, \frac{1}{2})$.
        \item $\varpi_i = \epsilon_1 + \cdots + \epsilon_i$ for $1 \le i \le {\rkCar}-1$; $\varpi_{{\rkCar}} = \frac{1}{2}(\epsilon_1+ \cdots \epsilon_{{\rkCar}-1} + \epsilon_{\rkCar})$.
        \item Diagram automorphism: $\Sigma = \{{\rm id}\}$.
        \item $\sigdomdel({\lLevel}) = \left\{ \sum\limits_{i=0}^{\rkCar} c_i \Lambda_i \mid  \sum\limits_{i=0}^{\rkCar} c_i a_i^\vee = {\lLevel}\right\}$.
    \end{itemize}

\medskip 
\noindent
{\bf A.4\quad Type $A^{(2)}_{2\rkCar-1}$, $\lghat = \lslhat{}^{(2)}({2 \rkCar})$}
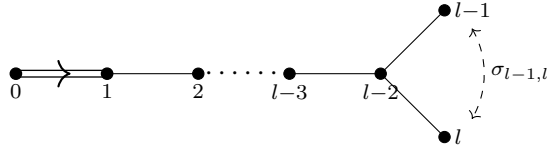
    \begin{center}
    \begin{tikzpicture}
    \begin{scope}[scale=1.2]
        \draw[line width=.5pt, double distance=2pt,->-] (0,0) node[dot]{} node[below]{${\scriptstyle 0}$} -- (1,0) node[dot]{}node[below]{${\scriptstyle 1}$};
        \draw (1,0)--(2,0)node[dot]{}node[below]{${\scriptstyle 2}$};
        \node at (2.5,0) {$\cdots\cdots$};
        \draw (3,0)node[dot]{}node[below]{${\scriptstyle l-3}$}--(4,0)node[dot]{}node[below]{${\scriptstyle l-2}$}--(4.7,0.7)node[dot]{}node[right]{${\scriptstyle l-1}$};
        \draw (4,0)--(4.7,-0.7)node[dot]{}node[right]{${\scriptstyle l}$};

        \draw[<->,dashed] (5,0.5) arc (30:-30:1);
        \node[right] at (5.1,0) {${\scriptstyle \sigma_{l-1,l}}$};
    \end{scope}
    \end{tikzpicture} 
    \end{center}
    \begin{itemize}
        \item $(a_0, \dots, a_{\rkCar}) = (1, 2, \dots, 2, 1, 1)$; $(a_0^{\vee}, \dots, a_{\rkCar}^{\vee}) = (2, 2, \dots, 2, 1, 1)$; $\dCox = {2 \rkCar}$.
        \item $\lieg = \lsl(2l)$, ${\rm dim} \lieg = (2{\rkCar}+1)({2 \rkCar}-1)$.
        \item $\alpha_i = \epsilon_i - \epsilon_{i+1}$ for $1 \le i \le {\rkCar}-1$; $\alpha_{\rkCar} = \epsilon_{{\rkCar}-1} + \epsilon_{{\rkCar}}$; $\theta = 2\epsilon_1$; $\rho = \sum\limits_{i=1}^{\rkCar} ({\rkCar}-i) \epsilon_i$.
        \item $(\epsilon_i | \epsilon_j) = \delta_{i, j}$; $D={\rm diag}(2, 1,\dots, 1, 1)$.
        \item $\varpi_i = \epsilon_1 + \cdots + \epsilon_i$ for $1 \le i \le {\rkCar}-2$; 
        
        $\varpi_{{\rkCar}-1} = \frac{1}{2}(\epsilon_1+ \cdots+ \epsilon_{{\rkCar}-1} - \epsilon_{\rkCar})$; $\varpi_{\rkCar} = \frac{1}{2}(\epsilon_1+ \cdots+ \epsilon_{{\rkCar}-1} + \epsilon_{\rkCar})$.
        \item Diagram automorphism: $\Sigma = \langle \sigma_{{\rkCar}-1, {\rkCar}} \rangle$,
        
        $\sigma_{{\rkCar}-1, {\rkCar}}$ switches the vertices ${\rkCar}-1$ and ${\rkCar}$.
        \item The group $\Sigma$ acts on $\tilde{\lieh}{}^*$ as follows. 
        $$\sigma_{{\rkCar}-1, {\rkCar}}: s\delta + {\lLevel}\Lambda_0 + \sum_{i=1}^{\rkCar} y_i \epsilon_i  \mapsto s\delta + {\lLevel}\Lambda_0 + \sum_{i=1}^{{\rkCar}-1} y_i \epsilon_i - y_{\rkCar} \epsilon_{\rkCar}.$$
        \item $\sigdomdel(\lLevel) = \left\{ \sum\limits_{i=0}^{\rkCar} c_i \Lambda_i \mid  \sum\limits_{i=0}^{\rkCar} c_i a_i^\vee = {\lLevel}, c_{\rkCar} \ge c_{{\rkCar}-1} \right\}$.
    \end{itemize}

\medskip 
\noindent
{\bf A.5\quad Type $A^{(2)}_{2\rkCar}$, $\lghat = \lslhat{}^{(2)}({2 \rkCar}+1)$}
    \begin{center}
        \begin{tikzpicture}
        \begin{scope}[scale=1.2]
            \draw[line width=.5pt, double distance=2pt,->-] (0,0) node[dot]{} node[below]{${\scriptstyle 0}$} -- (1,0) node[dot]{}node[below]{${\scriptstyle 1}$};
            \draw (1,0)--(2,0)node[dot]{}node[below]{${\scriptstyle 2}$};
            \node at (2.5,0) {$\cdots\cdots$};
            \draw (3,0)node[dot]{}node[below]{${\scriptstyle l-2}$}--(4,0);
            \draw[line width=.5pt, double distance=2pt,->-] (4,0) node[dot]{} node[below]{${\scriptstyle l-1}$} -- (5,0) node[dot]{}node[below]{${\scriptstyle l}$};
        \end{scope}
    \end{tikzpicture} 
    \end{center}
    \begin{itemize}
        \item $(a_0, \dots, a_{\rkCar}) = (1, 2, \dots, 2, 2)$; $(a_0^{\vee}, \dots, a_{\rkCar}^{\vee}) = (2, 2, \dots, 2, 1)$; $\dCox = {2 \rkCar +1}$.
        \item $\lieg = \lsl(2l+1)$, ${\rm dim} \lieg = 2{\rkCar}({2 \rkCar}+2)$.
        \item $\alpha_i = \epsilon_i - \epsilon_{i+1}$ for $1 \le i \le {\rkCar}-1$; $\alpha_{\rkCar} = \epsilon_{{\rkCar}}$; $\theta = 2\epsilon_1$; $\rho = \sum\limits_{i=1}^{\rkCar} ({\rkCar} + \frac{1}{2} - i) \epsilon_i$.
        \item $(\epsilon_i | \epsilon_j) = \delta_{i, j}$; $D={\rm diag}(2, 1,\dots, 1, \frac{1}{2})$.
        \item $\varpi_i = \epsilon_1 + \cdots + \epsilon_i$ for $1 \le i \le {\rkCar}-1$; $\varpi_{{\rkCar}} = \frac{1}{2}(\epsilon_1+ \cdots \epsilon_{{\rkCar}-1} + \epsilon_{\rkCar})$.
        \item Diagram automorphism: $\Sigma = \{{\rm id}\}$.
        \item $\sigdomdel(\lLevel) = \left\{ \sum\limits_{i=0}^{\rkCar} c_i \Lambda_i \mid \sum\limits_{i=0}^{\rkCar} c_i a_i^\vee = {\lLevel}\right\}$.
     \end{itemize}

\medskip 
\noindent
{\bf A.6\quad Type $A^{(1)}_{\rkCar-1}$, $\lghat = \lslhat({\rkCar})$}
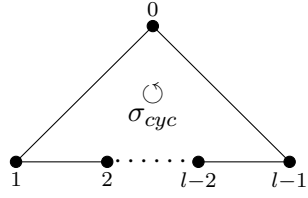
    \begin{center}
        \begin{tikzpicture}
        \begin{scope}[scale=1.2]
            \draw (1,0)node[dot]{}node[below]{${\scriptstyle 1}$}--(2,0)node[dot]{}node[below]{${\scriptstyle 2}$};
            \node at (2.5,0) {$\cdots\cdots$};
            \draw (3,0)node[dot]{}node[below]{${\scriptstyle l-2}$}--(4,0)node[dot]{}node[below]{${\scriptstyle l-1}$};
            \draw (1,0) -- (2.5,1.5)node[dot]{}node[above]{${\scriptstyle 0}$} -- (4,0);

            \node at (2.5,0.75) {$\circlearrowleft$};
            \node[below] at (2.5,0.7) {$\sigma_{cyc}$};
        \end{scope}
    \end{tikzpicture} 
    \end{center}
    \begin{itemize}
        \item $(a_0, \dots, a_{\rkCar-1}) = (1, \dots, 1)$; $(a_0^{\vee}, \dots, a_{\rkCar-1}^{\vee}) = (1, \dots, 1)$; $\dCox = {\rkCar}$.
        \item $\lieg = \lsl(l)$, ${\rm dim} \lieg = {\rkCar}^2-1$. Root vectors: 
        $$e_i = E_{i, i+1}(0) \mbox{ for } i = 1, \dots, \rkCar-1, e_0 = E_{l, 1}(1);$$ 
        $$f_i = E_{i+1, i}(0)\mbox{ for } i = 1, \dots, \rkCar-1, f_0 = E_{1, l}(-1).$$
        \item $\alpha_i = \epsilon_i - \epsilon_{i+1}$ for $1 \le i \le {\rkCar-1}$; $\theta = \epsilon_1 - \epsilon_{\rkCar}$; $\rho = \sum\limits_{i=1}^{\rkCar} ({\rkCar}+1-i) \epsilon_i$.
        \item $(\epsilon_i | \epsilon_j) = \delta_{i, j}$; $D={\rm diag}(1, \dots, 1)$.
        \item The fundamental weights of $\lgl(l)$: $\varpi_i = \epsilon_1 + \cdots + \epsilon_i$, for $1 \le i \le {\rkCar}$.
        \item For a weight $Y = \sum\limits_{i=1}^{\rkCar} y_i \epsilon_i$ of $\lgl(l)$, denote by $[Y] \coloneq  Y - \big(\sum\limits_{i=1}^{\rkCar} y_i\big)\mathbf{1}$ the corresponding weight of $\lsl(l)$, where $\mathbf{1}=(1,1,\cdots,1)$. 
        \item Diagram automorphism: $\Sigma = \langle \sigma^{\#}_{cyc} \rangle$, for $\lglhat({\rkCar})$. 
        
        Let $\sigma_{cyc}$ be the diagram automorphism of $\lslhat({\rkCar})$ permuting the vertices $0, 1, \dots, \rkCar-1$ cyclically. It lifts to an automorphism of infinite order $\sigma^{\#}_{cyc}$ on $\lglhat({\rkCar})$. (ref. \cite[\S 1.3.4]{Hasegawa1989}). 
        \item The automorphism $\sigma_{cyc}$ acts on $\tilde{\lieh}{}^*$ as follows.
        $$\sigma_{cyc}: s\delta + {\lLevel}\Lambda_0 + \sum_{i=1}^{\rkCar} y_i \epsilon_i  \mapsto (s - \lLevel(\rkCar - 1)/2\rkCar - y_{\rkCar})\delta + {\lLevel}\Lambda_0 + \sum_{i=1}^{{\rkCar}} (y_{i-1} + (\delta_{i1} - 1/\rkCar)\lLevel) \epsilon_i.$$
        \item The automorphism $\sigma^{\#}_{cyc}$ acts on $\lghat$ as follows.
        $$\sigma^{\#}_{cyc}: E_{ij}(n) \mapsto E_{i+1, j+1\pmod{\rkCar}}(n - \delta_{i{\rkCar}} + \delta_{j{\rkCar}}) - \delta_{n0}\delta_{i{\rkCar}}\delta_{j{\rkCar}} \cdot c.$$
        \item $\sigdomdel({\lLevel}) = \left\{ \sum\limits_{i=0}^{\rkCar} c_i \Lambda_i \mid  \sum\limits_{i=0}^{\rkCar} c_i a_i^\vee = {\lLevel}, c_0 > 0 \right\}.$
    \end{itemize}

\medskip
\noindent
{\bf A.7\quad Type $C^{(1)}_{\rkCar}$, $\lghat = \lsphat({2 \rkCar}$)}
    \begin{center}
        \begin{tikzpicture}
        \begin{scope}[scale=1.2]
            \draw[line width=.5pt, double distance=2pt,->-] (0,0) node[dot]{} node[below]{${\scriptstyle 0}$} -- (1,0) node[dot]{}node[below]{${\scriptstyle 1}$};
            \draw (1,0)--(2,0)node[dot]{}node[below]{${\scriptstyle 2}$};
            \node at (2.5,0) {$\cdots\cdots$};
            \draw (3,0)node[dot]{}node[below]{${\scriptstyle l-2}$}--(4,0);
            \draw[line width=.5pt, double distance=2pt,->-] (5,0) node[dot]{} node[below]{${\scriptstyle l}$} -- (4,0) node[dot]{}node[below]{${\scriptstyle l-1}$};
        \end{scope}
    \end{tikzpicture} 
    \end{center}
    \begin{itemize}
        \item $(a_0, \dots, a_{\rkCar}) = (1, 2, \dots, 2, 1)$; $(a_0^{\vee}, \dots, a_{\rkCar}^{\vee}) = (2, 2, \dots, 2, 2)$; $\dCox = 2{\rkCar}+2$.
        \item $\lghat = \lsphat({2 \rkCar}):= \{x\in \lgl(2\rkCar)\mid x^T J_{sp} + J_{sp} x = 0\}$, where 
        $J_{sp} = \left(\begin{matrix}
            J_{\rkCar} &   \\ & -J_{\rkCar}
        \end{matrix}\right)$; 
        ${\rm dim} \lieg = {\rkCar}({2 \rkCar}+1)$. 
        Root vectors: Let $C_{ij} = E_{ij} - {\rm sgn}(ij) E_{-j, -i}$. 
        $$e_i = C_{i, i+1}(0)\mbox{ for } i = 1, \dots, \rkCar-1,\; e_0 = C_{-1, 1}(1),\; e_{\rkCar}=C_{\rkCar-1,-\rkCar}(0);$$ 
        $$f_i = C_{i+1, i}(0)\mbox{ for } i = 1, \dots, \rkCar-1,\; f_0 = C_{1, -1}(-1),\; f_{\rkCar}=C_{-\rkCar,\rkCar-1}(0).$$
        \item $\alpha_i = \epsilon_i - \epsilon_{i+1}$ for $1 \le i \le {\rkCar}-1$; $\alpha_{\rkCar} = 2\epsilon_{{\rkCar}}$; $\theta = 2\epsilon_1$; $\rho = \sum\limits_{i=1}^{\rkCar} ({\rkCar} + 1 - i) \epsilon_i$.
        \item $(\epsilon_i | \epsilon_j) = \delta_{i, j}$; $D={\rm diag}(2, 1,\dots, 1, 2)$.
        \item $\varpi_i = \epsilon_1 + \cdots + \epsilon_i$ for $1 \le i \le {\rkCar}$.
        \item Diagram automorphism: $\Sigma = \{{\rm id}\}$.
        \item $\sigdomdel({\lLevel}) = \left\{ \sum\limits_{i=0}^{\rkCar} c_i \Lambda_i \mid \sum\limits_{i=0}^{\rkCar} c_i a_i^\vee = {\lLevel}\right\}$.
     \end{itemize}

\section{The set of positive roots for non-branching types}\label{pos-roots}
Referring to \cite[Prop 6.3]{Kac} and \cite[Appendix A]{Carter},
we list the set of positive roots $\dot{\Delta}_{\ex}^+$ for all non-branching type classical affine Dynkin diagram.
The Dynkin diagram ${}^t\!A^{(2)}_{2\rkCar}$ is obtained from $A^{(2)}_{2\rkCar}$ by reindex the vertices reversedly, see Remark \ref{reverse-A}.

Here we denote by $\Delta^+_{im}$ (resp. $\Delta^+_{re}$) the set of imaginary (resp. real) positive roots, with the subscripts $s,l,m$ corresponding to the set of the shortest, the longest, or the medium real roots, respectively.

In all the following sets, the indices $i,j$ always run over $1\le i < j \le l$.
The set of positive roots for each non-branching type is given in Table \ref{tbl-non-branching}.

\begin{table}[h]
\centering
\renewcommand{\arraystretch}{1.8} 
\begin{tabular}{l l p{10cm}}
\toprule
\textbf{Type} & \textbf{Positive Roots ($\Delta^+$)} \\ \midrule
$A^{(1)}_{\rkCar-1}$ & 
\begin{tabular}{@{}l@{}} 
$\Delta^+_{im} = \mathbb{N}^+ \delta$ \\
$\Delta^+_{re} = (\{\epsilon_i - \epsilon_j \} + \mathbb{N}\delta) \cup (\{-\epsilon_i + \epsilon_j\} + \mathbb{N}^+\delta)$
\end{tabular} \\ \hline

$C^{(1)}_{\rkCar}$ & 
\begin{tabular}{@{}l@{}} 
$\Delta^+_{im} = \mathbb{N}^+ \delta$ \\
$\Delta^+_{re,s} = (\{\epsilon_i \pm \epsilon_j \} + \mathbb{N}\delta) \cup (\{ - \epsilon_i \pm \epsilon_j\} + \mathbb{N}^+\delta)$ \\
$\Delta^+_{re,l} = (\{2\epsilon_i\} + \mathbb{N}\delta) \cup (\{ - 2\epsilon_i \} + \mathbb{N}^+\delta)$
\end{tabular} \\ \hline

$D^{(2)}_{\rkCar+1}$ & 
\begin{tabular}{@{}l@{}} 
$\Delta^+_{im} = \frac{1}{2}\mathbb{N}^+ \delta$ \\
$\Delta^+_{re,l} = (\{\epsilon_i \pm \epsilon_j \} + \mathbb{N}\delta) \cup (\{ - \epsilon_i \pm \epsilon_j\} + \mathbb{N}^+\delta)$ \\
$\Delta^+_{re,s} = (\{\epsilon_i\} + \frac{1}{2}\mathbb{N}\delta) \cup (\{ - \epsilon_i \} + \frac{1}{2}\mathbb{N}^+\delta)$
\end{tabular} \\ \hline

$A^{(2)}_{2\rkCar}$ & 
\begin{tabular}{@{}l@{}} 
$\Delta^+_{im} = \mathbb{N}^+ \delta$ \\
$\Delta^+_{re,m} = (\{\epsilon_i \pm \epsilon_j \} + \mathbb{N}\delta) \cup (\{ - \epsilon_i \pm \epsilon_j\} + \mathbb{N}^+\delta)$ \\
$\Delta^+_{re,s} = (\{\epsilon_i\} + \mathbb{N}\delta) \cup (\{ - \epsilon_i \} + \mathbb{N}^+\delta)$ \\
$\Delta^+_{re,l} = (\{ \pm 2\epsilon_i\} + (1 + 2\mathbb{N})\delta)$
\end{tabular} \\ \hline

${}^t\!A^{(2)}_{2\rkCar}$ & 
\begin{tabular}{@{}l@{}} 
$\Delta^+_{im} = \mathbb{N}^+ \delta$ \\
$\Delta^+_{re,m} = (\{\epsilon_i \pm \epsilon_j \} + \mathbb{N}\delta) \cup (\{ - \epsilon_i \pm \epsilon_j\} + \mathbb{N}^+\delta)$ \\
$\Delta^+_{re,l} = (\{2\epsilon_i\} + \mathbb{N}\delta) \cup (\{ - 2\epsilon_i \} + \mathbb{N}^+\delta)$ \\
$\Delta^+_{re,s} = (\{ \pm \epsilon_i\} + (\frac{1}{2} + \mathbb{N})\delta)$
\end{tabular} \\ 
\bottomrule
\end{tabular}
\caption{Root System Data for Non-Branching Types}\label{tbl-non-branching}
\end{table}

\section{Actions of pairs $(\lghat,\lgdhat)$}\label{appendix-aff-aff}

We list for each type the algebras $\lghat,\lgdhat$ and $\lghat_{\bigalg}$, together with the detailed actions of $\lghat,\lgdhat$ on $\F$.

\medskip 
\noindent
{\bf C.1\quad Type $(O^{(\r)}, O^{(\r)})$}
Let $(\lghat, \lgdhat) = (\lsohat(L_{\bar{0}}, L_{\bar{1}}), \lsohat(R_{\bar{0}}, R_{\bar{1}}))$ and $\lieg_{\bigalg}=\lso(L_{\bar{0}} R_{\bar{0}} + L_{\bar{1}} R_{\bar{1}}, L_{\bar{1}} R_{\bar{0}} + L_{\bar{0}} R_{\bar{1}})$.
As a level-$1$ $\lghat_{\bigalg}$-module, $\F$ admits commuting actions by $\lghat_{\lleft}$ and $\lgdhat$ via restriction:
\[\begin{array}{cc}
\begin{array}{rcl}
     \pi_{\lleft}:  \lghat_{\lleft} & \to & \End (\F) \\
     D_{ij}(n) & \mapsto & \sum\limits_{(s, \gamma)} \nord{\psi^{i,s}(n-\gamma) \psi^{-j, -s}(\gamma)} \\
     c_{\lleft} &\mapsto&  R;\\
\end{array}
& 
\begin{array}{rcl}
  \dot{\pi}:  \lgdhat & \to & \End (\F) \\
     D_{pq}(n) & \mapsto & \sum\limits_{(k, \gamma)} \nord{\psi^{k,p}(n-\gamma) \psi^{-k, -q}(\gamma)} \\
    \dot{c} & \mapsto & L.\\
\end{array}
\end{array}\]

\medskip 
\noindent
{\bf C.2\quad Type $(A^{(2)}, A^{(2)})$}
Let $(\lghat_{\lleft}, \lgdhat) = (\lglhat{}^{(2)}(L), \lslhat{}^{(2)}(R))$, $\lghat_{\bigalg} = \lglhat{}^{(2)}(LR)$ and fix $h=0$.
As a level-$1$ $\lghat_{\bigalg}$-module, $\F$ admits commuting actions by $\lghat_{\lleft}$ and $\lgdhat$ via restriction:

\begin{align*}
    \pi_{\lleft}: & \lghat_{\lleft} \to \End (\F) \\
    & A_{ij}(n) \mapsto \sum_{s}\sum_{\gamma} \nord{\psi^{i,s}(n-\gamma) \psi^{-j, -s}(\gamma)} \\
    & A_{-i, j}(n) \mapsto \sum_{s}\sum_{\gamma} (-1)^{\gamma} \nord{ \psi^{-i,s}(n-\gamma) \psi^{-j, -s}(\gamma) } \\
    & A_{i, -j}(n) \mapsto \sum_{s}\sum_{\gamma} (-1)^{\gamma} \nord{ \psi^{i,s}(n-\gamma) \psi^{j, -s}(\gamma) } \\
    & A_{i, 0}(n) \mapsto \sum_{s \ge 0}\sum_{\gamma} \nord{ \psi^{-i,s}(n-\gamma) \psi^{0, -s}(\gamma) } +  \sum_{s < 0}\sum_{\gamma} (-1)^{\gamma} \nord{ \psi^{-i,s}(n-\gamma) \psi^{0, -s}(\gamma) }\\
    & A_{-i, 0}(n) \mapsto \sum_{s \ge 0}\sum_{\gamma} (-1)^{\gamma} \nord{ \psi^{-i,s}(n-\gamma) \psi^{0, -s}(\gamma) } +  \sum_{s < 0}\sum_{\gamma} \nord{ \psi^{-i,s}(n-\gamma) \psi^{0, -s}(\gamma) }\\
    & A_{0, 0}(n) \mapsto \sum_{s \ne 0}\sum_{\gamma} \nord{ \psi^{0,s}(n-\gamma) \psi^{0, -s}(\gamma) } +  \sum_{\gamma} (-1)^{\gamma} \nord{ \psi^{0,0}(n-\gamma) \psi^{0, 0}(\gamma) }\\
    & c_{\lleft} \mapsto R;\\
\end{align*} 
for $i, j>0$, and
\begin{align*}
    \dot{\pi}: & \lgdhat \to \End (\F) \\
    & A_{pq}(n) \mapsto
    \begin{cases}
        \sum\limits_{\gamma} \left(\sum\limits_{k \ge 0}\nord{ \psi^{k,p}(n-\gamma) \psi^{-k, -q}(\gamma) } + (-1)^n \sum\limits_{k < 0}\nord{ \psi^{k,p}(n-\gamma) \psi^{-k, -q}(\gamma) }\right),\\
        \hfill \text{ if } p > 0 , q \ge 0;\\
        \sum\limits_{\gamma} \left(\sum\limits_{k > 0}\nord{ \psi^{k,p}(n-\gamma) \psi^{-k, -q}(\gamma) } + (-1)^n \sum\limits_{k \le 0}\nord{ \psi^{k,p}(n-\gamma) \psi^{-k, -q}(\gamma) }\right),\\
        \hfill \text{ if } p \le 0 , q < 0;\\
        \sum\limits_{\gamma} \left(\sum\limits_{k > 0}\nord{ \psi^{k,p}(n-\gamma) \psi^{-k, -q}(\gamma) } + (-1)^n \sum\limits_{k < 0}\nord{ \psi^{k,p}(n-\gamma) \psi^{-k, -q}(\gamma) }\right)   \\
        \hspace{1cm} + \sum\limits_{\gamma}\nord{ (-1)^{\gamma} \psi^{0,p}(n-\gamma) \psi^{0, -q}(\gamma) },\hfill \text{ otherwise}.
    \end{cases}\\
    & \dot{c} \mapsto L.
\end{align*}

\medskip 
\noindent
{\bf C.3\quad Type $(C^{(1)}, C^{(1)})$}
Let $(\lghat_{\lleft}, \lgdhat) = (\lsphat(2l), \lsphat(2r))$ and $\lghat_{\bigalg} = \lsohat(4lr)$.
As a level-$1$ $\lghat_{\bigalg}$-module, $\F$ admits commuting actions by $\lghat_{\lleft}$ and $\lgdhat$ via restriction:
\begin{align*}
    \pi_{\lleft}: & \lghat_{\lleft} \to \End (\F) \\
    & C_{ij}(n) \mapsto \sum_{s=1}^r \sum_{\gamma} (\nord{\psi^{i,s}(n-\gamma) \psi^{-j, -s}(\gamma)} + \sgn(ij) \nord{\psi^{i,-s}(n-\gamma) \psi^{-j, s}(\gamma)}) \\
    & c_{\lleft} \mapsto r;\\
    \dot{\pi}: & \lgdhat \to \End (\F) \\
    & C_{pq}(n) \mapsto \sum_{k=1}^l \sum_{\gamma} (\nord{\psi^{k,p}(n-\gamma) \psi^{-k, -q}(\gamma)} + \sgn(pq) \nord{\psi^{-k,p}(n-\gamma) \psi^{k, -q}(\gamma)}) \\
    & \dot{c} \mapsto l.\\
\end{align*}

\medskip 
\noindent
{\bf C.4\quad Type $(A^{(1)}, A^{(1)})$}
Let $(\lghat,\lgdhat)=(\lglhat(l),\lslhat(r))$ and $\lghat_{\bigalg}= \lsohat(2lr)$.
As a level-$1$ $\lghat_{\bigalg}$-module, $\F$ admits commuting actions by $\lghat_{\lleft}$ and $\lgdhat$ via restriction:
\[\begin{array}{cc}
 \begin{array}{rcl}
    \pi_{\lleft}: \lghat_{\lleft} & \to &\End (\F) \\
     E_{ij}(n) & \mapsto & \sum_{s=1}^r \sum_{\gamma} \nord{\psi^{i,s}(n-\gamma) \psi^{-j, -s}(\gamma)} \\
     c_{\lleft} & \mapsto & r;\\
 \end{array}   
 \begin{array}{rcl}
    \dot{\pi}: \lgdhat &\to& \End (\F) \\
     E_{pq}(n) & \mapsto & \sum_{k=1}^l \sum_{\gamma} \nord{\psi^{k,p}(n-\gamma) \psi^{-k, -q}(\gamma)} \\
     \dot{c} & \mapsto & l.\\
 \end{array}
\end{array}
\]

\section*{Acknowledgements}
The first and the second authors are partially supported by Beijing Natural Science Foundation (No. 1252011). The work was also supported by the Open Project Program of Key Laboratory of Mathematics and Complex System(No. K202402), Beijing Normal University. The authors thank Dr. Huang Lin for valuable discussions.

\bigskip 
\noindent 
Wei Hu, School of Mathematical Sciences, Beijing Normal University, Beijing 100875, China

\medskip 
\noindent
Email: {\tt huwei@bnu.edu.cn}

\medskip 
\noindent
Feiyue Huang, School of Mathematical Sciences, Beijing Normal University, Beijing 100875, China

\medskip 
\noindent
Email: {\tt fyhuang@mail.bnu.edu.cn}

\medskip 
\noindent
Yanbo Li, School of Mathematics and Statistics, Northeastern University at Qinhuangdao, Qinhuangdao 066004, China

\medskip 
\noindent
Email: {\tt liyanbo707@163.com}

\medskip 
\noindent
Xiangyu Qi, School of Mathematics and Statistics, Beijing Institute of Technology, Beijing, 100081, China

\medskip 
\noindent
Email: {\tt qixiangyumath@163.com}
\end{document}